\documentclass{amsart}
\usepackage{metalogo}
\usepackage{indentfirst}

\usepackage{verbatim}

\usepackage{amsmath}
\usepackage{amsmath,amscd}
\usepackage{amsthm}
\usepackage{amsfonts}
\usepackage{amssymb}
\usepackage{mathrsfs}
\usepackage{fancyhdr}
\usepackage[all]{xy}
\usepackage{amsaddr}

\usepackage{amssymb, paralist, xspace, graphicx, url, amscd, euscript, mathrsfs, stmaryrd}
\usepackage[all]{xy}

\numberwithin{equation}{section}
2
\numberwithin{subsection}{section}
\allowdisplaybreaks[1]

\newenvironment{enumerate1} {\begin{enumerate}[\upshape (1)]} {\end{enumerate}}

\theoremstyle{definition}
\newtheorem{definition}[subsubsection]{Definition}
\newtheorem{theorem}[subsubsection]{Theorem}
\newtheorem{lemma}[subsubsection]{Lemma}

\newtheorem{corollary}[subsubsection]{Corollary}

\newtheorem{proposition}[subsubsection]{Proposition}

\newtheorem*{Assumption}{Assumption}

\theoremstyle{remark}
\newtheorem{rem}[subsubsection]{Remark}
\newtheorem{remark}{Remark}

\newcommand{\hig}{\mathcal{H}iggs_{G}}
\newcommand{\higc}{\mathcal{H}iggs_{\check{G}}}
\newcommand{\jtg}{\mathcal{T}ors(J_{G})}
\newcommand{\jtgcam}{\mathcal{T}ors(J_{G,\widetilde{X}})}
\newcommand{\jtgc}{\mathcal{T}ors(J_{\check{G}})}
\newcommand{\jtgccam}{\mathcal{T}ors(J_{\check{G},\widetilde{X}})}
\newcommand{\mjtgcam}{\textrm{Tors}(J_{G,\widetilde{X}})}
\newcommand{\mjtgccam}{\textrm{Tors}(J_{\check{G},\widetilde{X}})}
\newcommand{\jtgcamn}{\mathcal{T}ors(J_{G,\widetilde{X}^{n}})}
\newcommand{\jtgccamn}{\mathcal{T}ors(J_{\check{G},\widetilde{X}^{n}})}
\newcommand{\mjtgcamn}{\textrm{Tors}(J_{G,\widetilde{X}^{n}})}
\newcommand{\mjtgccamn}{\textrm{Tors}(J_{\check{G},\widetilde{X}^{n}})}
\newcommand{\cam}{\widetilde{X}}
\newcommand{\bunt}{\textrm{Bun}_{T}}
\newcommand{\buntc}{\textrm{Bun}_{\check{T}}}

\newcommand{\bung}{\textrm{Bun}_{G}}
\newcommand{\loc}{\textrm{Locsys}_{\check{G}}}
\newcommand{\quasicoh}{\textrm{QCoh}}
\newcommand{\mhig}{\textrm{Higgs}_{G}}
\newcommand{\mhigc}{\textrm{Higgs}_{\check{G}}}
\newcommand{\mhigccam}{\textrm{Higgs}_{\check{G}}(\widetilde{X})}
\newcommand{\mjtg}{\textrm{Tors}(J_{G})}
\newcommand{\mjtgc}{\textrm{Tors}(J_{\check{G}})}
\newcommand{\mjtgn}{\textrm{Tors}(J^{n}_{G})}
\newcommand{\de}{\textrm{det}}
\newcommand{\pic}{\textrm{Pic}}
\newcommand{\jtgsccam}{\mathcal{T}ors(J_{G_{sc},\widetilde{X}})}
\newcommand{\jtgadcam}{\mathcal{T}ors(J_{G_{ad},\widetilde{X}})}
\newcommand{\higsccam}{\mathcal{H}iggs_{G_{sc}}(\widetilde{X})}
\newcommand{\jtgsc}{\mathcal{T}ors(J_{G_{sc}})}
\newcommand{\jtgad}{\mathcal{T}ors(J_{G_{ad}})}
\newcommand{\dimension}{\textrm{dim}}

\newcommand{\supp}{\textrm{Supp}}
\newcommand{\affgras}{\textrm{Gr}}
\newcommand{\spec}{\textrm{Spec}}
\newcommand{\disc}{\textrm{Disc}}
\newcommand{\higsc}{\mathcal{H}iggs_{G_{sc}}}
\newcommand{\stackypic}{\mathcal{P}ic}

\begin{document}
\title{The Poincar\'e line bundle and autoduality of Hitchin fibers}

\maketitle

\vspace*{6pt}\tableofcontents

\newpage

\section{Introduction}
\subsection{The classical limit of the geometric Langlands program}
Let $k$ be an algebraically closed field of characteristic zero. Let $G$ be a reductive group over $k$ and let $\check{G}$ be its Langlands dual group. Let $X$ be a smooth projective curve over $k$. Let $\bung$ be the stack of $G$ bundles on $X$ and $\loc$ be the stack of $\check{G}$ local systems on $X$. The geometric Langlands program proposes an equivalence of categories between the category of D-modules on $\bung$ and certain appropriately defined category of quasi-coherent sheaves on $\loc$. For a precise formulation of the conjecture, see \cite{singsupp}. It has also been proposed that the geometric Langlands conjecture has a classical limit. To explain this, we fix a line bundle $L$ on $X$. Let $\hig$ be the stack of $G$-Higgs bundles on $X$ with value in $L$ and $\higc$ be the stack of $\check{G}$-Higgs bundles with value in $L$. Then conjectually, the classical limit of the geometric Langlands correspondence should provide an equivalence of categories:
\begin{equation}\label{conjectural equivalence}
\quasicoh({\hig})\simeq\quasicoh({\higc}) .
\end{equation}
Both sides in ~\ref{conjectural equivalence} are categories over the Hitchin base $H$ and this equivalence should be an equivalence of categories over $H$. In particular, if one fix a cameral cover $\cam$, one should get an equivalence:
\begin{equation}\label{conjecture when fix cameral cover}
\quasicoh(\hig(\cam))\simeq\quasicoh(\higc(\cam)) .
\end{equation}

When $G=GL_{n}$ and the spectral cover is integral, \ref{conjecture when fix cameral cover} follows from the autoduality result for the compactified Jacobians of planar curves in \cite{Autoduality}. For general reductive groups, this has been proved over the locus of $H$ where the cameral cover is smooth, see \cite{Langlands Hitchin} and \cite{Geometric Langlands}. The conjecture is still open for cameral covers that are not smooth. 

In this paper we establish a partial result for Hitchin fibers corresponding to singular cameral covers. More precisely, let $\cam$ be a cameral over of $X$ which we assume to be integral. Let $\hig(\cam)$ be the stack of $G$-Higgs bundles with cameral cover $\cam$, let $\higc(\cam)$ be the stack of $\check{G}$-Higgs bundles with cameral cover $\cam$. Let $\hig(\cam)^{reg}$ be the open substack of $\hig(\cam)$ corresponding to Higgs bundles with everywhere regular Higgs field. We will construct a line bundle $\mathcal{P}$ on $\hig(\cam)^{reg}\times \higc(\cam)$, which we will call the Poincar\'e line bundle. Moreover, we will prove that $\mathcal{P}$ induces a fully faithful functor:
$$\quasicoh(\hig(\cam)^{reg})\hookrightarrow\quasicoh(\higc(\cam)) .$$
For $G=GL_{n}$, this has already been established in \cite{Poincare line bundle}. 

\subsection{Summary of the article}
In this subsection we will give a brief summary of the main ideas of the construction. Let us fix an integral cameral cover $\cam$ and let $\cam^{un}$ be the open locus where $\cam$ is an unramified $W$ cover over $X$. Denote the corresponding regular centralizer groups scheme of $G$ over $X$ by $J_{G,\cam}$. Denote the stack of $J_{G,\cam}$ torsors on $X$ by $\jtgcam$. Using a Kostant section one can identify $\hig(\cam)^{reg}$ with $\jtgcam$.

Our first task is to construct the Poincar\'e line bundle on $\jtgcam\times\higc(\cam)$. This occupies Subsection $4.1$. Our first observation is that it is not hard to write down the formula for the pullback of the Poincar\'e line bundle on 
$$(\cam^{un})^{(n)}\times X_{*}(T)\times\higc(\cam)$$
under the Abel-Jacobi map, here $(\cam^{un})^{(n)}$ denotes the $n$th symmetric power. For the details, see Subsection $4.1$. One can then use the results of \cite{Geometric Langlands} and \cite{Langlands Hitchin} to show it descends. This is done in Corollary~\ref{finish the descend for adjoint}.

In Subsection $4.2$ we will compute the cohomology of the Poincar\'e line bundle and use it to prove the fully faithfulness of the Fourier-Mukai transform. Our strategy is similar to those in \cite{Poincare line bundle}. We first establish a codimension estimate about the support of the pushforward of the Poincar\'e line bundle, see Lemma~\ref{codim estimate}. To prove the main result one still needs to analyze the structure of $\hig(\cam)$ when $\cam$ has nodal singularities. This is done in Appendix $B$. The first theorem is Theorem~\ref{main theorem on coho}, which takes care of the case when $G$ is semisimple. The case for general reductive groups will be deduced from this in Theorem~\ref{2nd main theorem of cohomology of poincare line bundle}. The final result is the following:
\begin{theorem}
Let $G$ be reductive. Let $z$ be the dimension of the center of $G$ and $d$ be the dimension of $\mathcal{H}iggs_{(G/Z(G))}(\cam)$. Let $\mathcal{P}_{G}$ be the Poincar\'e line bundle on $\jtgcam\times\higc(\cam)$ and $p_{1}$ be the projection to $\jtgcam$. Then we have:
$$Rp_{1*}\mathcal{P}_{G}\simeq e_{*}(k)[-d-zg]$$
where $g$ is the genus of $X$. 
\end{theorem}
When $G=\textrm{GL}_{n}$ this recovers the result in \cite{Poincare line bundle}. For the Fourier-Mukai transform, we have the following (see Corollary~\ref{fully faithful fourier-mukai}):
\begin{theorem}
Consider the diagram:
$$\xymatrix{
\jtgcam\times\higc(\cam) \ar[r]^-{p_{2}} \ar[d]_{p_{1}} & \higc(\cam) \\
\jtgcam .
}
$$
Then the Fourier-Mukai functor induced by $\mathcal{P}_{G}$:
\begin{gather}
\quasicoh(\jtgcam)\mapsto\quasicoh(\higc(\cam)) \notag \\
\mathcal{G}\mapsto p_{2*}(p^{*}_{1}(\mathcal{G})\otimes\mathcal{P}_{G}) 
\end{gather}
is fully faithful. Here all functors in the formula $1.3$ are derived.
\end{theorem}

\subsection{Organization of the paper}
Let us now describe the organization of the paper. 

In Section $2$ we will give an overview of some properties of the stack of Higgs bundles that will be used in this paper. We will mostly follow \cite{Geometric Langlands}, \cite{Gerbe of Higgs}, \cite{Langlands Hitchin}, \cite{Higgs Faltings} and \cite{Fundamental lemma}. 

In Section $3$ we will construct the Abel-Jacobi map as well as the norm map for $\jtgcam$. These will be the main tools in the construction of the Poincar\'e line bundle. 

In Section $4$ we will construct the Poincar\'e line bundle and compute its cohomology. The construction of the Poincar\'e line bundle is in Subsection $4.1$, the main result is Corollary~\ref{finish the descend for adjoint}. In $4.2$ we determine the cohomology of the Poincar\'e line bundle, the main theorems are Theorem~\ref{main theorem on coho} and Theorem~\ref{2nd main theorem of cohomology of poincare line bundle}. 

In Appendix $A$ we gather some well-known facts about the Poincar\'e line bundle on $\bunt(C)\times\buntc(C)$.

In Appendix $B$ we will determine the structure of $\hig(\cam)$ when $\cam$ has nodal singularities. The main result is Theorem~\ref{main theorem of app B}. This will be needed in Subsection $4.2$. 

In Appendix $C$ we identify an open subscheme of the Hitchin base $H$ and estimate the codimension of its complement. This is a technical result that will be needed in Subsection $4.2$.

\subsection{Notations}
In this article we always work over an algebraically closed field $k$ of characteristic zero. Let $G$ be a reductive group over $k$. We denote the Langlands dual group by $\check{G}$. We denote by $\mathfrak{g}$ (respectively by $\mathfrak{\check{g}}$) the Lie algebra of $G$ (respectively $\check{G}$). We will denote the abstract Cartan of $G$ by $T$ and its Lie algebra by $\mathfrak{t}$. The counterparts on the Langlands dual side will be denoted by $\check{T}$ and $\mathfrak{\check{t}}$. The character group of $T$ will be denoted by $X^{*}(T)$ and the cocharacter group will be denoted by $X_{*}(T)$. We will freely identify $X_{*}(T)$ with $X^{*}(\check{T})$. The set of roots in $X^{*}(T)$ will be denoted by $R$. The Weyl group will be denoted by $W$. We will also fix a non-degenerate $W$ invariant bilinear form on $\mathfrak{t}$ which induces an isomorphism $\mathfrak{t}\simeq\mathfrak{\check{t}}$. It also induces a $G$-equivariant isomorphism $\mathfrak{g}\simeq\mathfrak{g}^{*}$. We will denote the Chevalley base by $\mathfrak{c}$. Recall that one has $\mathfrak{c}=k[\mathfrak{t}]^{W}\simeq k[\mathfrak{g}]^{G}$. We will denote regular elements in $\mathfrak{g}$ by $\mathfrak{g}_{reg}$

Let $X$ be a curve over $k$ and $L$ be a line bundle on $X$. Let $L^{\times}$ be the $\mathbb{G}_{m}$ torsor corresponding to $L$. We will denote by $\mathfrak{c}^{L}=\mathfrak{c}\times^{\mathbb{G}_{m}}L^{\times}$ the $\mathbb{G}_{m}$ twist of $\mathfrak{c}$. Similarly for $\mathfrak{t}^{L}$. If $E_{G}$ is a $G$ bundle on $X$, we will denote $\mathfrak{g}_{E_{G}}=\mathfrak{g}\times^{G} E_{G}$ the adjoint bundle. Cameral covers over $X$ will be denoted by $\cam$. 

\subsection{Acknowledgments}
I would like to thank Professor David Kazhdan and Professor Dima Arinkin for continuous interest in this work as well as many helpful suggestions. This work is supported by grant AdG 669655.

\section{Review about Higgs bundles}
In this section we collect some basic geometric properties about the stack of Higgs bundles. The main references are \cite{Geometric Langlands},\cite{Gerbe of Higgs},\cite{Langlands Hitchin},\cite{Higgs Faltings} and \cite{Fundamental lemma}.
\subsection{The Hitchin fibration}
Let $X$ be a smooth projective curve over $k$ with genus $g\geq 2$. Let $L$ be a line bundle on $X$ with degree $l$. We assume $l> 2g$. 
\begin{definition}
A $G$-Higgs bundle on $X$ with value in $L$ is a pair $(E_{G},\phi)$ such that $E_{G}$ is a $G$-bundle on $X$ and $\phi$ is a global section of $\mathfrak{g}_{E_{G}}\otimes L$. 
\end{definition}

It is well-known that $\hig$ is an algebraic stack locally of finite type. The connected components of $\hig$ are parameterized by $\pi_{1}(G)$. 

Let us consider $\mathfrak{c}^{L}$ as an affine space over $X$. We denote by $H$ the space of sections of $\mathfrak{c}^{L}$ over $X$. We will call $H$ the Hitchin base. The $W$ invariant bilinear form on $\mathfrak{t}$ induces an isomorphism between the Hitchin base of $G$ and the Hitchin base of $\check{G}$. For any point $\sigma$ in $H$, the corresponding cameral cover over $X$ is defined to be the fiber product:
$$\xymatrix{
\widetilde{X} \ar[r] \ar[d]^{\pi} & \mathfrak{t}^{L} \ar[d]\\
X \ar[r]^{\sigma} & \mathfrak{c}^{L} .
}
$$
One also has the universal cameral cover over $X\times H$. By Proposition $4.7.1$ of \cite{Fundamental lemma}, there exists an open subscheme $H^{0}$ of $H$ such that for any point in $H^{0}$, the corresponding cameral cover is smooth. 

It is well known that the Chevalley map $\mathfrak{g}\rightarrow\mathfrak{c}$ induces the Hitchin fibration $\hig\xrightarrow{h}H$. For any point $\sigma$ in $H$ with the corresponding cameral cover $\cam$, we will call the fiber $h^{-1}(\sigma)$ the stack of Higgs bundles with cameral cover $\widetilde{X}$. We will denote it by $\hig(\cam)$. The following proposition summarizes geometric properties of the Hitchin fibration, see \cite{Opers}, \cite{Geometric Langlands} and \cite{Higgs Faltings}:
\begin{proposition}\label{geo of hitchin fibration}
\begin{enumerate1}
\item $h$ is a relative complete intersection morphism. In particular, $\hig(\cam)$ is Gorenstein.
\item $\pi_{0}(\hig)\simeq \pi_{0}(G)$. Moreover, for any point $\sigma\in H$, the natural morphism $\hig(\cam)\rightarrow\hig$ induces an isomorphism on $\pi_{0}$.
\item A choice of a square root of $L$ induces a section of the Hitchin fibration: $H\rightarrow\hig$. 
\item $\hig$ admits a good moduli space, which will be denoted by $\mhig$. $\mhig$ is projective over $H$. 
\end{enumerate1}
\end{proposition}

\subsection{Regular centralizer group scheme}
For any cameral cover $\widetilde{X}$ over $X$, one can associate to it a smooth group scheme $J_{G,\cam}$ over $X$, called the regular centralizer group scheme (See Section $4$ of \cite{Gerbe of Higgs} and Section $2$ of \cite{Fundamental lemma}). Denote the Picard stack of $J_{G,\cam}$ torsor by $\jtgcam$. One also has the universal group scheme $J_{G}$ over $X\times H$. The assignment $\cam\rightarrow\jtgcam$ defines a Picard stack over $H$, denote it by $\jtg$. The importance of $J_{G}$ comes from the fact that one can twist Higgs bundles by $J_{G}$ torsors, see Subsection $4.3$ of \cite{Fundamental lemma} and Subsection $2.6$ of \cite{Geometric Langlands}. Hence one has a natural action of $\jtg$ on $\hig$ over $H$:
$$\jtg\times_{H}\hig\rightarrow\hig .$$
 
Let $\hig^{reg}$ be the open substack of $\hig$ parameterizing $(E_{G},\phi)$ such that for any point $x\in X$, the restriction of $\phi$ to $x$ is an element in $\mathfrak{g}_{reg}$ once we trivialize $E_{G}$ and $L$ at $x$. Then it is well-known that $\hig^{reg}$ is a torsor over $\jtg$. To summarize, we have the following proposition (See \cite{Geometric Langlands}, \cite{Gerbe of Higgs}, \cite{Langlands Hitchin} and \cite{Fundamental lemma}):
\begin{proposition}\label{geo of hig}
\begin{enumerate1}
\item $\hig^{reg}$ is a torsor over $\jtg$. Hence once we choose a section of the Hitchin fibration, we get an isomorphism: $\hig^{reg}\simeq\jtg$.
\item If $\cam$ is reduced, then $\hig^{reg}(\cam)$ is dense in $\hig(\cam)$.
\item Over the open subscheme $H^{0}$ of $H$, one has: $\hig\times_{H}H^{0}\simeq\hig^{reg}\times_{H}H^{0}$.
\item Over $H^{\circ}$, there exists a Poincar\'e line bundle on $\jtg\times_{H^{\circ}}\jtgc$ such that it induces an isomorphism of Picard stacks over $H^{\circ}$: $\jtg\simeq (\jtgc)^{\check{}}$ where $(\jtgc)^{\check{}}$ denotes the dual Picard stack (See \cite{Geometric Langlands} and \cite{Langlands Hitchin}). 
\item Under our assumptions, we have $H^{0}(X,J_{G,\cam})\simeq Z(G)$. 
\end{enumerate1}
\end{proposition}
Another important property about the group scheme $J_{G,\cam}$ is that when we pullback it to $\cam$, it has a natural morphism to $T\times\cam$, that is, one has a natural morphism of group schemes over $\cam$:
$$J_{G,\cam}\times_{X}\cam\rightarrow T\times\cam .$$
Moreover, let $X^{un}$ be an open subscheme of $X$ where the cameral cover is unramified and let $\cam^{un}$ be its preimage in $\cam$. Then the above morphism is an isomorphism over $\cam^{un}$. The following proposition summarizes what we need (See Section $3$ of \cite{Geometric Langlands}):
\begin{proposition}
\begin{enumerate1}\label{prop of jtors}
\item There exist natural morphisms of group stacks: 
$$\jtgcam\rightarrow\bunt^{W}(\cam)\rightarrow\bunt(\cam) ,$$
where $\bunt^{W}(\cam)$ stands for the stack of $W$-equivariant $T$-torsors on $\cam$.
\item The category of $J_{G,\cam}$-torsors on $X^{un}$ is equivalent to the category of $W$ equivariant $T$-torsors on $\cam^{un}$. So once we fix a Kostant section, the category of Higgs bundles on $X^{un}$ is equivalent to the category of $W$ equivariant $T$-torsors on $\cam^{un}$.
\item Once we fix a Kostant section, we have a morphism $\hig^{reg}(\cam)\rightarrow\bunt^{W}(\cam)$ via the identification $\hig^{reg}(\cam)\simeq\jtgcam$ induced by the Kostant section.
\end{enumerate1}
\end{proposition}

\section{The Abel-Jacobi map and the norm map}
In this section we study the norm map and the Abel-Jacobi map for the stack of $J_{G,\cam}$ torsors. Since we will mostly work with a fixed cameral cover, we will just write $J_{G}$ instead of $J_{G,\cam}$ in this section to simplify the notation. 

Let $\widetilde{X}\xrightarrow{\pi} X$ be a cameral cover over $X$ which we will assume to be integral and let $T$ be the corresponding torus. Let us recall the construction of the regular centralizer group scheme. In fact, one can associate to $\widetilde{X}$ three closely related sheaf of groups over $X$. First, we set $J^{1}_{G}=\pi_{*}(T\times\widetilde{X})^{W}$. Let $D^{\alpha}$ be the ramification divisor in $\widetilde{X}$ corresponding to the root $\alpha$. Then for any open subset $U\subseteq X$, we set (See Section $4$ of \cite{Gerbe of Higgs}):
$$J_{G}(U)=\{t\in J^{1}_{G}(U)\mid \alpha(t)=1  \  \mbox{for each root} \
\alpha \} .$$
Let $J^{0}_{G}$ be the neutral component of $J_{G}$. Then one has natural inclusions: 
$$J^{0}_{G}\subseteq J_{G}\subseteq J^{1}_{G} $$
and the quotients are sheaves supported on the ramification locus. Moreover, when $G$ is adjoint, then $J^{0}_{G}=J_{G}$ (See Proposition $2.3.1$ of \cite{Fundamental lemma}). When $G$ is simply connected, $J_{G}=J^{1}_{G}$ (See Section $4$ of \cite{Gerbe of Higgs}). 

We shall make use of the following lemma, which is a direct consequence of Tsen's theorem (See Chapter 10, Section $7$ of \cite{Local fields}):
\begin{lemma}\label{Tsen}
Let $\mathcal{G}$ be a sheaf of abelian groups on $X$ such that $\mathcal{G}$ is generically a torus. Then $H^{2}(X,\mathcal{G})=0$. If $\eta$ is the generic point of $X$, then $H^{i}(\eta,\mathcal{G}_{\eta})=0$ for $i\geqslant 1$.
\end{lemma}

Let us now construct the norm map. First, we claim that there exists a natural morphism of group schemes $\pi_{*}(T\times\widetilde{X})\rightarrow J^{1}_{G}$. In fact, by construction, for any $S$ over $X$, let $\widetilde{S}$ be the fiber product:
$$\xymatrix{
\widetilde{S}\ar[r] \ar[d] & \widetilde{X} \ar[d]\\
S \ar[r] & X ,
}$$
we have that $\pi_{*}(T\times\widetilde{X})(S)=X_{*}(T)\otimes O_{\widetilde{S}}^{\times}$ where $O_{\widetilde{S}}^{\times}$ stands for the group of units in $O_{\widetilde{S}}$. One consider the morphism:
\begin{gather}
\pi_{*}(T\times\widetilde{X})(S)\rightarrow \pi_{*}(T\times\widetilde{X})(S) \notag \\
\phi\otimes f\rightarrow\sum_{w\in W}w(\phi)\otimes w(f) \notag
\end{gather}
It is easy to see that the image lies in the $W$-invariant part, hence one gets a group homomorphism: $\pi_{*}(T\times\widetilde{X})\rightarrow J^{1}_{G}$. Moreover, since $\pi_{*}(T\times\widetilde{X})$ has geometrically connected fibers as group schemes over $X$, it actually factors through $J^{0}_{G}\subseteq J_{G}$. By looking at the morphism at the lie algebra level, one concludes that $\pi_{*}(T\times\widetilde{X})\rightarrow J^{0}_{G}$ is smooth and surjective. To summarize, we have the following:
\begin{proposition}\label{norm map}
There exists a natural morphism of groups schemes over $X$: 
$$\pi_{*}(T\times\widetilde{X})\xrightarrow{Nm}J_{G} .$$
We shall call it the norm map. The image of the norm map is equal to $J^{0}_{G}$. 
\end{proposition}

As a byproduct, let us prove the following:
\begin{lemma}\label{zariski local trivial}
$J^{0}_{G}$ torsors are Zariski locally trivial on $X$ (Hence same is true for $J_{G}$ or $J^{1}_{G}$ torsors).
\end{lemma}
\begin{proof}
Let $x\in X$, $O_{x}$ be the local ring at $x$ and $k(\eta)$ be the field of fractions. One needs to show that any $J^{0}_{G}$ torsor over $\spec(O_{x})$ is trivial. Let $\spec(\widetilde{O}_{x})$ be the induced cameral cover over $\spec(O_{x})$ and $k(\widetilde{\eta})$ be its field of fractions. One has:
\[\xymatrixcolsep{4pc}\xymatrix{
\widetilde{\eta} \ar[r]^{\widetilde{j}} \ar[d] & \spec(\widetilde{O}_{x}) \ar[d]_{\pi} \\
\eta \ar[r]^{j} & \spec(O_{x}) .
}
\]
Lemma~\ref{Tsen} implies that $H^{1}(\eta,J^{0}_{G,\eta})=0$ and that we have isomorphisms $j_{*}(J^{0}_{G,\eta})\simeq Rj_{*}(J^{0}_{G,\eta})$ and $\widetilde{j}_{*}(T_{\widetilde{\eta}})\simeq R\widetilde{j}_{*}(T_{\widetilde{\eta}})$. These implies that we have exact sequences of sheaves on $\spec(O_{x})$:
$$\xymatrix{
0 \ar[r] & \pi_{*}(T\times\spec(\widetilde{O}_{x})) \ar[r] \ar[d]_{Nm} & \pi_{*}(\widetilde{j}_{*}(T_{\widetilde{\eta}})) \ar[r] \ar[d]_{Nm} & \widetilde{Q} \ar[r] \ar[d] & 0\\
0 \ar[r] & J^{0}_{G} \ar[r] & j_{*}(J^{0}_{G,\eta}) \ar[r] & Q \ar[r] & 0
} 
$$
where $\widetilde{Q}$ and $Q$ are supported on the closed point. Proposition~\ref{norm map} implies that $\pi_{*}(T_{\widetilde{\eta}})\xrightarrow{Nm}J^{0}_{G,\eta}$ is surjective. Since its kernel is a torus over $\eta$, Lemma~\ref{Tsen} implies that $\pi_{*}(\widetilde{j}_{*}(T_{\widetilde{\eta}}))\rightarrow j_{*}(J^{0}_{G,\eta})$ is surjective, which implies that $\widetilde{Q}\rightarrow Q$ is surjective. Using the fact that $H^{1}(\spec(\widetilde{O}_{x}),T)=0$ we conclude that $H^{0}(\pi_{*}(\widetilde{j}_{*}(T_{\widetilde{\eta}})))\rightarrow H^{0}(\widetilde{Q})\rightarrow H^{0}(Q)$ are surjective. This implies $H^{0}(j_{*}(J^{0}_{G,\eta}))\rightarrow H^{0}(Q)$ is surjective. Recall that we have $j_{*}(J^{0}_{G,\eta})\simeq Rj_{*}(J^{0}_{G,\eta})$, so we get:
$$H^{1}(j_{*}(J^{0}_{G,\eta}))\simeq H^{1}(Rj_{*}(J^{0}_{G,\eta}))\simeq H^{1}(\eta,J^{0}_{G})=0 .$$
One conclude from the long exact sequence that $H^{1}(\spec(O_{x}),J^{0}_{G})=0$. 

\end{proof}

Lemma~\ref{Tsen} implies the following:
\begin{proposition}
The norm map induces a surjection: $H^{1}(\widetilde{X}, T)\rightarrow H^{1}(X, J_{G})$. Hence $\bunt(\cam)$ is a smooth cover of $\jtg$.
\end{proposition}
\begin{proof}
Let $K$ be the kernel of the norm map and let $Q$ be the quotient $J_{G}/J^{0}_{G}$. We have:
\begin{gather}
0\rightarrow K\rightarrow \pi_{*}(T\times\widetilde{X})\rightarrow J^{0}_{G}\rightarrow 0 \notag\\
0\rightarrow J^{0}_{G}\rightarrow J_{G}\rightarrow Q\rightarrow 0 \notag.
\end{gather}
Both $\pi_{*}(T\times\widetilde{X})$ and $J^{0}_{G}$ are generically torus on $X$, hence $K$ satisfies the condition of Lemma~\ref{Tsen}. Moreover, since $\cam$ is a finite cover of $X$, one has $H^{1}(\cam, T)\simeq H^{1}(X,\pi_{*}(T\times\widetilde{X}))$. Hence we get a surjection: $H^{1}(\widetilde{X}, T)\twoheadrightarrow H^{1}(X, J^{0}_{G})$. Since the support of $Q$ is finite, we also have a surjection $H^{1}(X,J^{0}_{G})\twoheadrightarrow H^{1}(X,J_{G})$. So the first claim follows from this. A tangent space calculation combined with the first claim shows that $\bunt(\cam)$ is a smooth cover of $\jtg$.
\end{proof}

We also need the following lemma, see Lemma $4.10.2$ and Proposition $4.10.3$ of \cite{Fundamental lemma}:
\begin{lemma}\label{surjection on pio}
The morphism $\bunt(\cam)\rightarrow\jtgcam$ induces surjective morphisms:
$$\pi_{0}(\bunt(\cam))\simeq X_{*}(T)\twoheadrightarrow X_{*}(T)/W\twoheadrightarrow\pi_{0}(\jtgcam)$$
\end{lemma}

The Abel-Jacobi map is defined as follows. Let $\cam^{un}$ be the open subscheme of $\cam$ defined right before Proposition~\ref{prop of jtors}. One has the Abel-Jacobi map for $T$-torsors on $\cam$:
\begin{gather}
\cam^{un}\times X_{*}(T)\rightarrow \bunt(\cam) \notag\\
(\widetilde{x},\phi)\rightarrow \phi_{*}(O(\widetilde{x})):=O(\widetilde{x})\times^{\mathbb{G}_{m},\phi}T .\notag
\end{gather}
We compose it with the norm map: $\bunt(\cam)\xrightarrow{Nm}\jtgcam$ and call it the Abel-Jacobi map for $\jtgcam$. More generally, for symmetric powers of $\cam^{un}$, one also has the Abel-Jacobi map:
\begin{gather}
(\cam^{un})^{(n)}\times X_{*}(T)\rightarrow \bunt(\cam)\xrightarrow{Nm} \jtgcam \notag\\
(\widetilde{D},\phi)\rightarrow \phi_{*}(O(\widetilde{D})):=O(\widetilde{D})\times^{\mathbb{G}_{m},\phi}T\rightarrow Nm(\phi_{*}(O(\widetilde{D}))) \notag
\end{gather}
where we have used the fact that points in $(\cam^{un})^{(n)}$ corresponds to degree $n$ divisors on $\cam$ that are contained in $\cam^{un}$. In this way we get for each $\phi\in X_{*}(T)$ a morphism:
$$(\cam^{un})^{(n)}\phi\rightarrow\bunt(\cam)\xrightarrow{Nm}\jtgcam .$$
For any finite set of cocharacters $\phi_{1}, \cdots, \phi_{k}$ and integers $n_{1}, \cdots, n_{k}$, one can generalize the construction above further and get a morphism:
$$(\cam^{un})^{(n_{1})}\phi_{1}\times\cdots\times(\cam^{un})^{(n_{k})}\phi_{k}\rightarrow\bunt(\cam)\xrightarrow{Nm}\jtgcam .$$

\section{The Poincar\'e line bundle}
In this section we will construct the Poincar\'e line bundle in $4.1$ and determine the cohomology of the Poincar\'e line bundle in $4.2$. We will need some technical results about the structure of $\hig(\cam)$ when $\cam$ has nodal singularities, those will be given in Appendix $B$. 

\subsection{The construction of the Poincar\'e line bundle}
In this subsection we construct the Poincar\'e line bundle. The main result is Corollary~\ref{finish the descend for adjoint}. After that we establish some additional compatibilities. We will always work with Higgs bundles with integral cameral covers. Let us denote the open subscheme of $H$ parameterizing integral cameral covers by $H_{int}$. To simplify the notation, let us denote the restriction of the stacks $\hig$ and $\jtg$ to $H_{int}$ still by $\hig$ and $\jtg$. We will construct the Poincar\'e line bundle on $\jtg\times_{H_{int}}\higc$. We shall also fix a Kostant section for $\higc$ throughout this section. 

Now let us start to construct the Poincar\'e line bundle. We first construct it on a smooth cover of $\jtgcam\times\higc(\cam)$, then we will show it descend. Namely, let us start with the following setup. By our discussions about the Abel-Jacobi map in Section $3$, we have a natural morphism:
$$(\cam^{un})^{(n)}\times X_{*}(T)\rightarrow\jtgcam .$$
We shall now construct the Poincar\'e line bundle on $(\cam^{un})^{(n)}\times X_{*}(T)\times\higc(\cam)$. First, for any $\check{G}$-Higgs bundle $(E_{\check{G}},\varphi)$ on $X$, one gets a $\check{T}$-torsor on $\cam^{un}$ using part $(2)$ of Proposition~\ref{prop of jtors}. Let us denote the $\check{T}$-torsor on $\cam^{un}$ by $(E_{\check{G}},\varphi)'$. Then we declare that the fiber of the Poincar\'e line bundle at the point
$$(\widetilde{D},\phi,(E_{\check{G}},\varphi))\in(\cam^{un})^{(n)}\times X_{*}(T)\times\higc(\cam)$$ 
is given by:
$$\de(\phi((E_{\check{G}},\varphi)')_{\widetilde{D}})\otimes\de(O_{\widetilde{D}})^{-1} .$$
Here $\phi((E_{\check{G}},\varphi)')$ stands for the line bundle on $\cam^{un}$ induced by the $\check{T}$-torsor $(E_{\check{G}},\varphi)'$ and $\phi\in X^{*}(\check{T})=X_{*}(T)$. And $\phi((E_{\check{G}},\varphi)')_{\widetilde{D}}$ stands for the restriction of the line bundle $\phi((E_{\check{G}},\varphi)')$ to the divisor $\widetilde{D}$. This defines a line bundle on 
$$(\cam^{un})^{(n)}\times X_{*}(T)\times\higc(\cam).$$

When we fix a cocharacter $\phi$, we get a line bundle on $(\cam^{un})^{(n)}\phi\times\higc(\cam)$. More generally, for any finite set of cocharacters $\phi_{1}, \cdots, \phi_{k}$ and integers $n_{1}, \cdots, n_{k}$, one can construct a line bundle on 
$$(\cam^{un})^{(n_{1})}\phi_{1}\times\cdots\times(\cam^{un})^{(n_{k})}\phi_{k}\times\higc(\cam)$$
by pulling back the line bundle on each component $(\cam^{un})^{(n_{i})}\times\higc(\cam)$ and then take the tensor product. 
\begin{rem}
If we take $n=k=1$, then the construction above implies that the fiber of the line bundle at the point $(\widetilde{x},\phi,(E_{G},\varphi))\in\cam^{un}\times X_{*}(T)\times\higc(\cam)$ is given by $\phi((E_{G},\varphi)')\mid_{\widetilde{x}}$.
\end{rem}

We have the following lemma:
\begin{lemma}\label{descend to bunt}
Let $\phi_{1}, \cdots, \phi_{r}$ be a basis of $X_{*}(T)$ and $n_{i}$'s be sufficiently large integers. Then the line bundle on
$$(\cam^{un})^{(n_{1})}\phi_{1}\times\cdots\times(\cam^{un})^{(n_{r})}\phi_{r}\times\higc(\cam)$$
descends to $\bunt(\cam)\times\higc(\cam)$. Moreover, if we denote this line bundle by $\mathcal{P}'_{G}$, then the restriction of $\mathcal{P}'_{G}$ to $\bunt(\cam)\times\higc^{reg}(\cam)\simeq \bunt(\cam)\times\jtgccam$ is naturally a biextension of $\bunt(\cam)\times\jtgccam$ by $\mathbb{G}_{m}$ (For the notion of biextension, see Subsection $1.3$ of \cite{Biext} and Subsection $10.3$ of \cite{AF}).
\end{lemma}
\begin{proof}
First of all, $\cam$ is a Gorenstein integral curve, the morphism
$$(\cam^{un})^{(n_{1})}\phi_{1}\times\cdots\times(\cam^{un})^{(n_{r})}\phi_{r}\rightarrow\bunt(\cam)$$
is smooth since we assume $n_{i}$'s to be large. Let us look at the restriction of the line bundle to 
$$(\cam^{un})^{(n_{1})}\phi_{1}\times\cdots\times(\cam^{un})^{(n_{r})}\phi_{r}\times\higc^{reg}(\cam) .$$
By part $(3)$ of Proposition~\ref{prop of jtors}, we have a morphism $\higc^{reg}(\cam)\rightarrow\buntc(\cam)$ (Recall that we have fixed a Kostant section for $\higc$). Using our discussions about the Poincar\'e line bundle on $\bunt(\cam)\times\buntc(\cam)$ in Proposition~\ref{propositions of poincare line bundle on bunt} of Appendix $A$, one checks directly that the line bundle constructed above is isomorphic to the pullback of the Poincar\'e line bundle on $\bunt(\cam)\times\buntc(\cam)$ via:
\begin{gather}
(\cam^{un})^{(n_{1})}\phi_{1}\times\cdots\times(\cam^{un})^{(n_{r})}\phi_{r}\times\higc^{reg}(\cam)\rightarrow 
\bunt(\cam)\times\higc^{reg}(\cam)\rightarrow \notag \\
\bunt(\cam)\times\buntc(\cam) \notag
\end{gather}
So we get descend data on $(\cam^{un})^{(n_{1})}\phi_{1}\times\cdots\times(\cam^{un})^{(n_{r})}\phi_{r}\times\higc^{reg}(\cam)$. Now we look at the universal case, i.e, we work over $H_{int}$ (Recall that $H_{int}$ stands for the open subscheme of $H$ parameterizing integral cameral covers) and look at the line bundle on: 
$$(\cam^{un})^{(n_{1})}\phi_{1}\times_{H_{int}}\cdots\times_{H_{int}}(\cam^{un})^{(n_{r})}\phi_{r}\times_{H_{int}}\higc .$$
It is equipped with descend data on the open set:
$$(\cam^{un})^{(n_{1})}\phi_{1}\times_{H_{int}}\cdots\times_{H_{int}}(\cam^{un})^{(n_{r})}\phi_{r}\times_{H_{int}}\higc^{reg} .$$
By part $(2)$ and $(3)$ of Proposition~\ref{geo of hig}, we see that the complement of $\higc^{reg}$ in $\higc$ has codimension greater than or equals to two. Hence the descend data extends. This shows the line bundle descends to $\bunt(\cam)\times\higc(\cam)$. 

The claim that $\mathcal{P}'_{G}$ is a biextension follows directly from the observation that the restriction of $\mathcal{P}'_{G}$ to $\bunt(\cam)\times\jtgccam$ is naturally isomorphic to the pullback of the Poincar\'e line bundle on $\bunt(\cam)\times\buntc(\cam)$ via $\bunt(\cam)\times\jtgccam\rightarrow\bunt(\cam)\times\bunt(\cam)$. 

\end{proof}

Next we will show the following:
\begin{lemma}\label{new descend to bunt}
The line bundle $\mathcal{P}'_{G}$ on $\bunt(\cam)\times\higc(\cam)$ constructed in the previous lemma descend to $\jtgcam\times\higc(\cam)$. 
\end{lemma}
The following lemma shows that it is enough to prove the descend statement over the open part $\bunt(\cam)\times\higc^{reg}(\cam)\simeq\bunt(\cam)\times\jtgccam$:
\begin{lemma}\label{a first reduction}
Consider the universal case $\bunt(\cam/H_{int})\times_{H_{int}}\higc$. If the restriction of $\mathcal{P}'_{G}$ to $\bunt(\cam/H_{int})\times_{H_{int}}\higc^{reg}$ descend to $\jtg\times_{H_{int}}\higc^{reg}$, then the descend data extends to $\bunt(\cam/H_{int})\times_{H_{int}}\higc$.

\end{lemma}
\begin{proof}
Again, this follows from the fact that the codimension of the complement of $\higc^{reg}$ is greater than or equals to two.
\end{proof}

So from now on we will look at $\jtgcam\times\jtgccam\simeq\jtgcam\times\higc^{reg}(\cam)$. We have the following observation:
\begin{lemma}\label{a key lemma}
\begin{enumerate1}
\item Consider the diagram:
\[\xymatrixcolsep{4pc}\xymatrix{
\bunt(\cam)\times\buntc(\cam) \ar[r]^{Nm_{G}\times\textrm{id}} \ar[d]^{\textrm{id}\times Nm_{\check{G}}} &\jtgcam\times\buntc(\cam)\\
\bunt(\cam)\times\jtgccam .
}
\]
Consider $\mathcal{P}'_{G}$ on $\bunt(\cam)\times\jtgccam$ constructed in Lemma~\ref{descend to bunt} and its counterpart $\mathcal{P}'_{\check{G}}$ on $\jtgcam\times\buntc(\cam)$. Then one has an isomorphism of biextensions: $$(\textrm{id}\times Nm_{\check{G}})^{*}(\mathcal{P}'_{G})\simeq (Nm_{G}\times\textrm{id})^{*}(\mathcal{P}'_{\check{G}}) .$$
\item Let $\mathcal{K}$ be the kernel of $\bunt(\cam)\rightarrow\jtgcam$. Then there exists a trivialization of the restriction of $\mathcal{P}'_{G}$ to $\mathcal{K}\times\jtgccam$ that is compatible with the group structure on $\mathcal{K}$. 
\end{enumerate1}
\end{lemma}
\begin{proof}
For part $(1)$, recall that one has a $W$-action on $\bunt(\cam)$ given by (See Subsection $3.1$ of \cite{Geometric Langlands} and Section $5$ of \cite{Gerbe of Higgs}):
$$w(\mathcal{F}_{T})=T\times^{T,w}(w^{-1})^{*}(\mathcal{F}_{T}) .$$
One set 
\begin{gather}
Nm_{T}: \bunt(\cam)\rightarrow\bunt(\cam) \notag\\
\mathcal{F}_{T}\rightarrow\otimes_{w\in W}w(\mathcal{F}_{T}) \notag
\end{gather}
 and similarly for $Nm_{\check{T}}$.
Let us denote the Poincar\'e line bundle on $\bunt(\cam)\times\buntc(\cam)$ by $\mathcal{Q}$. From the construction in Lemma~\ref{descend to bunt} we see that $$(\textrm{id}\times Nm_{\check{G}})^{*}(\mathcal{P}'_{G})\simeq (\textrm{id}\times Nm_{\check{T}})^{*}(\mathcal{Q}) .$$ 
While
$$(Nm_{G}\times\textrm{id})^{*}(\mathcal{P}'_{\check{G}})\simeq (Nm_{T}\times\textrm{id})^{*}(\mathcal{Q}) .$$
Using the formula for $\mathcal{Q}$ in Appendix $A$ one checks that we have a natural isomorphism
$$(\textrm{id}\times Nm_{\check{T}})^{*}(\mathcal{Q})\simeq (Nm_{T}\times\textrm{id})^{*}(\mathcal{Q}) .$$
This finishes part $(1)$.

For part $(2)$, we look at the universal case first. Let $\mathcal{K}_{H_{int}}$ be the kernel of $\bunt(\cam/H_{int})\rightarrow\jtg$. Part $(4)$ of Proposition~\ref{geo of hig} and the constructions in Section $3$ of \cite{Geometric Langlands} implies that when we restrict to the open
set $H^{\circ}$, there exists a trivialization of $\mathcal{P}'_{G}$ on $\mathcal{K}_{H^{\circ}}\times_{H^{\circ}}\jtgc$ that is compatible with the group structure on $\mathcal{K}_{H^{\circ}}$. So it remains to show this trivialization extends. By part $(1)$ of the lemma, we see that the pullback of $\mathcal{P}'_{G}$ to $\mathcal{K}_{H_{int}}\times_{H_{int}}\buntc(\cam/H_{int})$ is trivial. So we have a meromorphic trivialization of $\mathcal{P}'_{G}$ over $\mathcal{K}_{H_{int}}\times_{H_{int}}\jtgc$ which becomes regular when we pullback to the smooth cover $\mathcal{K}_{H_{int}}\times_{H_{int}}\buntc(\cam/H_{int})$. Hence the trivialization on $\mathcal{K}_{H^{\circ}}\times_{H^{\circ}}\jtgc$ extends.

\end{proof}

With the previous lemma at hand, we can now prove Lemma~\ref{new descend to bunt}:
\begin{proof}
By Lemma~\ref{a first reduction} we only need to prove the restriction of $\mathcal{P}'_{G}$ to $\bunt(\cam)\times\jtgccam$ descends. Let $\mathcal{K}$ be the kernel of $\bunt(\cam)\rightarrow\jtgcam$. Since $\mathcal{P}'_{G}$ is a biextension, it is enough to show that the restriction of $\mathcal{P}'_{G}$ to $\mathcal{K}\times\jtgccam$ is equipped with a  trivialization that is compatible with the group structure of $\mathcal{K}$. So part $(2)$ of Lemma~\ref{a key lemma} finishes the proof. 
\end{proof}

\begin{corollary}\label{finish the descend for adjoint}
The line bundle $\mathcal{P}'_{G}$ on $\bunt(\cam)\times\higc(\cam)$ descends to $\jtgcam\times\higc(\cam)$. Denote it by $\mathcal{P}_{G}$. We will call it the Poincar\'e line bundle on $\jtgcam\times\higc(\cam)$. Furthermore, if we restrict the Poincar\'e line bundle $\mathcal{P}$ to $$\jtgcam\times\higc^{reg}(\cam)\simeq\jtgcam\times\jtgccam$$ and view it as a line bundle on $\jtgcam\times\jtgccam$. Then $\mathcal{P}_{G}$ is naturally a biextension of $\jtgcam\times\jtgccam$ by $\mathbb{G}_{m}$ 
\end{corollary}
\begin{proof}
By Lemma~\ref{new descend to bunt} we know $\mathcal{P}'_{G}$ descends. Since $\mathcal{P}'_{G}$ is naturally a biextension on $\bunt(\cam)\times\jtgccam$, to show the biextension structure descends we need to show the biextension structure is compatible with descend data. To prove this one looks at the universal case, i.e. the line bundle $\mathcal{P}'_{G}$ on $\bunt(\cam/H_{int})\times_{H_{int}}\jtgc$. By part $(4)$ of Proposition~\ref{geo of hig}, over the open set $H^{\circ}\subseteq H_{int}$, the biextension structure on $\bunt(\cam/H^{\circ})\times_{H^{\circ}}\jtgc$ descends to $\jtg\times_{H^{\circ}}\jtgc$. So compatibility holds over a dense open set, hence everywhere. 
\end{proof}

Let us also state some additional compatibilities:
\begin{corollary}\label{additional compatibility}
\begin{enumerate1}
\item One has an isomorphism of biextensions:
$$\mathcal{P}_{G}\mid_{\jtgcam\times\jtgccam}\simeq\mathcal{P}_{\check{G}}\mid_{\jtgcam\times\jtgccam} .$$
That is, the line bundles $\mathcal{P}_{G}$ and $\mathcal{P}_{\check{G}}$ defines isomorphic biextensions when restricted to $\jtgcam\times\jtgccam$.
\item Consider:
$$\jtgcam\times\jtgcam\times\higc(\cam)\xrightarrow{m_{G}\times\textrm{id}}\jtgcam\times\higc(\cam). $$
Then we have $(m_{G}\times\textrm{id})^{*}(\mathcal{P}_{G})\simeq p_{13}^{*}(\mathcal{P}_{G})\otimes p_{23}^{*}(\mathcal{P}_{G})$.
\item Consider 
$$\jtgcam\times\jtgccam\times\higc(\cam)\xrightarrow{\textrm{id}\times m_{\check{G}}}\jtgcam\times\higc(\cam) .$$
We have $(\textrm{id}\times m_{\check{G}})^{*}(\mathcal{P_{G}})\simeq p_{12}^{*}(\mathcal{P}_{G})\otimes p_{13}^{*}(\mathcal{P}_{G})$.
\item Let $(-1)$ be the involution on $\jtgcam$ given by $x\rightarrow -x$. Then we have $((-1)\times\textrm{id})^{*}\mathcal{P}_{G}\simeq \check{\mathcal{P}}_{G}$.
\end{enumerate1}
\end{corollary}
\begin{proof}
For part $(1)$, we will work with the universal case. Consider $\jtg\times_{H_{int}}\times\jtgc$. Part $(4)$ of Proposition~\ref{geo of hig} implies that the claim is true over the open subscheme $H^{\circ}$. So one needs to prove this meromorphic identification extends. To see that we take the pullbacks of $\mathcal{P}_{G}$ and $\mathcal{P}_{\check{G}}$ to $\bunt(\cam/H_{int})\times_{H_{int}}\buntc(\cam/H_{int})$. Part $(1)$ of Lemma~\ref{a key lemma} implies that the pullback of the meromorphic identification to the smooth cover $\bunt(\cam/H_{int})\times_{H_{int}}\buntc(\cam/H_{int})$ is actually regular. Hence it extends.

For part $(2)$, we will again consider the universal case. Using the fact that the complement of $\higc^{reg}\simeq\jtgc$ in $\higc$ has codimension greater than or equals to two, we see that we only need to prove the claim over the open set $\jtg\times_{H_{int}}\jtg_{H_{int}}\times_{H_{int}}\jtgc$. By part $(4)$ of Proposition~\ref{geo of hig} and the construction in Section $3$ of \cite{Geometric Langlands}, the statement holds over the open set $H^{\circ}$. So one needs to show this isomorphism extends. It follows from the construction of $\mathcal{P}_{G}$ that if we pullback to $\bunt(\cam/H_{int})\times_{H_{int}}\bunt(\cam/H_{int})\times_{H_{int}}\jtgc$, the claim holds. So the isomorphism on $H^{\circ}$ extends.

The proof for part $(3)$ and part $(4)$ are similar to the proof of part $(2)$. 

\end{proof}

For later use we shall make some additional analysis in the case when $G$ is adjoint. Let us first make some general observations about $\jtg$ and $\higc$ in this case:
\begin{lemma}\label{more geometry}
\begin{enumerate1}
\item $\jtg$ is a quasi-projective scheme over $H_{int}$. If we fix a point in $H_{int}$ and consider the corresponding fiber $\jtgcam$, then we have $\pi_{0}(\jtgcam)\simeq X_{*}(T)/W\simeq Z(\check{G})^{\check{}}$.
\item $\higc$ is fiberwise connected over $H_{int}$. It admits a moduli space which we will denote by $\mhigc$. $\mhigc$ is a projective scheme over $H_{int}$ and $\higc$ is a $Z(\check{G})$ gerbe over $\mhigc$.
\item Let $\mjtgc$ be the moduli space of $\jtgc$ and let $\mathcal{D}$ be the line bundle on $H_{int}$ given by the dual of the determinant of the lie algebra of $\mjtgc$. Then the relative dualizing sheaf of $\mhigc$ over $H_{int}$ is isomorphic to $f^{*}(\mathcal{D})$ where $f$ is the morphism $\mhigc\rightarrow H_{int}$.
\end{enumerate1}
\end{lemma}

\begin{proof}
We start with part $(1)$. By our assumptions on $L$, for each root $\alpha$, the ramification divisor $\widetilde{D}_{\alpha}$ is nonempty. Hence we see from the formula of $J_{G,\cam}$ that $H^{0}(X,J_{G,\cam})\simeq Z(G)=0$. By our assumptions on the cameral over, we have that all Higgs bundles with cameral cover $\cam$ are stable. We conclude from Section $2$ of \cite{Higgs Faltings} that the stack $\jtgcam$ is isomorphic to its moduli space, which is quasi-projective. The claim on $\pi_{0}$ follows from Proposition $4.10.13$ of \cite{Fundamental lemma} and the fact that $J^{0}_{G,\cam}=J_{G,\cam}$ since $G$ is adjoint.

For part $(2)$, since $\check{G}$ is simply connected, the claim that $\higc(\cam)$ is connected follows from part $(2)$ of Proposition~\ref{geo of hitchin fibration}. By \cite{Higgs Faltings} the moduli space $\mhigc(\cam)$ is projective. Since we assume $\cam$ is integral, all Higgs bundles with cameral cover $\cam$ are stable. And because $\check{G}$ is simply connected, we have $\check{T}^{W}\simeq Z(\check{G})$. Now Proposition $4.11.2$ of \cite{Fundamental lemma} implies that the automorphism group of every Higgs bundle with cameral cover $\cam$ is equal to $Z(\check{G})$. In this case Section $2$ of \cite{Higgs Faltings} implies that $\higc(\cam)$ is a $Z(\check{G})$ gerbe over $\mhigc(\cam)$. 

For part $(3)$, the expression for the dualizing sheaf is true over the group scheme $\mjtg$. Since $\mhig$ is Gorenstein and $\mjtg$ is dense in $\mhigc$ with the codimension of the complement is greater than or equals to two, this implies the claim.
\end{proof}

Let us consider the functor $\pic(\higc(\cam))$, which parameterizes line bundles on $\higc(\cam)$ together with a trivialization at the point of $\higc(\cam)$ corresponds to the Kostant section. We have:
\begin{lemma}\label{additional data adjoint}
\begin{enumerate1}
\item $\pic(\higc(\cam))$ is a scheme. It fits into an exact sequence:
$$0\rightarrow\pic(\mhigc(\cam))\rightarrow\pic(\higc(\cam))\rightarrow Z(\check{G})^{\check{}}\rightarrow 0$$
where $Z(\check{G})^{\check{}}$ is the Cartier dual group of $Z(\check{G})$.
\item The group homomorphism $\jtgcam\rightarrow\pic(\higc(\cam))$ induces a morphism between exact sequences:
$$\xymatrix{
0\ar[r] & \jtgcam^{0} \ar[r] \ar[d] & \jtgcam \ar[r] \ar[d] & \pi_{0}(\jtgcam) \ar[r] \ar[d] & 0\\
0\ar[r] & \pic(\mhigc(\cam)) \ar[r] & \pic(\higc(\cam)) \ar[r] & Z(\check{G})^{\check{}} \ar[r] & 0
}
$$
where $\jtgcam^{0}$ denotes the neutral component of $\jtgcam$. The arrow $\pi_{0}(\jtgcam)\rightarrow Z(\check{G})^{\check{}}$ is given by part $(1)$ of Lemma~\ref{more geometry}.
\end{enumerate1}
\end{lemma}

\begin{proof}
Let us start with part $(1)$. By Lemma~\ref{more geometry}, the Kostant section induces a morphism $BZ(\check{G})\rightarrow\higc(\cam)$, hence we get a morphism $\pic(\higc(\cam))\rightarrow\pic(BZ(\check{G}))$. It is well-known that $\pic(BZ(\check{G}))\simeq Z(\check{G})^{\check{}}$, so we get $\pic(\higc(\cam))\rightarrow Z(\check{G})^{\check{}}$. We first will prove it is surjective. For this we need to analyze the action of the automorphism group $Z(\check{G})$ on fibers of the Poincar\'e line bundle. Let us fix a point $\widetilde{x}\in\cam^{un}$ and a cocharacter $\phi\in X_{*}(T)$. Consider the point $\phi_{*}(O(\widetilde{x}))\in\bunt(\cam)$ and look at the restriction of the line bundle $\mathcal{P}'_{G}$ on $\phi_{*}(O(\widetilde{x}))\times\higc(\cam)$. From the construction of the Poincar\'e line bundle on $\cam^{un}\times X_{*}(T)\times\higc(\cam)$, it is not hard to check that for any point $s\in\higc(\cam)$, the automorphism of the fiber of $\mathcal{P}'_{G}$ at $(\phi_{*}(O(\widetilde{x})),s)\in\bunt(\cam)\times\higc(\cam)$ induced by an element $y$ in $Z(\check{G})$ is given by multiplication by $\phi(y)$. Here we have identified $\phi$ with an element in $X^{*}(\check{T})$, hence it induces a character of $Z(\check{G})$ by restriction. Since $X_{*}(T)\simeq X^{*}(\check{T})\twoheadrightarrow Z(\check{G})^{\check{}}$ is surjective, we see that $\pic(\higc(\cam))\rightarrow Z(\check{G})^{\check{}}$ is surjective. Next we will prove the kernel is isomorphic to $\pic(\mhigc(\cam))$. Indeed, since one has a natural morphism $\higc(\cam)\rightarrow\mhigc(\cam)$, we get a functor $\pic(\mhig(\cam))\rightarrow\pic(\higc(\cam))$ given by pullback line bundles. Proposition~\ref{more geometry} implies that the category of line bundles on $\mhigc(\cam)$ is equivalent to the category of line bundles on $\higc(\cam)$ such that the action of $Z(\check{G})$ on the fiber of the line bundle at every point of $\higc(\cam)$ is trivial. Since $\higc(\cam)$ is connected by Proposition~\ref{more geometry}, this is equivalent to the action of $Z(\check{G})$ on the fiber at the point given by the Kostant section is trivial. Part $(1)$ follows from this. 

For part $(2)$, one needs to show that $\pi_{0}(\jtgcam)\rightarrow Z(\check{G})^{\check{}}$ is given by the natural isomorphism $X_{*}(T)/W\simeq Z(\check{G})^{\check{}}$. Recall that we have $\pi_{0}(\bunt(\cam))\simeq X_{*}(T)$ and that $\bunt(\cam)\rightarrow\jtgcam$ induces the natural projection $\pi_{0}(\bunt(\cam))\simeq X_{*}(T)\rightarrow X_{*}(T)/W\simeq\pi_{0}(\jtgcam)$, see Lemma~\ref{surjection on pio}. Hence for any $\phi\in X_{*}(T)$ and $\widetilde{x}\in\cam^{un}$, the image of $(\widetilde{x},\phi)$ in $\pi_{0}(\jtgcam)$ under the morphisms:
\begin{gather}
\cam^{un}\times X_{*}(T)\rightarrow\bunt(\cam)\rightarrow\jtgcam\rightarrow\pi_{0}(\jtgcam) \notag \\
(\widetilde{x},\phi)\rightarrow \phi_{*}(O(\widetilde{x}))\xrightarrow{Nm} Nm(\phi_{*}(O(\widetilde{x})))\rightarrow\pi_{0}(\jtgcam) \notag
\end{gather}
is given by the class of $\phi$ in $X_{*}(T)$. Hence if we denote the restriction of the Poincar\'e line bundle to $Nm(\phi_{*}(O(\widetilde{x})))\times\higc(\cam)$ by $\mathcal{P}_{Nm(\phi_{*}(O(\widetilde{x})))}$, we need to prove that if we look at the $Z(\check{G})$ action on the fiber of $\mathcal{P}_{Nm(\phi_{*}(O(\widetilde{x})))}$ at any point in $\higc(\cam)$, then the action is given by the character of $Z(\check{G})$ induced by $\phi$. To see this, notice that by our construction of the Poincar\'e line bundle on $\cam^{un}\times X_{*}(T)\times\higc(\cam)$, the fiber of $\mathcal{P}_{Nm(\phi_{*}(O(\widetilde{x})))}$ at $(E_{G},\varphi)$ is given by $\phi((E_{G},\varphi)')\mid_{\widetilde{x}}$. Here let us remind the reader that $(E_{G},\varphi)'$ stands for the $\check{T}$-bundle on $\cam^{un}$ obtained from the Higgs bundle $(E_{G},\varphi)$, see part $(2)$ of Proposition~\ref{prop of jtors}. And $\phi((E_{G},\varphi)')$ is the line bundle on $\cam^{un}$ obtained from the $\check{T}$-bundle $(E_{G},\varphi)'$ by identifying $\phi$ as an element in $X^{*}(\check{T})$. The claim now follows directly from this expression.

\end{proof}

For the purpose of the next subsection, let us also prove the following:
\begin{lemma}\label{pullback from normalization}
Let $G$ be adjoint. Let $\cam$ be a cameral cover satisfying the assumptions in Appendix $B$ and let $\cam^{n}$ be its normalization. Then the biextension $\mathcal{P}_{G}$ on $\jtgcam\times\jtgccam$ is the pullback of a biextension $\mathcal{P}^{n}_{G}$ on $\jtgcamn\times\jtgccamn$.

\end{lemma}

\begin{proof}
By Lemma~\ref{describe kernel}, the biextension $\mathcal{P}'$ on $\bunt(\cam)\times\jtgccam$ actually lives on $X_{*}(T)\otimes\pic(\cam)\times\jtgccam$ since $\bunt(\cam)\rightarrow\jtgcam$ factors through $X_{*}(T)\otimes\pic(\cam)$. Hence the automorphism group $T(k)$ acts trivially on each fiber. The construction of the line bundle $\mathcal{P}'$ in Lemma~\ref{new descend to bunt} shows that the line bundle 
$\mathcal{P}'$ is the pullback of the Poincar\'e line bundle on $\bunt(\cam)\times\buntc(\cam)$ via $\bunt(\cam)\times\jtgccam\rightarrow\bunt(\cam)\times\buntc^{W}(\cam)\rightarrow\bunt(\cam)\times\buntc(\cam)$. By our discussions about the Poincar\'e line bundles in Appendix $A$, we see that the Poincar\'e line bundle on $\bunt(\cam)\times\buntc(\cam)$ is isomorphic to the pullback of the Poincar\'e line bundle on $\bunt(\cam^{n})\times\buntc(\cam^{n})$. Hence $\mathcal{P}'$ is isomorphic to the pullback of the line bundle on $\bunt(\cam)\times\buntc^{W}(\cam^{n})$. Since the following diagram is commutative:
$$\xymatrix{
\jtgccam \ar[r] \ar[d] & \buntc^{W}(\cam) \ar[r] \ar[d] & \buntc(\cam^{n}) \ar[d] \\
\jtgccamn \ar[r] & \buntc^{W}(\cam^{n}) \ar[r] & \buntc(\cam^{n}) ,
}
$$
we conclude that the biextension $\mathcal{P}'$ is isomorphic to the pullback of a biextension on $\bunt(\cam)\times\jtgccamn$ via $\bunt(\cam)\times\jtgccamn\rightarrow\bunt(\cam)\times\buntc^{W}(\cam^{n})$. Since $T(k)$ acts trivially on the fiber, this biextension actually lives on $X_{*}(T)\otimes\pic(\cam)\times\jtgccamn$. Since $\cam^{n}$ is a smooth cameral cover over $X$, by Lemma~\ref{more geometry} we conclude that $\jtgccamn$ is a $Z(\check{G})$ gerbe over an abelian variety. If we denote the dual Picard stack of $\jtgccamn$ by $\mathbb{D}(\jtgccamn)$, then it is a projective algebraic group. Hence one has a morphism of algebraic groups $X_{*}(T)\otimes\pic(\cam)\rightarrow\mathbb{D}(\jtgccamn)$. Since the line bundle on $X_{*}(T)\otimes\pic(\cam)\times\jtgccamn$ is the pullback of the line bundle on $X_{*}(T)\otimes\pic(\cam)\times\buntc^{W}(\cam^{n})$, using the $W$-equivariance and our discussions about Poincar\'e line bundles in Appendix $A$, we conclude that elements in $X_{*}(T)\otimes\pic(\cam)$ of the form $w(\phi)\otimes (w^{-1})^{*}(\mathcal{M})-\phi\otimes\mathcal{M}$ will be sent to zero in $\mathbb{D}(\jtgccamn)$. Here $w\in W$, $\phi\in X_{*}(T)$ and $\mathcal{M}$ is a line bundle. From Lemma~\ref{describe kernel} we conclude that if $K$ is the kernel of $X_{*}(T)\otimes\pic(\cam)\rightarrow\jtgcam$, then $K$ will be sent to zero under $X_{*}(T)\otimes\pic(\cam)\rightarrow\mathbb{D}(\jtgccamn)$. Hence the morphism $X_{*}(T)\otimes\pic(\cam)\rightarrow\mathbb{D}(\jtgccamn)$ factors through $\jtgcam$. Moreover, since $\mathbb{D}(\jtgccamn)$ is projective, we conclude from Lemma~\ref{structure of jtors} that it factors through $\jtgcamn$. Hence biextension descends to $\jtgcamn\times\jtgccamn$. 

\end{proof}

\begin{corollary}\label{description of the fiber}
Consider the Abel-Jacobi map for $\cam^{n}$ (Since $\cam^{n}$ is a smooth cameral cover over $X$, the Abel-Jacobi map makes sense over $\cam^{n}$):
\begin{gather}
\cam^{n}\times X_{*}(\check{T})\rightarrow \buntc(\cam^{n})\xrightarrow{Nm}\jtgccamn \notag \\
\widetilde{x}\times\phi\rightarrow\phi_{*}(O(\widetilde{x}))\rightarrow Nm(\phi_{*}(O(\widetilde{x}))) . \notag
\end{gather}
Let $\mathcal{F}^{n}_{T}$ be the universal $T$ bundle on $\cam^{n}\times\bunt(\cam)$. Then the restriction of $\mathcal{P}^{n}_{G}$ to $\jtgcamn\times Nm(\phi_{*}(O(\widetilde{x})))\hookrightarrow\jtgcamn\times\jtgccamn$ is isomorphic to the pullback of the line bundle $\phi(\mathcal{F}^{n}_{T})\mid_{\widetilde{x}}$ on $\bunt(\cam^{n})$ to $\jtgcamn$ via $\jtgcamn\rightarrow\bunt^{W}(\cam^{n})\rightarrow\bunt(\cam^{n})$. 
\end{corollary}
\begin{proof}
From our discussions in Lemma~\ref{pullback from normalization} we see that the pullback of $\mathcal{P}^{n}_{G}$ to $\jtgcamn\times\buntc(\cam^{n})$ via the norm map is isomorphic to the pullback of the Poincar\'e line bundle on $\bunt(\cam^{n})\times\buntc(\cam^{n})$ via the natural morphism: $\jtgcamn\times\buntc(\cam^{n})\rightarrow\bunt(\cam^{n})\times\buntc(\cam^{n})$. Now the claim follows directly from the expression of the Poincar\'e line bundle in Appendix $A$.
\end{proof}

\begin{lemma}\label{describe kernel}
Let $G$ be adjoint. The morphism $\bunt(\cam)\rightarrow\jtgcam$ induces a morphism: $X_{*}(T)\otimes\pic(\cam)\rightarrow\jtgcam$. Its kernel is equal to the image of:
\begin{gather}
\mathbb{Z}[W]\otimes X_{*}(T)\otimes\pic(\cam)\rightarrow X_{*}(T)\otimes\pic(\cam) \notag \\
w\otimes\phi\otimes\mathcal{M}\rightarrow w(\phi)\otimes (w^{-1})^{*}(\mathcal{M})-\phi\otimes\mathcal{M} \notag
\end{gather}
where $\mathbb{Z}[W]$ stands for the group ring of $W$ over $\mathbb{Z}$.
\end{lemma}
\begin{proof}
Since $X_{*}(T)\otimes\pic(\cam)$ is the moduli space of $\bunt(\cam)$ and $\jtgcam$ is a scheme by Lemma~\ref{more geometry}, we have a factorization $$\bunt(\cam)\rightarrow X_{*}(T)\otimes\pic(\cam)\rightarrow\jtgcam .$$ Let $\cam\xrightarrow{\pi} X$ be the cameral cover. Since $G$ is adjoint, by our discussions in Section $3$, we have $J_{G,\cam}=J^{0}_{G,\cam}$. Using Proposition~\ref{norm map} we see that the morphism of sheaves $\pi_{*}(T\times\cam)\rightarrow J_{G,\cam}$ is surjective. Let $K$ be its kernel. Since $H^{1}(\cam,T)\simeq H^{1}(X,\pi_{*}(T\times\cam))$, the kernel of $H^{1}(\cam,T)\rightarrow H^{1}(X,J_{G,\cam})$ is equal to the image of $H^{1}(X,K)$ in $H^{1}(X,\pi_{*}(T\times\cam))$. It is easy to see that the image of:
\begin{gather}
\mathbb{Z}[W]\otimes\pi_{*}(T\times\cam)\rightarrow \pi_{*}(T\times\cam) \notag\\
w\otimes s \rightarrow w(s)-s \notag
\end{gather}
is contained in $K$, where $s$ is a local section of $\pi_{*}(T\times\cam)$. Denote the image of this morphism by $A$ and denote its kernel by $B$. We have exact sequences:
\begin{gather}
0\rightarrow B\rightarrow\mathbb{Z}[W]\otimes\pi_{*}(T\times\cam)\rightarrow A\rightarrow 0 \notag\\
0\rightarrow A\rightarrow K\rightarrow K/A\rightarrow 0 .
\end{gather}
Since $\cam\rightarrow X$ is unramified over $X^{un}$, we have that $A=K$ over $X^{un}$. Hence $K/A$ has finite support. Since $B$ is generically a torus, Lemma~\ref{Tsen} tells us that $H^{1}(X,\mathbb{Z}[W]\otimes\pi_{*}(T\times\cam))\rightarrow H^{1}(X,A)$ is surjective. Since $K/A$ has finite support, we also have that $H^{1}(X,A)\rightarrow H^{1}(X,K)$ is surjective. So we have shown that the image of $H^{1}(X,K)$ in $H^{1}(X,\pi_{*}(T\times\cam))$ is equal to the image of $H^{1}(X,\mathbb{Z}[W]\otimes\pi_{*}(T\times\cam))$ in $H^{1}(X,\pi_{*}(T\times\cam))$. Now identity points in $X_{*}(T)\otimes\pic(\cam)$ with $H^{1}(X,\pi_{*}(T\times\cam))$, we see that the claim follows directly from this. 
\end{proof}

\subsection{Cohomology of the Poincar\'e line bundle}
In this subsection we determine the cohomology of the Poincar\'e line bundle and use it to show the Fourier-Mukai functor $\quasicoh(\jtgcam)\rightarrow \quasicoh(\higc(\cam))$ induced by the Poincar\'e line bundle is fully faithful. The main results are Theorem~\ref{main theorem on coho}, Theorem~\ref{2nd main theorem of cohomology of poincare line bundle} and Corollary~\ref{fully faithful fourier-mukai}.

In this subsection we shall always assume the $\delta$ constant strata $H_{int,\delta}$ has codimension greater than or equals to $\delta$, see $[ ]$. In the case when $L=\omega_{X}(D)$ for some effective divisor $D$, then one can show that this assumption always holds.

The first theorem of this subsection is the following:
\begin{theorem}\label{main theorem on coho}
Assume $G$ is semisimple. Let $\mathcal{P}_{G}$ be the Poincar\'e line bundle on $\jtgcam\times\higc(\cam)$. Consider:
$$\xymatrix{
\jtgcam\times\higc(\cam) \ar[d]_{p_{1}}\\
\jtgcam .
}
$$
Let $d$ be the dimension of $\higc(\cam)$ and let $e$ be the morphism $\textrm{Spec}(k)\rightarrow\jtgcam$ corresponds to the unit of $\jtgcam$. Then we have $Rp_{1 *}(\mathcal{P}_{G})\simeq e_{*}(k)[-d]$. 
\end{theorem}

We will first analyze the case when $G$ is adjoint. Let us make some general observations about the support of $Rp_{1,*}(\mathcal{P}_{G})$ in this case:
\begin{lemma}\label{codim estimate}
\begin{enumerate1}
\item Let $\alpha$ be a point in $\jtgcam$ such that $H^{*}(\higc,\mathcal{P}_{\alpha})\neq 0$ where $\mathcal{P}_{\alpha}$ stands for the restriction of $\mathcal{P}_{G}$ to $\{\alpha\}\times\higc$. Then $\alpha\in\jtgcam^{0}$ where $\jtgcam$ stands for the neutral component. Moreover, the restriction of $\mathcal{P}_{\alpha}$ to $\jtgccam$ is trivial.
\item Under the same assumptions as in part $(1)$, the $J_{G,\cam}$ torsor $\alpha$ on $X$ is trivial when restricted to the open set $X^{un}$. 
\item $\dimension\supp Rp_{1 *}(\mathcal{P}_{G})\leq \delta$, where $\delta$ stands for the $\delta$ invariant associated to $\cam$, see Subsection $5.6$ of \cite{Fundamental lemma}.
\end{enumerate1}
\end{lemma} 
\begin{proof}
We will use the same strategy as in \cite{Poincare line bundle}. 

For part $(1)$, using part $(2)$ of Lemma~\ref{additional data adjoint}, we see that if $\jtg^{\gamma}$ is a component of $\jtg$ corresponds to a nonzero element $\gamma\in\pi_{0}(\jtg)$, then the $Z(\check{G})$ action on the restriction of $\mathcal{P}_{G}$ to $\jtg^{\gamma}\times_{H_{int}}\higc$ is nontrivial. Since $\higc$ is a $Z(\check{G})$ gerbe over $\mhigc$ and that $p_{1}$ factors as 
$$\jtg\times_{H_{int}}\higc\rightarrow\jtg\times_{H_{int}}\mhigc\rightarrow\jtg ,$$
we conclude that the restriction of $Rp_{1*}\mathcal{P}_{G}$ to $\jtg^{\gamma}$ is trivial. Hence it is supported on $\jtg^{0}$. So now we shall work over $\jtg^{0}$. In this case $\mathcal{P}_{G}$ lives on $\jtgcam^{0}\times\mhigc(\cam)$. Notice that if we denote the restriction of $\mathcal{P}_{\alpha}$ to $\mjtgccam$ by $A$ ($\mjtgccam$ stands for the moduli space of $\jtgccam$), then part $(3)$ of Corollary~\ref{additional compatibility} implies $A$ is an abelian group which is an extension of $\mjtgccam$ by $\mathbb{G}_{m}$ and it acts on $\mathcal{P}_{\alpha}$. When $H^{*}(\mhigc,\mathcal{P}_{\alpha})\neq 0$, we get an action of $A$ on a vector space. Since $A$ is abelian, there exists a one dimensional subspace fixed by $A$. Moreover, it is easy to see that the subgroup $\mathbb{G}_{m}$ acts by scalars. This provides a splitting of $A$.

For part $(2)$, let us look at the pullback of $\mathcal{P}_{\alpha}$ along the Abel-Jacobi map for $\jtgccam$:
$$\cam^{un}\times X_{*}(\check{T})\xrightarrow{AJ} \jtgccam .$$
Using Part $(1)$ of Corollary~\ref{additional compatibility} and the construction of the Poincar\'e line bundle , one can check that the restriction of $AJ^{*}(\mathcal{P}_{\alpha})$ to $\cam^{un}\times\phi$ is isomorphic to $\phi(\mathcal{F}_{T})$ where $\mathcal{F}_{T}$ is the $W$-equivariant $T$-bundle on $\cam^{un}$ corresponds to $\alpha$, see part $(2)$ Proposition~\ref{prop of jtors}. $\phi(\mathcal{F}_{T})$ is the line bundle induced from the $T$-bundle $\mathcal{F}_{T}$ and the character $\phi\in X^{*}(T)\simeq X_{*}(\check{T})$. By part $(1)$, we see that $\phi(\mathcal{F}_{T})$ is trivial for all $\phi$. Moreover, from the construction of the Abel-Jacobi map, we see that it commutes with the $W$-action on $\cam^{un}\times X_{*}(\check{T})$, i.e. we have the following commutative diagram for every $w\in W$:
$$\xymatrix{
\cam^{un}\times X_{*}(\check{T}) \ar[rd]_{AJ} \ar[rr]^{w\times w} & & \cam^{un}\times X_{*}(\check{T}) \ar[ld]^{AJ}\\
 & \jtgccam & .
}
$$
This implies that $\mathcal{F}_{T}$ is actually $W$-equivariantly trivial, i.e. it is trivial as a $W$-equivariant $T$-bundle on $\cam^{un}$. Now part $(2)$ of Proposition~\ref{prop of jtors} implies the claim.

For part $(3)$, notice that by part $(1)$ and $(2)$, if $\alpha\in\supp Rp_{1 *}(\mathcal{P}_{G})$, then $\alpha$ is trivial when restricted to $X^{un}$. Hence $\alpha$ lies in the image of the natural morphism: 
\begin{equation*}
\prod_{x\in X-X^{un}} \affgras_{J_{G,\cam},x}\rightarrow \jtgcam .
\end{equation*}
where $\affgras_{J_{G},x}$ stands for the affine grassmannian associated to the group scheme $J_{G,\cam}$ at the point $x\in X$. So it remains to show the dimension of the reduced part of the ind-scheme $\prod_{x\in X-X^{un}} \affgras_{J_{G},x}$ is bounded by $\delta$. To do this let us denote the normalization of $\cam$ by $\cam^{n}$. $\cam^{n}$ is a $W$-cover over $X$. Let us set $J^{b}_{G,\cam}=\pi^{n}_{*}(T\times\cam^{n})^{W}$ where $\pi^{n}$ is the projection from $\cam^{n}$ to $X$. From its definition we see that one has a morphism of group schemes $J_{G,\cam}\rightarrow J^{b}_{G,\cam}$. Moreover, for each $x\in X$, one has $J_{G,\cam}(\widehat{O}_{x})\subseteq J^{b}_{G,\cam}(\widehat{O}_{x})$ and $J_{G,\cam}(\widehat{\mathcal{K}}_{x})\simeq J^{b}_{G,\cam}(\widehat{\mathcal{K}}_{x})$ where $\widehat{O}_{x}$ is the formal completion of $X$ at $x$ and $\widehat{\mathcal{K}}_{x}$ is the field of fractions of $\widehat{O}_{x}$. So one has an exact sequence:
\begin{equation*}
0\rightarrow \prod_{x\in X-X^{un}}J^{b}_{G,\cam}(\widehat{O}_{x})/J_{G,\cam}(\widehat{O}_{x})\rightarrow \prod_{x\in X-X^{un}} \affgras_{J_{G,\cam},x}\rightarrow\prod_{x\in X-X^{un}} \affgras_{J^{n}_{G,\cam},x}\rightarrow 0 .
\end{equation*}
The reduced part of $ \affgras_{J^{b}_{G,\cam},x}$ is discrete and it is known that 
\begin{equation*}
\dimension \prod_{x\in X-X^{un}}J^{b}_{G,\cam}(\widehat{O}_{x})/J_{G,\cam}(\widehat{O}_{x})=\delta ,
\end{equation*}
see $4.9.1$ of \cite{Fundamental lemma}. This finishes the proof.
\end{proof}

\begin{corollary}\label{Cohen-Macaulayness}
Consider the universal Poincar\'e line bundle $\mathcal{P}_{G}$ on $\jtg\times_{H_{int}}\higc$. Then $Rp_{1*}(\mathcal{P}_{G})[d]$ is a Cohen-Macaulay sheaf of codimension  $d$ supported in the neutral component $\jtg^{0}$ of $\jtg$.
\end{corollary}
\begin{proof}
By part $(2)$ of Lemma~\ref{additional data adjoint}, the restriction of the Poincar\'e line bundle on $\jtgcam^{0}\times\higc(\cam)$ descends to $\mjtgcam^{0}\mhigc$. Using part $(1)$ of Lemma~\ref{codim estimate} , we can just look at:
$$\xymatrix{
\jtg^{0}\times_{H_{int}}\mhigc \ar[d] _{p_{1}} \\
\jtg^{0} .
}
$$
Now $p_{1}$ is projective and of relative dimension $d$, hence $Rp_{1*}\mathcal{P}_{G}[d]$ is concentrated in cohomological degree less than or equals to zero. Let us also consider the dual $\mathcal{R}Hom(Rp_{1*}\mathcal{P}_{G}[d], O_{\jtg})$. By part $(3)$ of Lemma~\ref{more geometry} and part $(4)$ of Corollary~\ref{additional compatibility}, we conclude by Grothendieck duality that $\mathcal{R}Hom(Rp_{1*}\mathcal{P}_{G}[d], O_{\jtg})$ sits in cohomological degree less than or equals to $d$. Moreover, from part $(3)$ of Lemma~\ref{codim estimate} and our assumptions on the codimension of $\delta$ strata, we see that the codimension of the support of $Rp_{1*}(\mathcal{P}_{G})$ is less than or equals to $d$. So the claim follows from Lemma $7.7$ of \cite{Autoduality}.  
\end{proof}

\begin{lemma}~\label{a morphism}
Let $\mathfrak{j}$ be the line bundle on $H_{int}$ given by the determinant of the Lie algebra of $\mjtgc$. Then  one has a natural morphism $Rp_{1*}(\mathcal{P}_{G})[d]\rightarrow e_{*}(\mathfrak{j})$ where $e$ is the unit $H_{int}\xrightarrow{e}\jtg$.
\end{lemma}
\begin{proof}
Using the fact that the complement of $\higc^{reg}$ in $\higc$ has codimension greater than or equals to two, it is easy to see that the pullback of $\mathfrak{j}$ to $\mhigc$ is isomorphic to $\omega^{-1}_{\mhigc\mid H_{int}}$ where $\omega_{\mhigc\mid H_{int}}$ is the relative dualizing line bundle. By Grothendieck duality, a morphism 
$$Rp_{1*}(\mathcal{P}_{G})[d]\rightarrow e_{*}(\mathfrak{j})$$
is the same as a morphism 
$$\mathcal{P}_{G}[d]\rightarrow p^{!}_{1}(\mathfrak{j})\simeq e_{*}(O_{\mhigc})[d] .$$
Since the restriction of $\mathcal{P}_{G}$ to $e\times_{H_{int}}\mhigc$ is canonically trivial, we are done.  
\end{proof}

We will eventually prove that the morphism $Rp_{1*}(\mathcal{P}_{G})[d]\rightarrow e_{*}(\mathfrak{j})$ is an isomorphism. To do that we need the following result:
\begin{lemma}\label{a key lemma on etaleness}
The morphism $\jtgcam^{0}\rightarrow \textrm{Pic}(\mhigc)$ constructed in part $(2)$ of Lemma~\ref{additional data adjoint} is etale. 
\end{lemma}
\begin{proof}
By Corollary~\ref{desciption of H1} in Appendix $B$, one conclude that $\textrm{Pic}(\mhigc)$ is a smooth group scheme over $H_{int}$. Hence we only need to show the induced morphism at the level of lie algebras is an isomorphism. In fact, since the lie algebras are vector bundles of the same rank over $H_{int}$ by Lemma~\ref{desciption of H1}, it is enough to prove they are isomorphic over an open subscheme of $H_{int}$ whose complement has codimension greater than or equals to two. We will take the open subscheme to be $U$ constructed in Lemma~\ref{an open set of H} in Appendix $C$. We will show that for any cameral cover corresponds to a point in $U$, the induced morphism on the level of lie algebras is an isomorphism. By part $(2)$ of Lemma~\ref{an open set of H}, we see that for any point in $U$, the corresponding cameral cover either has the property that $\delta=0$ or $\delta=1$. Let us fix a cameral cover $\cam$ corresponds to a point in $U$ and let $\theta$ be a point the kernel of $\jtgcam^{0}\rightarrow\textrm{Pic}(\mhigccam)$. Let $\mathcal{P}_{\theta}$ be the restriction of the Poincar\'e line bundle to $\theta\times\mhigccam$. Since by our assumption $\mathcal{P}_{\theta}\simeq O_{\mhigccam}$, we conclude from Part $(3)$ of Lemma~\ref{codim estimate} that if $\delta=0$, then the kernel of $\jtgcam\rightarrow\textrm{Pic}(\mhigccam)$ is zero dimensional, hence the induced morphism on lie algebras is an isomorphism. When $\delta=1$, $\cam$ has nodal singularities over a unique point $x\in X$ and we have seen in the proof of part $2$ of Lemma~\ref{codim estimate} that $\theta$ lies in the image of $\textrm{Gr}_{J_{G,\cam},x}$ in $\jtgcam$. Moreover, in the proof of part $(2)$ of Lemma~\ref{codim estimate} we have seen that the neutral component of $\textrm{Gr}_{J_{G,\cam},x}$ can be identified with $J_{G,\cam^{n}}(\widehat{O}_{x})/J_{G,\cam}(\widehat{O}_{x})$ where $\cam^{n}$ is the normalization of $\cam$. Lemma~\ref{structure of jtors} implies that the neutral component of the kernel of $\jtgcam\rightarrow\textrm{Pic}(\mhigccam)$ is contained in the kernel of $\jtgcam\rightarrow\jtgcamn$, which is $\mathbb{G}_{m}$. So we conclude that the neutral component of the kernel of $\jtgcam^{0}\rightarrow\pic(\higc(\cam))$ is contained in this $\mathbb{G}_{m}$.

Next we claim that the restriction of the Poincar\'e line bundle to $\mathbb{G}_{m}\times\mhigc(\cam)$ is nontrivial, this will finish the proof. To do so we will use the description of $\mhigc(\cam)$ in Theorem~\ref{main theorem of app B}. It implies that a line bundle on $\mjtgcam^{0}\times\mhigc(\cam)$ is equivalent to a line bundle on $\mjtgcam^{0}\times (\mathbf{P}^{1}\times^{\mathbb{G}_{m}}\mjtgccam)$ plus a gluing data on $\mjtgcam^{0}\times(\mjtgccamn\amalg\mjtgccamn)$. We will prove that the gluing data is nontrivial when we restrict to $\mathbb{G}_{m}\times(\mjtgccamn\amalg\mjtgccamn)$, this will imply the claim. Since $\mjtgcam^{0}\times(\mathbf{P}^{1}\times^{\mathbb{G}_{m}}\mjtgccam)$ is a $\mathbf{P}^{1}$ bundle over $\mjtgcam^{0}\times\mjtgccamn$ by Lemma~\ref{structure of jtors}, the pullback of the Poincar\'e line bundle to $\mjtgcam^{0}\times(\mathbf{P}^{1}\times^{\mathbb{G}_{m}}\mjtgccam)$ is of the form $O(k)\otimes f^{*}(\mathcal{N})$ where we denote the canonical line bundle associated to the projective bundle $\mathbf{P}^{1}\times^{\mathbb{G}_{m}}\mjtgccam\rightarrow\mjtgccamn$ by $O(1)$ and $O(k)$ is its $k$th tensor power, $f$ is the morphism $\mjtgcam^{0}\times(\mathbf{P}^{1}\times^{\mathbb{G}_{m}}\mjtgccam)\xrightarrow{f}\mjtgcam^{0}\times\mjtgccamn$ and $\mathcal{N}$ is a biextension on $\mjtgcam^{0}\times\mjtgccamn$. Because the restriction of the Poincar\'e line bundle to $e\times\mhigc(\cam)$ is trivial where $e$ is the unit element in $\mjtgcam^{0}$, this implies that $k=0$. Moreover, by Lemma~\ref{pullback from normalization}, the restriction of $f^{*}(\mathcal{N})$ to the open set $\mjtgcam^{0}\times\mjtgccam\subseteq\mjtgcam^{0}\times(\mathbf{P}^{1}\times^{\mathbb{G}_{m}}\mjtgccam)$ is isomorphic to the pullback of the biextension $(q\times\textrm{id})^{*}(\mathcal{P}^{n}_{G})$ on $\mjtgcam^{0}\times\mjtgccamn$. Here $q$ stands for $\mjtgcam^{0}\times\mjtgccamn\rightarrow\mjtgcam^{0}\rightarrow\mjtgcamn^{0}$. So we have $\mathcal{N}\simeq (q\times\textrm{id})^{*}\mathcal{P}^{n}_{G}$. Using Lemma~\ref{the gluing map} we conclude that the gluing data amongs to an isomorphism 
$$(\textrm{id}\times t_{Nm(\check{\alpha}_{*}(O(\widetilde{x}_{1})))})^{*}(q\times\textrm{id})^{*}(\mathcal{P}^{n}_{G})\simeq (q\times\textrm{id})^{*}\mathcal{P}^{n}_{G} ,$$
where $t_{Nm(\check{\alpha}_{*}(O(\widetilde{x}_{1})))}$ stands for the automorphism of $\mjtgccamn$ given by translation by the element $Nm(\check{\alpha}_{*}(O(\widetilde{x}_{1})))$, see Lemma~\ref{the gluing map}. Since $\mathcal{P}^{n}_{G}$ is a biextension, one has a canonical isomorphism 
$$(\textrm{id}\times t_{Nm(\check{\alpha}_{*}(O(\widetilde{x}_{1})))})^{*}(q\times\textrm{id})^{*}(\mathcal{P}^{n}_{G})\simeq q^{*}(\mathcal{P}^{n}_{Nm(\check{\alpha}_{*}(O(\widetilde{x}_{1})))})\otimes(q\times\textrm{id})^{*}(\mathcal{P}^{n}_{G})$$
where $\mathcal{P}^{n}_{Nm(\check{\alpha}_{*}(O(\widetilde{x}_{1})))}$ stands for the restriction of $\mathcal{P}^{n}_{G}$ to 
$$\mjtgcamn^{0}\times Nm(\check{\alpha}_{*}(O(\widetilde{x}_{1})))\hookrightarrow\mjtgcamn^{0}\times\mjtgccamn .$$ 
Hence we need a trivialization of $q^{*}(\mathcal{P}^{n}_{Nm(\check{\alpha}_{*}(O(\widetilde{x}_{1})))})$ on $\mjtgcam^{0}$. By Corollary~\ref{description of the fiber}, $\mathcal{P}^{n}_{Nm(\check{\alpha}_{*}(O(\widetilde{x}_{1})))}$ is isomorphic to $\alpha(\mathcal{F}^{n}_{T})\mid_{\widetilde{x}_{1}}$. By the next lemma, the trivialization of the pullback of $\alpha(\mathcal{F}^{n}_{T})\mid_{\widetilde{x}_{1}}$ to $\mjtgcam^{0}$ is provided by Lemma~\ref{structure of jtors}. In particular, when we restrict it to $\mathbb{G}_{m}\subseteq\mjtgcam^{0}$, this trivialization gives nontrivial gluing data. This finishes the proof.

\end{proof}

\begin{lemma}
Consider the morphism $\mathbf{P}^{1}\rightarrow\higc(\cam)$ provided by Theorem~\ref{main theorem of app B}. Then the pullback of $\mathcal{P}_{G}$ to $\mjtgcam^{0}\times\mathbf{P}^{1}$ is canonically trivial. Moreover, if one denotes the restriction of $\mathcal{P}_{G}$ to $\mjtgcam^{0}\times 0$ and $\mjtgcam^{0}\times\infty$ by $\mathcal{P}_{G,0}$ and $\mathcal{P}_{G,\infty}$, respectively, then the trivializations of $\mathcal{P}_{G,0}$ and $\mathcal{P}_{G,\infty}$ induces a trivialization of the pullback of $\alpha(\mathcal{F}^{n}_{T})\mid_{\widetilde{x}_{1}}$ to $\mjtgcam^{0}$ via $\mjtgcam^{0}\xrightarrow{q}\mjtgcamn^{0}$, which coincide with the natural trivialization given in Lemma~\ref{structure of jtors}. 
\end{lemma}
\begin{proof}
Let $\pic(\mathbf{P}^{1})$ be the Picard scheme of $\mathbf{P}^{1}$ parameterizing line bundles together with a trivialization at the point in $\mathbf{P}^{1}$ corresponds to the Kostant section. The restriction of $\mathcal{P}_{G}$ on $\mjtgcam^{0}\times\mathbf{P}^{1}$ induces a group morphism $\mjtgcam^{0}\rightarrow\pic(\mathbf{P}^{1})$. Moreover, the morphism $\mathbf{P}^{1}\rightarrow\higc(\cam)$ is obtained by interpreting $\mathbf{P}^{1}$ as a component of the affine springer fiber of $\check{G}$. So points in $\mathbf{P}^{1}$ corresponds to the trivial $W$ equivariant $T$ bundle on $\cam^{un}$. By our construction of the Poincar\'e line bundle on $\cam^{un}\times X_{*}(T)\times\higc(\cam)$ in Subsection $4.1$, we conclude that the pullback of $\mathcal{P}_{G}$ to $\cam^{un}\times X_{*}(T)\times\mathbf{P}^{1}$ is canonically trivial. Hence the induced morphism $\cam^{un}\times X_{*}(T)\rightarrow\pic(\mathbf{P}^{1})$ sends $\cam^{un}\times X_{*}(T)$ to the unit of $\pic(\mathbf{P}^{1})$. This implies that $\jtgcam^{0}\rightarrow\pic(\mathbf{P}^{1})$ is the trivial one. 

For the second claim, our discussions in Lemma~\ref{normalization} and Corollary~\ref{the gluing map} implies that $0$ and $\infty$ are identified under the $\mathbb{Z}$ action on the affine springer fiber. More precisely, 
in the proof of Corollary~\ref{the gluing map} we have shown that $0$ and $\infty$ are identified by the action of an element $\eta\in\textrm{Gr}_{J_{G,\cam_{x},x}}$ whose image in $\textrm{Gr}_{J_{G,\cam^{n}_{x},x}}$ is given by $Nm(\check{\alpha}(O(\widetilde{x}_{1})))$. Hence from part $(3)$ of Corollary~\ref{additional compatibility}, Lemma~\ref{pullback from normalization}, Corollary~\ref{description of the fiber}, Lemma~\ref{structure of jtors} and  Corollary~\ref{the gluing map}, we conclude that $\mathcal{P}_{G,\infty}\simeq\mathcal{P}_{G,0}\otimes q^{*}(\mathcal{P}^{n}_{Nm(\check{\alpha}_{*}(O(\widetilde{x}_{1})))})$. So the trivializations of $\mathcal{P}_{G,0}$ and $\mathcal{P}_{G,\infty}$ induces a trivialization of $q^{*}(\mathcal{P}^{n}_{Nm(\check{\alpha}_{*}(O(\widetilde{x}_{1})))})$. We will show that this trivialization is the one given by Lemma~\ref{structure of jtors}. By Corollary~\ref{description of the fiber} and Proposition~\ref{propositions of poincare line bundle on bunt}, we see that if we consider the pullback of $q^{*}(\mathcal{P}^{n}_{Nm(\check{\alpha}_{*}(O(\widetilde{x}_{1})))})$ on $\mjtgcam$ to $\cam^{un}\times \phi\subseteq\cam^{un}\times X_{*}(T)$, then its fiber at $(\widetilde{x},\phi)$ is given by 
\begin{equation}
\otimes_{w\in W}O(w(\widetilde{x}_{1}))^{\otimes<w(\phi),\alpha>}\mid_{\widetilde{x}} .
\end{equation}
Here we used the fact that the coroot $\check{\alpha}$ for $\check{G}$ corresponds to the root $\alpha$ of $G$. Since $w(\widetilde{x}_{1})$ are not in $\cam^{un}$, this is canonically trivial. This describes the pullback of the trivialization of $q^{*}(\mathcal{P}^{n}_{Nm(\check{\alpha}_{*}(O(\widetilde{x}_{1})))})$ on $\mjtgcam$ to $\cam^{un}\times X_{*}(T)$.  Using Proposition~\ref{propositions of poincare line bundle on bunt} again, we see that the fiber can also be written as 
\begin{equation}\label{expresstion of the fiber}
\otimes_{w\in W}O(\widetilde{x})^{\otimes<w(\phi),\alpha>}\mid_{w(\widetilde{x}_{1})}
\end{equation}
and the trivialization comes from the fact that $w(\widetilde{x}_{1})$ are not in $\cam^{un}$. It can also be described as follows: Let $s_{\alpha}(\widetilde{x}_{1})=\widetilde{x}_{2}$, see the notations right before Lemma~\ref{structure of jtors}. If we view $O(\widetilde{x})$ as a line bundle on $\cam$, then its pullback to $\cam^{n}$ has the property that its fibers at $w(\widetilde{x}_{1})$ and $ws_{\alpha}(\widetilde{x}_{1})=w(\widetilde{x}_{2})$ are canonically identified since they are mapped to the same point in $\cam$. Since $O(\widetilde{x})^{\otimes<w(\phi),\alpha>}\simeq O(\widetilde{x})^{\otimes-<ws_{\alpha}(\phi),\alpha>}$, we conclude that this induces a canonical trivialization of~\ref{expresstion of the fiber}. Using the discussions in Lemma~\ref{structure of jtors}, it is not hard to see that this trivialization coincide with the pullback of the trivialization of $q^{*}(\mathcal{P}^{n}_{Nm(\check{\alpha}_{*}(O(\widetilde{x}_{1})))})$ on $\mjtgcam$ to $\cam^{un}\times X_{*}(T)$. This finishes the proof.
\end{proof}

Now one can prove:
\begin{corollary}\label{proof in the adjoint case}
Theorem~\ref{main theorem on coho} holds when $G$ is adjoint.
\end{corollary}
\begin{proof}
By Lemma~\ref{a morphism} and Lemma~\ref{Cohen-Macaulayness}, we have a morphism of Cohen-Macaulay sheaves of codimension $d$: $Rp_{1*}(\mathcal{P}_{G})[d]\rightarrow e_{*}(\mathfrak{j})$. We will first claim the support of $Rp_{1*}(\mathcal{P}_{G})[d]$ is equal to the image of the closed embedding $e$. So suppose $\theta$ is a point in the support of $Rp_{1*}(\mathcal{P}_{G})[d]$. Then $H^{d}(\mhigccam,\mathcal{P}_{\theta})\neq 0$ where $\mathcal{P}_{\theta}$ stands for the restriction of $\mathcal{P}_{G}$ to $\theta\times\mhigccam$. By Grothendieck duality and part $(4)$ of Corollary~\ref{additional compatibility}, this implies $H^{0}(\mhigccam,\mathcal{P}_{-\theta})\neq 0$. Notice that since $\theta$ is a point in the neutral component $\mjtgcam^{0}$, $\mathcal{P}_{-\theta}$ is algebraically equivalent to the trivial line bundle. Since $\mhigccam$ is a reduced and connected projective variety by Lemma~\ref{more geometry}, this implies that we must have $\mathcal{P}_{-\theta}\simeq O_{\mhigccam}$. So $\theta$ is a point in the kernel of $\mjtgcam^{0}\rightarrow\pic^{0}(\mhigccam)$, which is an etale group scheme over $H_{int}$ by Lemma~\ref{a key lemma on etaleness}. Moreover, when we restrict to the open subscheme $H^{0}\subseteq H_{int}$, the results in $[ ]$ implies that the support of $Rp_{1*}(\mathcal{P}_{G})[d]$ is equal to $e(H^{0})$. Since the support is etale over $H_{int}$ and that $Rp_{1*}(\mathcal{P}_{G})[d]$ is Cohen-Macaulay of codimension $d$, it implies $e(H^{0})$ must be dense in the support. So $Rp_{1*}(\mathcal{P}_{G})[d]$ is set theoretically supported on $e(H_{int})$. Moreover, Theorem $1.2.1$ of \cite{Geometric Langlands} implies that $Rp_{1*}(\mathcal{P}_{G})[d]$ is scheme-theoretically supported on $e(H^{0})$ when restricted to $H^{0})$, this shows that it is actually scheme-theoretically supported in $e(H_{int})$.

By the results above, to finish the proof we only need to show that the morphism $R^{d}p_{1*}(\mathcal{P}_{G})\rightarrow e_{*}(\mathfrak{j})$ is an isomorphism when we restrict both sides to $e(H_{int})$. This follows from the fact that the restriction of $\mathcal{P}_{G}$ to $e(H_{int})\times_{H_{int}}\mhigc$ is trivial, so that 
$$Rp^{d}_{1*}(\mathcal{P}_{G}\mid_{e(H_{int})\times_{H_{int}}\mhigc})\simeq R^{d}p_{1*}(O_{e(H_{int})\times_{H_{int}}\mhigc})\simeq e_{*}(\mathfrak{j})$$
by Grothendieck duality and Corollary~\ref{Cohen-Macaulayness}. 
\end{proof}

Next we will prove Theorem~\ref{main theorem on coho} for general semisimple groups. Let $G$ be a semisimple algebraic group and $\check{G}$ be its Langlands dual. Let $G_{ad}$ be the adjoint group associated with $G$ and $G_{sc}$ be the simply connected cover of $\check{G}$. One has exact sequences:
\begin{equation}\label{exact sequences}
\begin{gathered}
0\rightarrow Z(G)\rightarrow G\rightarrow G_{ad}\rightarrow 0  \\
0\rightarrow Z(G)^{\check{}}\rightarrow G_{sc}\rightarrow\check{G}\rightarrow 0  
\end{gathered}
\end{equation}
where $Z(G)^{\check{}}$ denotes the Cartier dual of $Z(G)$. We fix a Kostant section for $\higsc$, which will induce a Kostant section for $\higc$.

First we prove a compatibility results for Poincar\'e line bundles:
\begin{proposition}\label{compatibility between different groups}
Consider the diagram:
$$\xymatrix{
\jtgcam\times\higsccam \ar[r] \ar[d] & \jtgcam\times\higc(\cam)\\
\jtgadcam\times\higsccam .
}
$$
Denote the Poincar\'e line bundle on $\jtgadcam\times\higsccam$ by $\mathcal{P}_{G_{ad}}$ and the Poincar\'e line bundle on $\jtgcam\times\higc(\cam)$ by $\mathcal{P}_{G}$, respectively. Then the pullback of $\mathcal{P}_{G_{ad}}$ and $\mathcal{P}_{G}$ to $\jtgcam\times\higsccam$ are isomorphic. 
\end{proposition}

\begin{proof}
We will prove that the biextensions on $\jtgcam\times\jtgsccam$ obtained by pullback of $\mathcal{P}_{G_{ad}}$ and $\mathcal{P}_{G}$ are isomorphic. This is enough since if we work over the universal case, since the complement of $\mathcal{H}iggs^{reg}_{G_{sc}}$ in $\mathcal{H}iggs_{G_{sc}}$ has codimension greater than or equals to two. 

To simplify the notations, we will write $\jtg$ instead of $\jtgcam$ in the rest of the proof. Consider the following commutative diagram:
\[\xymatrixcolsep{1pc}\xymatrix{
\bunt(\cam)\times\textrm{Bun}_{T_{sc}}(\cam) \ar[rd] \ar[rr] \ar[dd] & & \bunt(\cam)\times\buntc(\cam) \ar[d]\\
 & \jtg\times\jtgsc \ar[r] \ar[d] & \jtg\times\jtgc \\
\textrm{Bun}_{T_{ad}}(\cam)\times \textrm{Bun}_{T_{sc}}(\cam) \ar[r] & \jtgad\times\jtgsc .
}
\]
Our construction of the Poincar\'e line bundle implies that the pullback of $\mathcal{P}_{G}$ to $\bunt(\cam)\times\buntc(\cam)$ is canonically isomorphic to $(Nm_{T}\times\textrm{id})^{*}(\mathcal{Q}_{T})$ where $Nm_{T}$ is the morphism $\buntc(\cam)\rightarrow\buntc(\cam)$ given in the proof of part $(1)$ of Lemma~\ref{a key lemma}, $\mathcal{Q}_{T}$ is the Poincar\'e line bundle on $\bunt(\cam)\times\buntc(\cam)$. Similarly, the pullback of $\mathcal{P}_{G_{ad}}$ to $\textrm{Bun}_{T_{ad}}(\cam)\times \textrm{Bun}_{T_{sc}}(\cam)$ is canonically isomorphic to $(Nm_{T_{ad}}\times\textrm{id})^{*}(\mathcal{Q}_{T_{ad}})$. Since the morphisms of the torus $T\rightarrow T_{ad}$ and $T_{sc}\rightarrow\check{T}$ are dual to each other, our discussions about Poincar\'e line bundle in Appendix $A$ implies that the pullbacks of $(Nm_{T}\times\textrm{id})^{*}(\mathcal{Q}_{T})$ and $(Nm_{T_{ad}}\times\textrm{id})^{*}(\mathcal{Q}_{T_{ad}})$ to $\bunt(\cam)\times\textrm{Bun}_{T_{sc}}(\cam)$ are canonically isomorphic. Now we consider the universal case. The constructions in Subsection $3.3$ of \cite{Geometric Langlands} and Theorem $1.2.1$ of \cite{Geometric Langlands} implies that the claim of the lemma is true over the open subscheme $H^{0}\subseteq H_{int}$. So we need to show this morphism of line bundles defined over $H^{0}$ is regular everywhere. Our discussions above shows that it is regular when we pullback to the smooth cover $\bunt(\cam)\times\textrm{Bun}_{T_{sc}}(\cam)$, hence it is regular everywhere.

\end{proof}

We also need the following lemma:
\begin{lemma}\label{connections between groups}
\begin{enumerate1}
\item Let $\higc^{0}(\cam)$ be the connected component of $\higc(\cam)$ containing the Kostant section. Then the restriction of $\mathcal{P}_{G}$ to $\jtgcam\times\higc^{0}(\cam)$ descends to $\mjtgcam\times\higc^{0}(\cam)$. Here we remind the reader that $\mjtgcam$ denotes the moduli space of $\jtgcam$.
\item $\higsc(\cam)$ has a natural action by $B(Z(G)^{\check{}})$ such that if we denote the quotient by $\overline{\higsc}(\cam)$, then $\higsc(\cam)\rightarrow\higc^{0}(\cam)$ factors through $\overline{\higsc}(\cam)$. Moreover, it is an $H^{1}(X,Z(G)^{\check{}})$ torsor on $\higc^{0}(\cam)$.
\item Let $u$ be the projection $\overline{\higsc}(\cam)\xrightarrow{u}\higc^{0}(\cam)$. We have a decomposition $u_{*}(O_{\overline{\higsc}(\cam)})\simeq\oplus_{\chi}\mathcal{M}_{\chi}$ where $\mathcal{M}_{\chi}$ is a line bundle on $\higc^{0}(\cam)$ and $\chi$ ranges over all the characters of the finite group $H^{1}(X,Z(G)^{\check{}})$.
\item Consider $\mathcal{P}_{G}\mid_{H^{1}(X,Z(G))\times\higc^{0}(\cam)}$. There exists a perfect pairing:
$$H^{1}(X,Z(G))\times H^{1}(X,Z(G)^{\check{}})\rightarrow\mathbb{G}_{m} $$
such that for each $\chi\in H^{1}(X,Z(G))$, if we denote the corresponding character of $H^{1}(X,Z(G)^{\check{}})$ still by $\chi$, then we have $\mathcal{P}_{G}\mid_{\chi\times\higc^{0}(\cam)}\simeq\mathcal{M}_{\chi}$. 
Here we recall that $H^{0}(X,J_{G,\cam})\simeq Z(G)$, see part $(5)$ of Proposition~\ref{geo of hig}. Using this one get an embedding  $H^{1}(X,Z(G))\hookrightarrow\mjtgcam$. Hence we can $H^{1}(X,Z(G))\times\higc^{0}(\cam)$ as a closed substack of $\mjtgcam\times\higc^{0}(\cam)$. 
\end{enumerate1} 
\end{lemma}
\begin{proof}
For part $(1)$, since $\jtgadcam$ is a scheme, we see that $Z(G)$ acts trivially on the fibers of the pullback of $\mathcal{P}_{G}$ to $\jtgcam\times\higsc(\cam)$. Also, since $\higsc(\cam)$ is connected by part $(2)$ of Proposition~\ref{geo of hitchin fibration}, the image of $\higsc(\cam)\rightarrow\higc(\cam)$ is equal to $\higc^{0}(\cam)$. Now Proposition~\ref{compatibility between different groups} implies the claim.

For part $(2)$, part $(5)$ of Proposition~\ref{geo of hig} implies that we have an embedding $BZ(G_{sc})$ into $\jtgsccam$ and it induces an action of $BZ(G_{sc})$ on $\higsc(\cam)$. We can restrict it to $BZ(G)^{\check{}}$. Now the exact sequences~\ref{exact sequences} implies that one has exact sequence:
$$0\rightarrow H^{0}(X,Z(G)^{\check{}})\rightarrow H^{0}(X,G_{sc})\rightarrow H^{0}(X,\check{G})\rightarrow $$
$$H^{1}(X,Z(G)^{\check{}})\rightarrow H^{1}(X,G_{sc})\rightarrow H^{1}(X,\check{G}) .$$
This implies the claim.

Part $(3)$ follows immediately from part $(2)$. 

For part $(4)$, let us denote the quotient $\jtgsccam/BZ(G)^{\check{}}$ by $\overline{\mathcal{T}ors}(J_{G_{sc},\cam})$. The $\jtgsccam$ action on $\higsc(\cam)$ induces an action of $\overline{\mathcal{T}ors}(J_{G_{sc},\cam})$ on $\overline{\higsc}(\cam)$. Moreover, the exact sequences~\ref{exact sequences} induce the following exact sequence of sheaves on $X$:
$$0\rightarrow Z(G)^{\check{}}\rightarrow J_{G_{sc},\cam}\rightarrow J_{\check{G},\cam}\rightarrow 0.$$
From this we conclude from part $(5)$ of Proposition~\ref{geo of hig} that $H^{1}(X,Z(G)^{\check{}})$ is isomorphic to the kernel of the morphism $\overline{\mathcal{T}ors}(J_{G_{sc},\cam})\rightarrow\jtgccam$. Now notice that part $(1)$ implies that the restriction of $\mathcal{P}_{G}$ to $\jtgcam\times\higc^{0}(\cam)$ actually lives on $\mjtgcam\times\higc^{0}(\cam)$. By part $(2)$, $\higsc(\cam)\rightarrow\higc^{0}(\cam)$ factors through $\overline{\higsc}(\cam)$. Since $H^{1}(X,Z(G)^{\check{}})$ is the kernel of $\overline{\mathcal{T}ors}(J_{G_{sc},\cam})\rightarrow\jtgccam$, the action of $\overline{\mathcal{T}ors}(J_{G_{sc},\cam})$ on $\overline{\higsc}(\cam)$ induces an action of $H^{1}(X,Z(G)^{\check{}})$ on the pullback of $\mathcal{P}_{G}$ to $H^{1}(X,Z(G))\times\overline{\higsc}(\cam)$ via:
$$H^{1}(X,Z(G))\times\overline{\higsc}(\cam)\xrightarrow{i}\mjtgcam\times\overline{\higsc}(\cam)\xrightarrow{\textrm{id}\times u} $$
$$\mjtgcam\times\higc^{0}(\cam) .$$
Since $H^{1}(X,Z(G))$ lies in the kernel of $\mjtgcam\rightarrow\jtgadcam$, Proposition~\ref{compatibility between different groups} implies that $i^{*}(\textrm{id}\times u)^{*}(\mathcal{P}_{G})$ is canonically trivial. So the $H^{1}(X,Z(G)^{\check{}})$ action on $i^{*}(\textrm{id}\times u)^{*}(\mathcal{P}_{G})$ is induced by a pairing:
$$H^{1}(X,Z(G))\times H^{1}(X,Z(G)^{\check{}})\rightarrow\mathbb{G}_{m} .$$
The same argument shows that we have a similar pairing in the universal case as a pairing between constant group schemes over $H_{int}$. The results in $[ ]$ implies that in the universal case, this pairing is nondegenerate over the open subscheme $H^{0}$. Hence it is nondegenerate everywhere. For any element $\chi\in H^{1}(X,Z(G))$, let us denote the corresponding character of $H^{1}(X,Z(G)^{\check{}})$ still by $\chi$. To summarize, we have shown that the restriction of $i^{*}(\textrm{id}\times u)^{*}(\mathcal{P}_{G})$ to $\chi\times\overline{\higsc}(\cam)$ is isomorphic to $O_{\overline{\higsc}(\cam)}$, but their $H^{1}(X,Z(G)^{\check{}})$ actions differ by $\chi$. The claim follows from this. 

\end{proof}

\begin{lemma}\label{almost there}
Consider the restriction of the Poincar\'e line bundle on $\mjtgcam\times\higc^{0}(\cam)$ (Part $(1)$ of Lemma~\ref{connections between groups} implies that the Poincar\'e line bundle lives on $\mjtgcam\times\higc^{0}(\cam)$). Then we have $Rp_{1*}\mathcal{P}_{G}\simeq e_{*}(k)[-d]$, where $p_{1}$ denotes $\mjtgcam\times\higc^{0}(\cam)\xrightarrow{p_{1}}\mjtgcam$.
\end{lemma}
\begin{proof}
Denote $\overline{\higsc}(\cam)\rightarrow\higc^{0}(\cam)$ by $u$. Combining Corollary~\ref{proof in the adjoint case} and Proposition~\ref{compatibility between different groups} together, we conclude that we have an isomorphism 
$$Rp'_{1*}(\textrm{id}\times u)^{*}(\mathcal{P}_{G})\simeq i_{*}(O_{H^{1}(X,Z(G))})[-d]$$
where $p'_{1}$ denotes the projection $\mjtgcam\times\overline{\higsc}(\cam)\xrightarrow{p'_{1}}\mjtgcam$ and $i$ is the embedding $H^{1}(X,Z(G))\hookrightarrow\mjtgcam$. This implies that 
$$Rp_{1*}(\textrm{id}\times u)_{*}(\textrm{id}\times u)^{*}(\mathcal{P}_{G})\simeq i_{*}(O_{H^{1}(X,Z(G))})[-d] .$$
We have 
$$(\textrm{id}\times u)_{*}(\textrm{id}\times u)^{*}(\mathcal{P}_{G})\simeq\oplus_{\chi}\mathcal{P}_{G}\otimes\mathcal{M}_{\chi}, $$
see the notations of Part $(3)$ of Lemma~\ref{connections between groups}. Part $(4)$ of Lemma~\ref{connections between groups} and part $(2)$ of Corollary~\ref{additional compatibility} implies that $\mathcal{P}_{G}\otimes\mathcal{M}_{\chi}\simeq (t_{\chi}\times\textrm{id})^{*}(\mathcal{P}_{G})$ where $t_{\chi}$ stands for the automorphism of $\mjtgcam$ given by translation by $\chi$. Combining the discussions above together, we get:
$\oplus_{\chi} t_{\chi}^{*}(Rp_{1*}\mathcal{P}_{G})\simeq i_{*}(O_{H^{1}(X,Z(G))})[-d]$ where $\chi\in H^{1}(X,Z(G))$. Now it remains to show that the image of $e$ lies in the support of $Rp_{1*}\mathcal{P}_{G}$. This follows from the universal case, since we know the claim is true over $H^{0}$ by Theorem $1.2.1$ of \cite{Geometric Langlands}, so the image of $e$ must be in the support. 

\end{proof}

Now we shall prove Theorem~\ref{main theorem on coho} for all semisimple groups:
\begin{proof}
Let $\gamma\in \pi_{1}(\check{G})$. We shall identify it with an element in $Z(G)^{\check{}}$ via exact sequences~\ref{exact sequences}. Let $\mathcal{P}^{\gamma}_{G}$ be the restriction of $\mathcal{P}_{G}$ to $\jtgcam\times\higc^{\gamma}(\cam)$ where $\higc^{\gamma}(\cam)$ is the connected component of $\higc(\cam)$ determined by $\gamma$, see part $(2)$ of Proposition~\ref{geo of hitchin fibration}. Lemma~\ref{almost there} implies that $Rp_{1*}\mathcal{P}^{0}_{G}\simeq a_{*}(O_{BZ(G)})[-d]$ where we denote the embedding $BZ(G)\hookrightarrow\jtgcam$ by $a$. It remains to determine $Rp_{1*}\mathcal{P}^{\gamma}_{G}$ for $\gamma\neq 0$. Let us fix $\widetilde{x}\in \cam^{un}$ and $\phi\in X_{*}(\check{T})$ such that the image of $\phi$ in $\pi_{1}(\check{G})$ is equal to $\gamma$, see Lemma~\ref{surjection on pio}. Let $\xi$ be the image of $(\widetilde{x},\phi)$ under the Abel-Jacobi map $\cam^{un}\times X_{*}(\check{T})\rightarrow \jtgccam$. Then translation by $\xi$ induces an isomorphism $\higc^{0}(\cam)\xrightarrow{t_{\xi}}\higc^{\gamma}(\cam)$. Part $(3)$ of Corollary~\ref{additional compatibility} implies that $(\textrm{id}\times t_{\xi})^{*}(\mathcal{P}^{\gamma}_{G})\simeq \mathcal{P}^{0}_{G}\otimes p_{1}^{*}(\mathcal{P}_{G,\xi})$ where $\mathcal{P}_{G,\xi}$ stands for the restriction to $\jtgcam\times\xi$ of the biextension on $\jtgcam\times\jtgccam$ corresponds to $\mathcal{P}_{G}$. This implies that $Rp_{1*}\mathcal{P}^{\gamma}_{G}\simeq Rp_{1*}(\mathcal{P}^{0}_{G})\otimes\mathcal{P}_{G,\xi}$. Since $Rp_{1*}\mathcal{P}^{0}_{G}\simeq a_{*}(O_{BZ(G)})[-d]$ by Lemma~\ref{almost there}, it remains to determine $a^{*}(\mathcal{P}_{G,\xi})$ where we remind the reader that $a$ is the embedding $BZ(G)\rightarrow\jtgcam$ at the unit element. Using our construction of the Poincar\'e line bundle in Subsection $4.1$ it is not hard to see that if we pullback $a^{*}(\mathcal{P}_{G,\xi})$ along the natural morphism $\spec(k)\rightarrow BZ(G)$, then the $Z(G)$ action on the fiber is given by the character of $Z(G)$ determined by the image of $\phi$ in $\pi_{1}(\check{G})\simeq Z(G)^{\check{}}$. This shows that $Rp_{1*}\mathcal{P}^{\gamma}_{G}\simeq a_{*}\mathcal{L}_{\gamma}[-d]$ where $\mathcal{L}_{\gamma}$ is the line bundle on $BZ(G)$ determined by the character $\gamma\in Z(G)^{\check{}}$. Hence we have $Rp_{1*}\mathcal{P}_{G}\simeq\oplus_{\gamma}a_{*}\mathcal{L}_{\gamma}[-d]\simeq e_{*}(k)[-d]$.

\end{proof}

Now we shall indicate how to extend Theorem~\ref{main theorem on coho} to all reductive groups. We will prove:
\begin{theorem}\label{2nd main theorem of cohomology of poincare line bundle}
Let $G$ be reductive. Let $z$ be the dimension of the center of $G$ and $d$ be the dimension of $\mathcal{H}iggs_{(G/Z(G))}(\cam)$. Let $\mathcal{P}_{G}$ be the Poincar\'e line bundle on $\jtgcam\times\higc(\cam)$ and $p_{1}$ be the projection to $\jtgcam$. Then we have:
$$Rp_{1*}\mathcal{P}_{G}\simeq e_{*}(k)[-d-zg]$$
where $g$ is the genus of $X$. 
\end{theorem}

We will prove Theorem~\ref{2nd main theorem of cohomology of poincare line bundle} by reducing it to the semisimple case. First, since $G$ is reductive, one can choose a torus $S$, a simply connected group $G'_{sc}$ and its dual $G'_{ad}$ so that we have exact sequences:
\begin{equation}\label{exact sequences for reductive groups}
\begin{gathered}
0\rightarrow K\rightarrow G\rightarrow G'_{ad}\times S\rightarrow 0 \\
0\rightarrow \check{K}\rightarrow G'_{sc}\times\check{S}\rightarrow \check{G}\rightarrow 0 
\end{gathered}
\end{equation}
where $K$ is a finite group and $\check{K}$ is its dual. In this case one have natural isomorphisms: 
$$\mathcal{T}ors(J_{G'_{ad}\times S,\cam})\simeq\mathcal{T}ors(J_{G'_{ad},\cam})\times\textrm{Bun}_{S}(X)$$
$$\mathcal{H}iggs_{G'_{ad}\times S}(\cam)\simeq\mathcal{H}iggs_{G'_{ad}}(\cam)\times\textrm{Bun}_{S}(X) $$
$$\mathcal{T}ors(J_{G'_{sc}\times \check{S},\cam})\simeq\mathcal{T}ors(J_{G'_{sc},\cam})\times\textrm{Bun}_{\check{S}}(X)$$
$$\mathcal{H}iggs_{G'_{sc}\times \check{S}}(\cam)\simeq\mathcal{H}iggs_{G'_{sc}}(\cam)\times\textrm{Bun}_{\check{S}}(X) .$$

We will make these identifications freely in the rest of this section. The dimension of $S$ is equal to the dimension of $Z(G)$, which is $z$. Notice that under these isomorphisms, the Poincar\'e line bundle on $\mathcal{T}ors(J_{G'_{ad}\times S,\cam})\times\mathcal{H}iggs_{G'_{sc}\times \check{S}}(\cam)$ is given by the tensor product of the pullback of the Poincar\'e line bundle on $\mathcal{T}ors(J_{G'_{ad},\cam})\times\mathcal{H}iggs_{G'_{ad}}(\cam)$ with the pullback of the Poincar\'e line bundle on $\textrm{Bun}_{S}(X)\times\textrm{Bun}_{\check{S}}(X)$. We fix a Kostant section for $\higsc(\cam)$. This plus the trivial $\check{S}$ bundle on $X$ induces a 
Kostant section of $\higc(\cam)$ via $\mathcal{H}iggs_{G'_{sc}\times \check{S}}(\cam)\rightarrow\higc(\cam)$

The analogue of Proposition~\ref{compatibility between different groups} still holds, with the same proof:
\begin{proposition}\label{compatibility for reductive groups}
Consider the diagram:
$$\xymatrix{
\jtgcam\times\mathcal{H}iggs_{G'_{sc}\times \check{S}}(\cam) \ar[r] \ar[d] & \jtgcam\times\higc(\cam)\\
\mathcal{T}ors(J_{G'_{ad}\times S,\cam})\times\mathcal{H}iggs_{G'_{sc}\times \check{S}}(\cam) .
}
$$
\end{proposition}

Then the pullback of the Poincar\'e line bundle $\mathcal{P}_{G}$ to $\jtgcam\times\mathcal{H}iggs_{G'_{sc}\times \check{S}}(\cam)$ is isomorphic to the pullback of the Poincar\'e line bundle $\mathcal{P}_{G'_{ad}\times S}$ to $\jtgcam\times\mathcal{H}iggs_{G'_{sc}\times \check{S}}(\cam)$.

The following lemma is the analogue of Lemma~\ref{connections between groups}:
\begin{lemma}\label{connections for different reductive groups}
\begin{enumerate1}
\item Let $\higc^{0}(\cam)$ be the connected component of $\higc(\cam)$ containing the Kostant section. Then the
restriction of $\mathcal{P}_{G}$ to $\jtgcam\times\higc^{0}(\cam)$ descends to $\mjtgcam\times\higc^{0}(\cam)$. Here we remind the reader that $\mjtgcam$ denotes the moduli space of $\jtgcam$.
\item Let $\mathcal{H}iggs^{0}_{G'_{sc}\times \check{S}}(\cam)$ be the connected component of $\mathcal{H}iggs_{G'_{sc}\times \check{S}}(\cam)$ containing the Kostant section (It is isomorphic to $\higsc(\cam)\times\textrm{Bun}^{0}_{\check{S}}(X)$ where $\textrm{Bun}^{0}_{\check{S}}(X)$ is the neutral component of $\textrm{Bun}^{0}_{\check{S}}(X)$). Then it has a natural action by $B\check{K}$ such that if we denote the quotient by $\overline{\mathcal{H}iggs^{0}}_{G'_{sc}\times \check{S}}(\cam)$, then $\mathcal{H}iggs^{0}_{G'_{sc}\times \check{S}}(\cam)\rightarrow\higc^{0}(\cam)$ factors through $\overline{\mathcal{H}iggs^{0}}_{G'_{sc}\times \check{S}}(\cam)$. Moreover, it is an $H^{1}(X,\check{K})$ torsor over $\higc^{0}(\cam)$. 
\item Let $u$ be the morphism $\overline{\mathcal{H}iggs^{0}}_{G'_{sc}\times \check{S}}(\cam)\xrightarrow{u}\higc^{0}(\cam)$. Then we have 
$$u_{*}(O_{\overline{\mathcal{H}iggs^{0}}_{G'_{sc}\times \check{S}}(\cam)})\simeq\oplus_{\chi}\mathcal{M}_{\chi}$$
where $\mathcal{M}_{\chi}$ is a line bundle on $\higc^{0}(\cam)$ and $\chi$ ranges over all characters of the finite group $H^{1}(X,\check{K})$. 
\item By part $(1)$ $\mathcal{P}_{G}$ lives on $\mjtgcam\times\higc^{0}(\cam)$ and $H^{1}(X,K)$ is naturally a closed subgroup of $\mjtgcam$. Consider $\mathcal{P}_{G}\mid_{H^{1}(X,K)}\times\higc^{0}(\cam)$. There exists a perfect pairing:
$$H^{1}(X,K)\times H^{1}(X,\check{K})\rightarrow\mathbb{G}_{m} $$
such that for each $\chi\in H^{1}(X,K)$, if we denote the corresponding character of $H^{1}(X,\check{K})$, still by $\chi$, then we have $\mathcal{P}_{G}\mid_{\chi\times\higc^{0}(\cam)}\simeq\mathcal{M}_{\chi}$. 

\end{enumerate1}
\end{lemma}
\begin{proof}
Part $(1)$ follows directly from part $(1)$ of Lemma~\ref{connections between groups} combined with Proposition~\ref{compatibility for reductive groups} as well as the fact that the restriction of the Poincar\'e line bundle on $\textrm{Bun}_{S}(X)\times\textrm{Bun}_{\check{S}}(X)$ to $\textrm{Bun}_{S}(X)\times\textrm{Bun}^{0}_{\check{S}}(X)$ descends to $X_{*}(S)\otimes\pic(X)\times\textrm{Bun}^{0}_{\check{S}}(X)$, see our discussions in Appendix $A$. 

The proof of part $(2)$ is also completely analogous to the proof of part $(2)$ of Lemma~\ref{connections between groups}.

Part $(3)$ follows directly from part $(2)$.

Part $(4)$ also follows the same way as part $(4)$ of Lemma~\ref{connections between groups}. One simply replace $\overline{\mathcal{T}ors}(J_{G_{sc},\cam})$ by $\mathcal{T}ors(J_{G_{sc}\times\check{S}})/B\check{K}$. 
\end{proof}

\begin{lemma}\label{almost there for reductive groups}
Consider the Poincar\'e line bundle $\mathcal{P}_{G}$ on $\mjtgcam\times\higc^{0}(\cam)$ (Part $(1)$ of Lemma~\ref{connections for different reductive groups} ensures that is lives on $\mjtgcam\times\higc^{0}(\cam)$). Then we have $Rp_{1*}\mathcal{P}_{G}\simeq e_{*}(k)[-d-zg]$, where $p_{1}$ denotes the projection $\mjtgcam\times\higc^{0}(\cam)\rightarrow\mjtgcam$. 
\end{lemma}
\begin{proof}
The argument is the same as the proof of Lemma~\ref{almost there}. One only needs the additional fact that if we consider the Poincar\'e line bundle $\mathcal{P}_{S}$ on $\textrm{Bun}_{S}(X)\times\textrm{Bun}^{0}_{\check{S}}(X)$, then it lives on $X_{*}(S)\otimes\pic(X)\times\textrm{Bun}^{0}_{\check{S}}(X)$. And if we consider $X_{*}(S)\otimes\pic(X)\times\textrm{Bun}^{0}_{\check{S}}(X)\xrightarrow{p_{1}}X_{*}(S)\otimes\pic(X)$, then one has $Rp_{1*}\mathcal{P}_{S}\simeq e_{*}(k)[-zg]$, see Appendix $A$.

\end{proof}

Now one can finish the proof of Theorem~\ref{2nd main theorem of cohomology of poincare line bundle}:

\begin{proof}
The argument is similar to the proof of Theorem~\ref{main theorem on coho}. Let $\gamma\in\pi_{1}(\check{G})$, denote the corresponding connected component of $\higc(\cam)$ by $\higc^{\gamma}(\cam)$. Denote the restriction of $\mathcal{P}_{G}$ to $\jtgcam\times\higc^{\gamma}(\cam)$ by $\mathcal{P}^{\gamma}_{G}$. Lemma~\ref{almost there for reductive groups} implies that we have $Rp_{1*}\mathcal{P}^{0}_{G}\simeq a_{*}(O_{BZ(G)})[-d-zg]$ where $a$ is the embedding $BZ(G)\xrightarrow{a}\jtgcam$. Now fix $\widetilde{x}\in\cam^{un}$ and $\phi\in X_{*}(\check{T})$ such that the image of $\phi$ in $\pi_{1}(\check{G})$ is equal to $\gamma$. Let $\xi$ be the image of $(\widetilde{x},\phi)$ under the Abel-Jacobi map $\cam^{un}\times X_{*}(\check{T})\rightarrow \jtgccam$. Then translation by $\xi$ induces an isomorphism $\higc^{0}(\cam)\xrightarrow{t_{\xi}}\higc^{\gamma}(\cam)$. Part $(3)$ of Corollary~\ref{additional compatibility} implies that $(\textrm{id}\times t_{\xi})^{*}(\mathcal{P}^{\gamma}_{G})\simeq \mathcal{P}^{0}_{G}\otimes p_{1}^{*}(\mathcal{P}_{G,\xi})$ where $\mathcal{P}_{G,\xi}$ stands for the restriction to $\jtgcam\times\xi$ of the biextension on $\jtgcam\times\jtgccam$ corresponds to $\mathcal{P}_{G}$. This implies that $Rp_{1*}\mathcal{P}^{\gamma}_{G}\simeq Rp_{1*}(\mathcal{P}^{0}_{G})\otimes\mathcal{P}_{G,\xi}$. Since $Rp_{1*}\mathcal{P}^{0}_{G}\simeq a_{*}(O_{BZ(G)})[-d-zg]$ by Lemma~\ref{almost there for reductive groups}, it remains to determine $a^{*}(\mathcal{P}_{G,\xi})$ where we remind the reader that $a$ is the embedding $BZ(G)\rightarrow\jtgcam$ at the unit element. Using our construction of the Poincar\'e line bundle in Subsection $4.1$ it is not hard to see that if we pullback $a^{*}(\mathcal{P}_{G,\xi})$ along the natural morphism $\spec(k)\rightarrow BZ(G)$, then the $Z(G)$ action on the fiber is given by the character of $Z(G)$ determined by the image of $\phi$ in $\pi_{1}(\check{G})\simeq Z(G)^{\check{}}$. This shows that $Rp_{1*}\mathcal{P}^{\gamma}_{G}\simeq a_{*}\mathcal{L}_{\gamma}[-d-zg]$ where $\mathcal{L}_{\gamma}$ is the line bundle on $BZ(G)$ determined by the character $\gamma\in Z(G)^{\check{}}$. Hence we have $Rp_{1*}\mathcal{P}_{G}\simeq\oplus_{\gamma}a_{*}\mathcal{L}_{\gamma}[-d-zg]\simeq e_{*}(k)[-d-zg]$. 

\end{proof}

\begin{corollary}\label{fully faithful fourier-mukai}
Keep the same assumptions as in Theorem~\ref{2nd main theorem of cohomology of poincare line bundle}. Consider the Fourier-Mukai functor 
$$\quasicoh(\jtgcam)\xrightarrow{F_{\mathcal{P}}}\quasicoh(\higc(\cam))$$
induced by $\mathcal{P}_{G}$. Then $F_{\mathcal{P}}$ is fully faithful.
\end{corollary}

\begin{proof}
Let us consider the functor 
$$\quasicoh(\higc(\cam))\xrightarrow{G_{\check{\mathcal{P}}}}\quasicoh(\jtgcam)$$
induced by $\check{\mathcal{P}_{G}}[d+zg]$. We will prove the composition $G_{\check{\mathcal{P}}}\circ F_{\mathcal{P}}\simeq\textrm{id}$. Indeed, the composition $G_{\check{\mathcal{P}}}\circ F_{\mathcal{P}}$ is induced by:
$$Rp_{13*}(p^{*}_{12}(\mathcal{P}_{G})\otimes p^{*}_{23}(\check{\mathcal{P}}_{G}))[d+zg]$$
where $p_{ij}$ denotes the projection of $\jtgcam\times\higc(\cam)\times\jtgcam$ to the $i,j$ components. We need to prove:
$$Rp_{13*}(p^{*}_{12}(\mathcal{P}_{G})\otimes p^{*}_{23}(\check{\mathcal{P}}_{G}))[d+zg]\simeq \Delta_{*}(O_{\jtgcam}) .$$
By part $(2)$ and part $(4)$ of Corollary~\ref{additional compatibility}, if we set $m'$ to be the morphism:
\begin{gather}
\jtgcam\times\jtgcam \rightarrow\jtgcam \notag \\
(\alpha,\beta)\mapsto\alpha-\beta , \notag
\end{gather}
then we have that $p_{12}^{*}(\mathcal{P}_{G})\otimes p_{23}^{*}(\check{\mathcal{P}}_{G})[d+zg]\simeq (m'\times\textrm{id})^{*}(\mathcal{P}_{G})$ under the morphism 
$$\jtgcam\times\higc(\cam)\times\jtgcam\xrightarrow{m'\times\textrm{id}}\jtgcam\times\higc(\cam) .$$
The claim now follows from Theorem~\ref{2nd main theorem of cohomology of poincare line bundle} by base change.

\end{proof}

\appendix

\section{Review about Poincar\'e line bundles on $\bunt(C)\times\buntc(C)$}

In this section we gather some facts about the Poincar\'e line bundle on $\bunt(C)\times\buntc(C)$ where $C$ is an integral projective curve. The main references are \cite{Poincare line bundle}, \cite{Fundamental lemma} and \cite{AF}. 

First let us look at the Poincar\'e line bundle on $\stackypic(C)\times\stackypic(C)$ where $\stackypic(C)$ is the Picard stack of $C$. 
\begin{proposition}
\begin{enumerate1}
\item There exists a Poincar\'e line bundle $\mathcal{P}$ on $\stackypic(C)\times\stackypic(C)$ such that the fiber of $\mathcal{P}$ at the point $(L_{1},L_{2})\in\stackypic(C)\times\stackypic(C)$ is given by:
$$\mathcal{P}=\de(R\Gamma(C,L_{1}\otimes L_{2}))\otimes\de(R\Gamma(C,O_{C}))\otimes\de(R\Gamma(C,L_{1}))^{-1}\otimes\de(R\Gamma(C,L_{2}))^{-1}$$
where $L_{1}$ and $L_{2}$ are line bundles on $C$. Moreover, $\mathcal{P}$ is canonically a biextension on $\stackypic(C)\times\stackypic(C)$. 
\item If $C^{n}$ is the normalization of $C$, then pullback of line bundles induces $\stackypic(C)\rightarrow\stackypic(C^{n})$. Moreover, the Poincar\'e line bundle on $\stackypic(C)\times\stackypic(C)$ is isomorphic to the pullback of the Poincar\'e line bundle on $\stackypic(C^{n})\times\stackypic(C^{n})$ under the morphism $\stackypic(C)\rightarrow\stackypic(C^{n})$. 
\item Let $\stackypic^{0}(C)$ be the neutral component of $\stackypic(C)$ and $\pic(C)$ be the Picard scheme of $C$. Then restriction of $\mathcal{P}$ to $\stackypic^{0}(C)\times\stackypic(C)$ descends to $\stackypic^{0}(C)\times\pic(C)$. 
\item Let $\textrm{Div}^{n}$ be the scheme parameterizing degree $n$ divisors on $C$. Consider the natural morphism $\textrm{Div}^{n}\xrightarrow{AJ}\stackypic(C)$ given by $D\mapsto O_{C}(D)$. If we pullback $\mathcal{P}$ along $\textrm{Div}^{n}\times\stackypic(C)\rightarrow\stackypic(C)\times\stackypic(C)$, then the fiber of $(AJ\times\textrm{id})^{*}(\mathcal{P})$ at $(D,L)$ is given by $\de(L_{D})\otimes\de(O_{D})^{-1}$ where $L_{D}$ stands for the restriction of $L$ to $D$. 
\item If $C$ is smooth of genus $g$, then we have $Rp_{1*}\mathcal{P}\simeq e_{*}(k)[-g]$ where $e$ is the unit $\spec(k)\xrightarrow{e}\stackypic(C)$. 
\end{enumerate1}
\end{proposition}

Now we look at the Poincar\'e line bundle on $\bunt(C)\times\buntc(C)$. When $\phi\in X_{*}(T)$ and $L$ is a line bundle on $C$, we will use the notation $\phi_{*}(L)$ to denote the $T$ bundle induced from $L$ and the group homomorphism $\mathbb{G}_{m}\xrightarrow{\phi} T$. Also, if $\mathcal{F}_{\check{T}}$ is a $\check{T}$ bundle, then $\phi(\mathcal{F}_{\check{T}})$ denotes the line bundle induced from $\mathcal{F}_{\check{T}}$ and the morphism $\check{T}\xrightarrow{\phi}\mathbb{G}_{m}$. 
\begin{proposition}\label{propositions of poincare line bundle on bunt}
\begin{enumerate1}
\item There exists a Poincar\'e line bundle $\mathcal{P}_{T}$ on $\bunt(C)\times\buntc(C)$ with the following property: Let $L_{1}$ and $L_{2}$ be line bundles and $\phi\in X_{*}(T)$, $\varphi\in X_{*}(\check{T})$, then the fiber of $\mathcal{P}_{T}$ at $(\phi_{*}(L_{1}),\varphi_{*}(L_{2}))$ is given by $\mathcal{P}^{\otimes <\phi,\varphi>}_{L_{1},L_{2}}$ where $\mathcal{P}_{L_{1},L_{2}}$ is the fiber of the Poincar\'e line bundle on $\stackypic(C)\times\stackypic(C)$ at $(L_{1}.L_{2})$. $\mathcal{P}_{T}$ is naturally a biextension on $\bunt(C)\times\buntc(C)$. 
\item Let $C^{n}$ be the normalization of $C$. Consider the morphisms $\bunt(C)\rightarrow\bunt(C^{n})$ and $\buntc(C)\rightarrow\buntc(C^{n})$ induced by pullback. Then the Poincar\'e line bundle on $\bunt(C)\times\buntc(C)$ is isomorphic to the pullback of the Poincar\'e line bundle on $\bunt(C^{n})\times\buntc(C^{n})$ via the morphisms above. 
\item Let $\bunt^{0}(C)$ be the neutral component of $\bunt(C)$ Then the restriction of $\mathcal{P}_{T}$ to $\bunt^{0}(C)\times\buntc(C)$ descends to $\bunt^{0}(C)\times X_{*}(\check{T})\otimes\pic(C)$. 
\item Let $T_{1}\rightarrow T_{2}$ be a group homomorphism and $\check{T}_{2}\rightarrow\check{T}_{1}$ be the induced morphism on the dual torus. Consider the diagram:
$$\xymatrix{
\textrm{Bun}_{T_{1}}(C)\times\textrm{Bun}_{\check{T}_{2}}(C) \ar[r] \ar[d] & \textrm{Bun}_{T_{1}}(C)\times\textrm{Bun}_{\check{T}_{1}}(C) \\
\textrm{Bun}_{T_{2}}(C)\times\textrm{Bun}_{\check{T}_{2}}(C) .
}
$$
Then the pullback of $\mathcal{P}_{T_{1}}$ to $\textrm{Bun}_{T_{1}}(C)\times\textrm{Bun}_{\check{T}_{2}}(C)$ is isomorphic to the pullback of $\mathcal{P}_{T_{2}}$ to $\textrm{Bun}_{T_{1}}(C)\times\textrm{Bun}_{\check{T}_{2}}(C)$. 
\item Consider the morphism $\textrm{Div}^{n}(C)\times X_{*}(T)\xrightarrow{AJ}\bunt(C)$ given by $(D,\phi)\mapsto \phi_{*}(O_{C}(D))$. Then the fiber of $(AJ\times\textrm{id})^{*}(\mathcal{P}_{T})$ at $((D,\phi),\mathcal{F}_{\check{T}})$ is given by $\de(\phi(\mathcal{F}_{\check{T}})\mid_{D})\otimes\de(O_{D})^{-1}$. 
\item If $C$ is smooth of dimension $g$, then we have $Rp_{1*}\mathcal{P}_{T}\simeq e_{*}(k)[-rg]$ where $r$ is the dimension of $T$.
\end{enumerate1}
\end{proposition}

\section{Detailed analysis of $\mathcal{H}iggs_{G}(\widetilde{X})$ when $\widetilde{X}$ has nodal singularities}
In this section we give a detailed analysis about the structure of $\hig(\cam)$ when $\cam$ has nodal singularities over a unique point of $X$. Here are the precise assumptions on $\cam$:
\begin{Assumption}
\begin{enumerate1}
\item $\cam$ is integral. 
\item For any pair of roots $\alpha$ and $\beta$ such that $\alpha\neq\pm\beta$, the ramification divisors $D_{\alpha}$ dose not intersect with $D_{\beta}$
\item There exists a unique point $x\in X$ such that $\cam$ is smooth over the preimage of $X-\{x\}$. Moreover, $\cam$ has nodal singularities over the preimage of $x$. 
\end{enumerate1}
\end{Assumption}
One can also reformulate it as follows: 
\begin{Assumption}
$\cam$ is integral and it is smooth away from $x$. Moreover, if we restrict $\cam$ to the formal disk $\spec(\widehat{O}_{x})$ and look at the induced morphism $\spec(\widehat{O}_{x})\rightarrow\mathfrak{c}$, then the image of $x$ lies in the smooth part of the discriminant divisor in $\mathfrak{c}$ and that the image of the formal curve $\spec(\widehat{O}_{x})$ intersects the discriminant divisor with multiplicity equals to two. 
\end{Assumption}

Let us denote the normalization of $\cam$ by $\cam^{n}$. Notice that in this case $\cam^{n}$ is also an abstract cameral cover over $X$ in the sense of \cite{Gerbe of Higgs}. Hence one can define the group scheme $J_{G,\cam^{n}}$ associated to the cameral cover $\cam^{n}$. One has a natural morphism $J_{G,\cam}\rightarrow J_{G,\cam^{n}}$ and hence a morphism $\jtgcam\rightarrow\jtgcamn$. Notice that since $\cam^{n}$ is smooth, the moduli space of $\jtgcamn$ is projective. Let us denote the moduli spaces of $\jtgcam$ and $\jtgcamn$ by $\mjtgcam$ and $\mjtgcamn$. $\jtgcam$ and $\jtgcamn$ are  $Z(G)$ gerbes over their moduli spaces.

The main result of this section is the following:
\begin{theorem}\label{main theorem of app B}
Assuming $G$ is simply connected. There exists a pushout diagram:
$$\xymatrix{
\mjtgcamn\amalg\mjtgcamn \ar[d] \ar[r] & \mjtgcamn \ar[d]\\
\mathbf{P}^{1}\times^{\mathbb{G}_{m}}\mjtgcam \ar[r] & \mhig(\cam) .
}
$$
Let us remind the reader that $\mjtgcam$ and $\mhigc(\cam)$ stands for the moduli space of $\jtgcam$ and $\higc(\cam)$. 
\end{theorem}

We will prove the theorem in several steps. To begin with, we need to determine the structure of $\jtgcam$. Let $\widetilde{x}$ be a point in $\cam$ in the preimage of $x\in X$ and let $\widetilde{x}_{1}$ and $\widetilde{x}_{2}$ be the two points in $\cam^{n}$ above $\widetilde{x}$. Assume $\widetilde{x}$ corresponds to the ramification divisor $D_{\alpha}$. One has $s_{\alpha}(\widetilde{x}_{1})=\widetilde{x}_{2}$. By part $(1)$ of Proposition~\ref{prop of jtors} we have a natural morphism $\jtgcamn\rightarrow\bunt(\cam^{n})$. Denote the universal $T$-bundle on $\cam^{n}\times\bunt(\cam^{n})$ by $\mathcal{F}^{n}_{T}$.
\begin{lemma}\label{structure of jtors}
$\jtgcam$ is a $\mathbb{G}_{m}$ torsor over $\jtgcamn$. More precisely, $\jtgcam$ is isomorphic to the $\mathbb{G}_{m}$ torsor corresponds to the pullback of the line bundle $\alpha(\mathcal{F}^{n}_{T})\mid_{\widetilde{x}_{1}}$ on $\bunt(\cam^{n})$ via $\jtgcamn\rightarrow\bunt(\cam^{n})$. The group homomorphism $\mathbb{G}_{m}\rightarrow\jtgcam$ can be identified with $\mathbb{G}_{m}\rightarrow\textrm{Gr}_{J_{G,\cam},x}\rightarrow\jtgcam$ where the first arrow is obtained by $\mathbb{G}_{m}\simeq J_{G,\cam^{n}}(\widehat{O}_{x})/J_{G,\cam}(\widehat{O}_{x})\subseteq\textrm{Gr}_{J_{G,\cam},x}$. 
\end{lemma}
\begin{proof}
To begin with, let us recall that $\jtgcamn$ parameterizes $W$-equivariant $T$-bundles on $\cam^{n}$ together with some additional structures at ramification points. For details, see Subsection $3.1$ of \cite{Geometric Langlands}. Let $T_{\alpha}$ be $T/(s_{\alpha}-1)$, the torus of coinvariants. One has a $W$-equivariant $T$-bundle on $\cam^{n}\times\jtgcamn$. The $W$-equivariant structure implies that we have an isomorphism of $T$-bundles on $\jtgcamn$: 
$$T\times^{T,s_{\alpha}}(s^{-1}_{\alpha})^{*}(\mathcal{F}^{n}_{T})\mid_{\widetilde{x}_{1}}\simeq\mathcal{F}^{n}_{T}\mid_{\widetilde{x}_{1}} .$$
Take the associated $T_{\alpha}$ bundles on both sides and use the fact that 
$$(s^{-1}_{\alpha})^{*}(\mathcal{F}^{n}_{T})\mid_{\widetilde{x}_{1}}\simeq \mathcal{F}^{n}_{T}\mid_{\widetilde{x}_{2}} ,$$ 
one concludes that there exists an isomorphism of $T_{\alpha}$-bundles:
$$T_{\alpha}\times^{T}\mathcal{F}^{n}_{T}\mid_{\widetilde{x}_{2}}\simeq T_{\alpha}\times^{T}\mathcal{F}^{n}_{T}\mid_{\widetilde{x}_{1}} .$$
So one has a trivialization of the $T_{\alpha}$ bundle $T_{\alpha}\times^{T}(\mathcal{F}^{n}_{T}\mid_{\widetilde{x}_{1}}\otimes(\mathcal{F}^{n}_{T})^{-1}\mid_{\widetilde{x}_{2}})$. Let $K_{\alpha}$ be the kernel of $T\rightarrow T_{\alpha}$. Then the discussions above implies that the $T$-bundle $\mathcal{F}^{n}_{T}\mid_{\widetilde{x}_{1}}\otimes(\mathcal{F}^{n}_{T})^{-1}\mid_{\widetilde{x}_{2}}$ on $\jtgcamn$ has a canonical reduction to a $K_{\alpha}$. Denote the $K_{\alpha}$ bundle by $\mathcal{F}_{K_{\alpha}}$. By definition, a point in $\mathcal{F}_{K_{\alpha}}$ corresponds to a point in $\jtgcamn$ together with a trivialization of $\mathcal{F}_{K_{\alpha}}$. By our discussions above, a trivialization of $\mathcal{F}_{K_{\alpha}}$ gives a trivialization of $\mathcal{F}^{n}_{T}\mid_{\widetilde{x}_{1}}\otimes(\mathcal{F}^{n}_{T})^{-1}\mid_{\widetilde{x}_{2}}$, so we get a gluing data between $\mathcal{F}^{n}_{T}\mid_{\widetilde{x}_{1}}$ and $\mathcal{F}^{n}_{T}\mid_{\widetilde{x}_{2}}$. Using the $W$-equivariance one concludes that we actually get gluing datas between $\mathcal{F}^{n}_{T}\mid_{w(\widetilde{x}_{1})}$ and $\mathcal{F}^{n}_{T}\mid_{w(\widetilde{x}_{2})}$ for all $w\in W$ and these are compatible with the $W$-equivariance structure. Since $\cam$ has nodal singularities, we see that these gluing data determines a $W$-equivariant $T$-torsor $\mathcal{G}_{T}$ on $\cam$. Moreover, one checks that the $W$-equivariance structure induces the identity automorphism on $T_{\alpha}\times^{T}\mathcal{G}_{T}\mid_{\widetilde{x}}$. And since $\cam^{n}$ is isomorphic to $\cam$ at points away from the preimage of $x\in X$, the $W$-equivariant $T$-torsor $\mathcal{G}_{T}$ is actually determines a $J_{G,\cam}$ torsor away from the preimage of $x$.

Assume first that $\check{\alpha}$ is primitive. Then we can conclude by Lemma $3.1.2$ of \cite{Geometric Langlands} that we actually get a $J_{G,\cam}$ torsor. So this defines a morphism $\mathcal{F}_{K_{\alpha}}\rightarrow\jtgcam$ and it is not hard to see the procedure above actually gives an isomorphism between them. We also claim that in this case $\mathcal{F}_{K_{\alpha}}$ is isomorphic to the pullback of the line bundle $\alpha(\mathcal{F}^{n}_{T})\mid_{\widetilde{x}_{1}}$ on $\bunt(\cam^{n})$. Indeed, since the coroot $\check{\alpha}$ is primitive, $\check{\alpha}$ induces an isomorphism: 
$$\mathbb{G}_{m}\xrightarrow{\check{\alpha}}K_{\alpha}\subseteq T .$$ 
We shall now identify $K_{\alpha}$ with $\mathbb{G}_{m}$ via $\check{\alpha}$. By the definition of $\mathcal{F}_{K_{\alpha}}$, we have 
$$\check{\alpha}_{*}(\mathcal{F}_{K_{\alpha}})\simeq \mathcal{F}^{n}_{T}\mid_{\widetilde{x}_{1}}\otimes(\mathcal{F}^{n}_{T})^{-1}\mid_{\widetilde{x}_{2}} .$$
We claim that 
$$\mathcal{F}^{n}_{T}\mid_{\widetilde{x}_{1}}\otimes(\mathcal{F}^{n}_{T})^{-1}\mid_{\widetilde{x}_{2}}\simeq\check{\alpha}(\alpha(\mathcal{F}_{T}))\mid_{\widetilde{x}_{1}} .$$
This will prove that $\mathcal{F}_{K_{\alpha}}\simeq \alpha(\mathcal{F}^{n}_{T})\mid_{\widetilde{x}_{1}}$ since $\check{\alpha}$ is primitive. It is easy to see that we have:
$$T\times^{T,s_{\alpha}}(s^{-1}_{\alpha})^{*}(\mathcal{F}^{n}_{T})\simeq \check{\alpha}(\alpha(s^{-1 *}_{\alpha}(\mathcal{F}^{n}_{T})))\otimes s^{-1*}_{\alpha}(\mathcal{F}^{n}_{T}) .$$
So the isomorphism:
$$T\times^{T,s_{\alpha}}(s^{-1}_{\alpha})^{*}(\mathcal{F}^{n}_{T})\mid_{\widetilde{x}_{2}}\simeq\mathcal{F}^{n}_{T}\mid_{\widetilde{x}_{2}}$$
implies that 
$$\mathcal{F}^{n}_{T}\mid_{\widetilde{x}_{1}}\otimes(\mathcal{F}^{n}_{T})^{-1}\mid_{\widetilde{x}_{2}}\simeq\check{\alpha}(\alpha(\mathcal{F}^{n}_{T}))\mid_{\widetilde{x}_{1}} .$$

Next we shall deal with the case when $\check{\alpha}$ is not primitive. In this case we claim that we still have a natural isomorphism $\jtgcam\simeq\alpha(\mathcal{F}^{n}_{T})\mid_{\widetilde{x}_{1}}$. Indeed, since $\check{\alpha}$ is not primitive, one concludes that $\alpha$ induces an isomorphism between $K_{\alpha}$ and $\mathbb{G}_{m}$. Since 
$$T\times^{K_{\alpha}}\mathcal{F}_{K_{\alpha}}\simeq\mathcal{F}^{n}_{T}\mid_{\widetilde{x}_{1}}\otimes(\mathcal{F}^{n}_{T})^{-1}\mid_{\widetilde{x}_{2}} , $$
we conclude that 
$$\mathcal{F}_{K_{\alpha}}\simeq \alpha(\mathcal{F}^{n}_{T})\mid_{\widetilde{x}_{1}}\otimes\alpha(\mathcal{F}^{n}_{T})^{-1}\mid_{\widetilde{x}_{2}}\simeq\alpha(\mathcal{F}^{n}_{T})^{\otimes 2}\mid_{\widetilde{x}_{1}}.$$
where the last isomorphism comes from $s_{\alpha}$ equivariance. So if we have a point in $\alpha(\mathcal{F}^{n}_{T})\mid_{\widetilde{x}_{1}}$, we get a trivialization of $\alpha(\mathcal{F}^{n}_{T})\mid_{\widetilde{x}_{1}}$. This induces a trivialization of $\mathcal{F}_{K_{\alpha}}$, so by our previous discussions this gives rise to a $W$-equivariant $T$-torsor $\mathcal{G}_{T}$ such that the $W$-equivariance structure induces the identity automorphism on $T_{\alpha}\times^{T}\mathcal{G}_{T}\mid_{\widetilde{x}}$. And that $\mathcal{G}_{T}$ corresponds to a $J_{G,\cam}$ torsor away from $x$. To lift this to a $J_{G,\cam}$ torsor near $x$, we we need a trivialization of $\alpha(\mathcal{G}_{T})\mid_{\widetilde{x}}$ that is compatible with the trivialization of $\check{\alpha}(\alpha(\mathcal{G}_{T}))\mid_{\widetilde{x}}$, see Lemma $3.1.3$ of \cite{Geometric Langlands}. 
This is provided by the trivialization of $\alpha(\mathcal{F}^{n}_{T})\mid_{\widetilde{x}_{1}}$. 

The last claim follows from an easy calculation of the cokernel of the group scheme $J_{G,\cam}\rightarrow J_{G,\cam^{n}}$, which is left the reader.
\end{proof}

\begin{remark}
It is possible to prove the Lemma~\ref{structure of jtors} by showing that the quotient $J_{G,\cam^{n}}/J_{G,\cam}$ is isomorphic to $i_{x *}(\mathbb{G}_{m})$ directly.
\end{remark}

To proceed further, one needs to analyze the structure of $\higc(\cam)$. To do so we need to study the local situation. That is, one needs to analyze the category of Higgs bundles on the formal disk at $x\in X$. We first recall the following well-known lemma:
\begin{lemma}\label{reduce to levi}
Let $L$ be a Levi subgroup of $G$. Denote the Chevalley base of $G$ by $\mathfrak{c}_{G}$ and the Chevalley base of $L$ by $\mathfrak{c}_{L}$. Let $\mathfrak{c}^{\circ}_{L}$ be the open subscheme of $\mathfrak{c}_{L}$ such that the natural morphism $\mathfrak{c}_{L}\rightarrow\mathfrak{c}_{G}$ is etale over $\mathfrak{c}^{\circ}_{L}$. Denote the lie algebra of $L$ by $\mathfrak{l}$ and lie algebra of $G$ by $\mathfrak{g}$. Then we have a Cartesian diagram:
$$\xymatrix{
(\mathfrak{l}/L)^{\circ} \ar[r] \ar[d] & \mathfrak{c}^{\circ}_{L} \ar[d] \\
\mathfrak{g}/G \ar[r] & \mathfrak{c}_{G}
}
$$
where $(\mathfrak{l}/L)^{\circ}$ is the preimage of $\mathfrak{c}^{\circ}_{L}$ in $\mathfrak{l}/L$. 
\end{lemma}

\begin{remark}
The open set $\mathfrak{c}^{\circ}_{L}$ can be described as follows: Let $x$ be a semisimple element in $\mathfrak{l}$. Then the conjugacy class of $x$ lies in $\mathfrak{c}^{\circ}_{L}$ iff $Z_{G}(x)\subseteq L$. 
\end{remark}

Here is a useful corollary:
\begin{corollary}\label{local tool}
Let $x\in X$ be a closed point and $\widehat{O}_{x}$ be the formal completion at $x$. Let $\cam_{G}$ be a cameral cover of $\spec(\widehat{O}_{x})$ for $G$ and $\cam_{L}$ be a cameral cover of $\spec(\widehat{O}_{x})$ for $L$. Assuming that there exists an isomorphism of $W_{G}$-covers: $\cam_{G}\simeq W_{G}\times^{W_{L}}\cam_{L}$ where $W_{G}$ and $W_{L}$ denotes the Weyl group of $G$ and $L$, respectively. Then the category of $G$-Higgs bundles on $\spec(\widehat{O}_{x})$ with cameral cover $\cam_{G}$ is equivalent to the category of $L$-Higgs bundles on $\spec(\widehat{O}_{x})$ with cameral cover $\cam_{L}$. 
\end{corollary}
\begin{proof}
The assumptions on the cameral covers implies that there exists a morphism $\spec(\widehat{O}_{x})\xrightarrow{s}\mathfrak{c}^{\circ}_{L}$ such that $\cam_{L}$ corresponds to $s$ and $\cam_{G}$ corresponds to the composition $\spec(\widehat{O}_{x})\xrightarrow{s}\mathfrak{c}^{\circ}_{L}\rightarrow\mathfrak{c}_{G}$. A $L$-Higgs bundle with cameral cover $\cam_{L}$ is equivalent to a factorization:
$$\xymatrix{
 & \spec(\widehat{O}_{x}) \ar@{-->}[ld] \ar[d]_{s}\\
\mathfrak{l}/L \ar[r] & \mathfrak{c}_{L} .
}
$$
While $G$-Higgs bundles with cameral cover $\cam_{G}$ is equivalent to a factorization:
$$\xymatrix{
 &  \spec(\widehat{O}_{x}) \ar@{-->}[ldd] \ar[d]_{s}\\
 & \mathfrak{c}_{L} \ar[d]\\
\mathfrak{g}/G \ar[r] & \mathfrak{c}_{G} .
}
$$
Since the image of $s$ lies in $\mathfrak{c}^{\circ}_{L}$, the previous lemma implies the claim.
\end{proof}

From Lemma~\ref{structure of jtors} we can derive the following:
\begin{lemma}\label{normalization}
Let $G$ be simply connected. There exists a surjective birational morphism:
$$\mathbf{P}^{1}\times^{\mathbb{G}_{m}}\mjtgcam\rightarrow\mhig(\cam) .$$
Moreover, this realizes $\mathbf{P}^{1}\times^{\mathbb{G}_{m}}\mjtgcam$ as the normalization of $\mhig(\cam)$.
\end{lemma}
\begin{proof}
First we claim that there exists a morphism $\mathbf{P}^{1}\rightarrow \hig(\cam)$. Indeed, by our assumptions on $\cam$, if we choose a point $\widetilde{x}\in D_{\alpha}$ above $x\in X$, we see that over the formal disk at $x$, the cameral cover $\cam$ is induced from a cameral cover for $L$ where $L$ is the minimal Levi subgroup of $G$ corresponds to the root $\alpha$. Hence Corollary~\ref{local tool} applies. Since $G$ is simply connected, $\check{\alpha}$ is primitive, so $L$ is isomorphic to either $\textrm{GL}_{2}\times T_{0}$ or $\textrm{SL}_{2}\times T_{0}$ where $T_{0}$ is a torus, see Section $3$ of \cite{Higgs Faltings}. In either case, our assumptions about the cameral cover implies that the cameral cover for $L$ over $\spec(\widehat{O}_{x})$ is isomorphic to 
$$\spec (\widehat{O}_{x}[T]/(T^{2}-t^{2}))$$
where $t$ is the uniformizer for $\widehat{O}_{x}$. Now consider the matrix 
\begin{equation*}
\gamma=
\left( \begin{array}{cc}
t & 0 \\
0 & -t \\
\end{array} \right)
\end{equation*}
View it as an element in $\textrm{SL}_{2}$. Denote the affine springer fiber of $\gamma$ by $\textrm{Spr}(\gamma)$. It is well -known that $\textrm{Spr}(\gamma)$ parameterizes $\textrm{SL}_{2}$ Higgs bundles on $\spec(\widehat{O}_{x})$ with cameral cover $\spec (\widehat{O}_{x}[T]/(T^{2}-t^{2}))$ together with a trivialization over $\spec(\mathcal{K}_{x})$. It is well-known that affine springer fibers of $\gamma$ is an infinite chain of $\mathbf{P}^{1}$'s and regular Higgs bundles corresponds to $\mathbb{G}_{m}\subseteq\mathbf{P}^{1}$. Consider $\textrm{SL}_{2}$ as a subgroup of $L$ and fix a copy of $\mathbf{P}^{1}$, we conclude from Corollary~\ref{local tool} that we have a morphism $\mathbf{P}^{1}\rightarrow \mhig(\cam)$. Moreover, one checks that in our case, the neutral component of the reduced part of $\textrm{Gr}_{J_{G,\cam}}$ is isomorphic to $\mathbb{G}_{m}$. So from the product formula (See Proposition $4.15.1$ of \cite{Fundamental lemma}) we get the desired morphism. When we restrict to the open part $\mathbb{G}_{m}\times^{\mathbb{G}_{m}}\mjtgcam$ this is an isomorphism. Moreover, by Lemma~\ref{structure of jtors}, we conclude that $\mathbf{P}^{1}\times^{\mathbb{G}_{m}}\mjtgcam$ is a $\mathbf{P}^{1}$ fibration over $\mjtgcamn$, hence it is projective. So $\mathbf{P}^{1}\times^{\mathbb{G}_{m}}\mjtgcam\rightarrow\mhig(\cam)$ is surjective. This implies the claim.
\end{proof}

\begin{corollary}\label{description of projective bundle}
Let $\mathcal{Q}$ be the line bundle on $\mjtgcamn$ corresponds to the $\mathbb{G}_{m}$ torsor $\mjtgcam$ on $\mjtgcamn$, see Lemma~\ref{structure of jtors}. Then the morphism $\mathbf{P}^{1}\times^{\mathbb{G}_{m}}\mjtgcam\rightarrow\mjtgcamn$ identifies $\mathbf{P}^{1}\times^{\mathbb{G}_{m}}\mjtgcam$ with the projective bundle $\textrm{Proj}(O\oplus\mathcal{Q})$ over $\mjtgcamn$.
\end{corollary}
\begin{proof}
This follows directly from Lemma~\ref{structure of jtors}.
\end{proof}

We also have the following description about the complement of $\hig^{reg}(\cam)$ in $\hig(\cam)$:
\begin{lemma}\label{description of irregular higgs bundles}
The category of irregular Higgs bundles with cameral cover $\cam$ over $k$ form a single orbit under the action of the Picard stack $\jtgcam(k)$. Moreover, it is a torsor over $\jtgcamn(k)$.
\end{lemma}
\begin{proof}
Let $\lambda_{1}$ and $\lambda_{2}$ be two irregular Higgs bundles on $X$ with cameral cover $\cam$. We will prove they are locally isomorphic and that the sheaf of isomorphisms $\textrm{Iso}(\lambda_{1},\lambda_{2})$ is a $J_{G,\cam^{n}}$ torsor. This will imply the claim of the lemma. 

Since $\cam$ is smooth away from the preimage of $x$, we see that all Higgs bundles with cameral cover $\cam$ are regular over $X-x$, hence they form a torsor for the Picard stack $\jtgcam(X-x)\simeq\jtgcamn(X-x)$. So the claim is true over $X-x$. Now we pick an open set $U$ containing $x$. Using Lemma~\ref{zariski local trivial}, one may shrink $U$ so that $\lambda_{1}$ and $\lambda_{2}$ are isomorphic over $U-x$. To show they are isomorphic over $U$ we look at the formal disk $\spec(\widehat{O}_{x})$. We claim that all irregular Higgs bundles with cameral cover $\cam$ over $\spec(\widehat{O}_{x})$ are isomorphic and that the set of isomorphisms between $\lambda_{1}$ and $\lambda_{2}$ over $\spec(\widehat{O}_{x})$ is a $J_{G,\cam^{n}}(\widehat{O}_{x})$ torsor. Assuming this for the moment, we will finish the proof. Indeed, $\lambda_{1}$ and $\lambda_{2}$ are isomorphic on $U-x$ as well as on $\spec(\widehat{O}_{x})$, the obstruction for them to be isomorphic is an element in $\xi\in J_{G,\cam^{n}}(U-x)\setminus J_{G,\cam^{n}}(\mathcal{K}_{x})/J_{G,\cam^{n}}(\widehat{O}_{x})$. Now we can shrink $U$ so that it becomes zero. Hence $\lambda_{1}$ and $\lambda_{2}$ are isomorphic over $U$ and it is not hard to check that the set of isomorphisms between them is a $J_{G,\cam^{n}}(U)$ torsor.

Now we analyze the local situation. Let us recall that if we have a cameral cover $\cam\xrightarrow{\pi} X$ for $\textrm{GL}_{2}$, then the category of $\textrm{GL}_{2}$ Higgs bundles with cameral cover $\cam$ is equivalent to the category of torsion free rank on sheaves on $\cam$. Similarly, if $\cam\xrightarrow{\pi} X$ is a cameral cover for $\textrm{SL}_{2}$, then the category of $\textrm{SL}_{2}$ Higgs bundles with cameral cover $\cam$ is equivalent to the pair $(F,s)$ where $F$ is a torsion free rank on sheaf on $\cam$ and $s$ is an isomorphism $\de(\pi_{*}(F))\simeq O_{X}$. In the local situation, using Corollary~\ref{local tool} as we did in the proof of Lemma~\ref{normalization}, we reduce the question to $\textrm{SL}_{2}$ or $\textrm{GL}_{2}$ case. Denote the cameral cover and its normalization by $\spec(\widetilde{O}_{x})$ and $\spec(\widetilde{O}^{n})$. Since $\spec(\widetilde{O}_{x})\simeq \spec (\widehat{O}_{x}[T]/(T^{2}-t^{2}))$ has nodal singularity, it is well known that every torsion-free rank one sheaves on it that are not line bundles are isomorphic. So in the $\textrm{GL}_{2}$ case we conclude that irregular Higgs bundles over $\spec(\widehat{O}_{x})$ are all isomorphic. In the $\textrm{SL}_{2}$ case one observe that if $F$ is a torsion free rank one sheaf that is not a line bundle, then $F\simeq \widetilde{O}^{n}$. Since $\spec(\widetilde{O}^{n})$ is an unramified cameral cover over $\spec(\widehat{O}_{x})$, the determinant map $\widetilde{O}^{n \times}\xrightarrow{\de}\widehat{O}^{\times}_{x}$ is surjective. Hence this shows that in the $\textrm{SL}_{2}$ case all irregular Higgs bundles are isomorphic. The claim that the set of isomorphisms between $\lambda_{1}$ and $\lambda_{2}$ over $\spec(\widehat{O}_{x})$ is a $J_{G,\cam^{n}}(\widehat{O}_{x})$ torsor also follows from this. 

\end{proof}

Combining these together, we can show:
\begin{lemma}\label{construct pushout}
There exists a pushout diagram (In our case the pushout exists in the category of schemes exists, see $36.59$ of \cite{Stacks}):
$$\xymatrix{
\mjtgcamn\amalg\mjtgcamn \ar[d] \ar[r] & \mjtgcamn \ar[d]\\
\mathbf{P}^{1}\times^{\mathbb{G}_{m}}\mjtgcam \ar[r] & \mathcal{X} .
}
$$
Moreover, there exists a natural morphism $\mathcal{X}\rightarrow\mhig(\cam)$. It is a finite surjective morphism which induces a bijection on closed points.
\end{lemma}
\begin{proof}
In Lemma~\ref{normalization} we already constructed the morphism: 
$$\mathbf{P}^{1}\times^{\mathbb{G}_{m}}\mjtgcam\rightarrow \mhig(\cam) .$$ 
It comes from a morphism $\mathbf{P}^{1}\rightarrow \mhig(\cam)$ by realizing $\mathbf{P}^{1}$ as a component of the affine springer fiber of a Levi subgroup. As we already mentioned in the proof of Lemma~\ref{normalization}, the open set $\mathbb{G}_{m}\subseteq\mathbf{P}^{1}$ corresponds to regular Higgs bundles and the two points $0$ and $\infty$ correspond to irregular Higgs bundles. Denote the image of $0$ and $\infty$ in $\mhig(\cam)$ by $\lambda_{0}$ and $\lambda_{\infty}$. By Lemma~\ref{description of irregular higgs bundles} we see that there exists a $J_{G,\cam}$ torsor $\mathcal{T}$ such that $\lambda_{0}\otimes\mathcal{T}\simeq\lambda_{\infty}$. Now we have two morphisms $i_{0}$ and $i_{\infty}$ from $\mjtgcamn$ to $\mhig(\cam)$ induced by:
\begin{gather}
\mjtgcamn\simeq 0\times^{\mathbb{G}_{m}}\mjtgcam\rightarrow \mhig(\cam) \notag \\
\mjtgcamn\simeq\infty\times^{\mathbb{G}_{m}}\mjtgcam\rightarrow \mhig(\cam) \notag .
\end{gather}
Moreover, if $\mathcal{T}^{n}$ is the induced $J_{G,\cam^{n}}$ torsor from $\mathcal{T}$, we see from Lemma~\ref{description of irregular higgs bundles} that $i_{\infty}(-)=i_{0}(-\otimes\mathcal{T}^{n})$. Now let us identify the irregular Higgs bundles with $\mjtgcamn$ via Lemma~\ref{description of irregular higgs bundles} using the $\jtgcam$ action on $\lambda_{0}$. Then by our discussion above, the following diagram commutes:
\[\xymatrixcolsep{3pc}\xymatrix{
\mjtgcamn\amalg\mjtgcamn \ar[r]^-{\textrm{id}\amalg (\otimes \mathcal{T}^{n})} \ar[d]_{0\amalg\infty} & \mjtgcamn \ar[d]_{i_{0}}\\
\mathbf{P}^{1}\times^{\mathbb{G}_{m}}\mjtgcam \ar[r] & \mhig(\cam) .
}
\]
Hence if we denote the pushout by $\mathcal{X}$, then we get $\mathcal{X}\rightarrow\mhig(\cam)$. And by our discussions above, it is easy to see that it induces an isomorphism on closed points. By Lemma $36.59.4$ of \cite{Stacks} and the fact that $\mathbf{P}^{1}\times^{\mathbb{G}_{m}}\mjtgcam\rightarrow\mathcal{X}$ is surjective, we see that $\mathcal{X}$ is proper. Hence $\mathcal{X}\rightarrow\mhig(\cam)$ is a proper surjective morphism which induces an isomorphism on closed points, hence it is finite. 

\end{proof}

\begin{corollary}\label{the gluing map}
The gluing morphism between $\mjtgcamn$ in the pushout diagram in Lemma~\ref{construct pushout} can be described as follows: Let us fix a point $\widetilde{x}\in\cam$ in the preimage of $x\in X$ that lies in the ramification divisor $D_{\alpha}$. Let $\widetilde{x}_{1}$ be a point in $\cam^{n}$ above $\widetilde{x}$. Then the gluing morphism is induced by translation by $Nm(\check{\alpha}_{*}O(\widetilde{x}_{1}))$. Here $Nm$ is the morphism $\cam^{n}\times X_{*}(T)\rightarrow\bunt(\cam^{n})\rightarrow\jtgcamn$ defined in Section $3$ (Notice that since $\cam^{n}$ is a smooth cameral cover over $X$, $\cam^{n}\times X_{*}(T)\rightarrow\bunt(\cam^{n})$ is defined over the entire $\cam^{n}$). 
\end{corollary}
\begin{proof}
Recall from the proof of Lemma~\ref{normalization} that $\alpha$ determines a subgroup $\textrm{SL}_{2}\subseteq G$ and that the morphism $\mathbf{P}^{1}\rightarrow\hig(\cam)$ is obtained by identifying $\mathbf{P}^{1}$ as a component of the affine springer fiber of the matrix $\gamma$ for the group $\textrm{SL}_{2}$:
\begin{equation*}
\gamma=
\left( \begin{array}{cc}
t & 0 \\
0 & -t \\
\end{array} \right) .
\end{equation*}
The full affine springer fiber is a chain of $\mathbf{P}^{1}$'s equipped with an action of the group of integers $\mathbb{Z}$ where $1\in\mathbb{Z}$ acts on the lattices fixed by $\gamma$ via the matrix: 
\begin{equation*}
\left( \begin{array}{cc}
t & 0 \\
0 & t^{-1} \\
\end{array} \right) .
\end{equation*}
$0\in\mathbf{P}^{1}$ and $\infty\in\mathbf{P}^{1}$ are identified under this action. Notice that this $\mathbb{Z}$ action comes from the action of the affine grassmannian of the regular centralizer group scheme for $\textrm{SL}_{2}$. Moreover, if we denote the restriction of $\cam$ to $\spec(\widehat{O}_{x})$ by $\cam_{x}$, then from the proof of Lemma~\ref{normalization} and Corollary~\ref{local tool}, we see that $\cam_{x}$ is induced from a cameral cover $\widetilde{Y}_{x}$ for $\textrm{SL}_{2}$: $\widetilde{X}_{x}\simeq W\times^{S_{\alpha}}\widetilde{Y}_{x}$ where $S_{\alpha}$ stands for the subgroup generated by the reflection $s_{\alpha}$. We may assume $\widetilde{x}\in\widetilde{Y}_{x}$. This induces a morphism $\textrm{Gr}_{J_{\textrm{SL}_{2},\widetilde{Y}_{x},x}}\rightarrow\textrm{Gr}_{J_{G,\cam_{x},x}}$. Passing to normalizations we get $\textrm{Gr}_{J_{\textrm{SL}_{2},\widetilde{Y}^{n}_{x},x}}\rightarrow\textrm{Gr}_{J_{G,\cam^{n}_{x},x}}$.
One can view 
\begin{equation*}
\left( \begin{array}{cc}
t & 0 \\
0 & t^{-1} \\
\end{array} \right) 
\end{equation*}
as an element in $\textrm{Gr}_{J_{\textrm{SL}_{2},\widetilde{Y}_{x},x}}$. Hence we get the corresponding elements in 
$\textrm{Gr}_{J_{G,\cam_{x},x}}$ and $\textrm{Gr}_{J_{G,\cam^{n}_{x},x}}$.
It is not hard to check that the element 
\begin{equation*}
\left( \begin{array}{cc}
t & 0 \\
0 & t^{-1} \\
\end{array} \right) 
\end{equation*}
in $\textrm{Gr}_{J_{\textrm{SL}_{2},\widetilde{Y}^{n}_{x},x}}$ corresponds to the line bundle $O(\widetilde{x}_{1}-\widetilde{x}_{2})$ on $\widetilde{Y}^{n}_{x}$ with its natural trivialization at the generic point. So this is equal to $Nm_{\textrm{SL}_{2}}(\check{\alpha}_{*}O(\widetilde{x}_{1}))$ via the morphism:
$$\widetilde{Y}^{n}_{x}\times\mathbb{Z}\check{\alpha}\xrightarrow{Nm_{\textrm{SL}_{2}}}\textrm{Gr}_{J_{\textrm{SL}_{2},\widetilde{Y}^{n}_{x},x}} .$$
Unraveling the definitions, we see that its image under $$\textrm{Gr}_{J_{\textrm{SL}_{2},\widetilde{Y}^{n}_{x},x}}\rightarrow\textrm{Gr}_{J_{G,\cam^{n}_{x},x}}\rightarrow\jtgcam$$
is equal to $Nm(\check{\alpha}_{*}O(\widetilde{x}_{1}))$. 
\end{proof}

In order to show the morphism $\mathcal{X}\rightarrow\mhigc(\cam)$ in Lemma~\ref{construct pushout} is an isomorphism, one has to analyze the local structure of $\hig(\cam)$ near the points correspond to irregular Higgs bundles. This is provided by the following:

\begin{lemma}\label{local deformation}
Let $\cam$ be a cameral cover of $X$ satisfying the assumptions at the beginning of this section. Let $\cam_{x}$ be the induced cameral cover over the formal completion $\spec(\widehat{O}_{x})$ and let $\hig(\cam_{x})$ be the stack of Higgs bundles on $\spec(\widehat{O}_{x})$ with cameral cover $\cam_{x}$. Then :
\begin{enumerate1}
\item $\hig(\cam_{x})$ admits versal deformations.
\item The natural morphism $\hig(\cam)\rightarrow \hig(\cam_{x})$ is formally smooth.
\end{enumerate1}
\end{lemma}
\begin{proof}
For part $(1)$, it is not hard to check that the deformation functor $\textrm{Def}_{\hig(\cam_{x})}$ attached to $\hig(\cam_{x})$ satisfies the Schlessinger's conditions $S1$ and $S2$, see Section $87.16$ of \cite{Stacks}. Moreover, its tangent spaces are also of finite dimension. From this one concludes that it admits versal deformations, see Lemma $87.13.4$ of \cite{Stacks}.

For part $(2)$, let $\spec(A)\hookrightarrow\spec(A')$ be a small extension. Assume we have a commutative diagram:
$$\xymatrix{
 & & \hig(\cam) \ar[d]\\
\spec(A) \ar[r] \ar[rru] & \spec(A') \ar[r] \ar@{.>}[ru] & \hig(\cam_{x})
}
$$
where $\spec(A)\rightarrow\hig(\cam)$ is given by a Higgs bundle $x_{A}$ on $X$ and $\spec(A')\rightarrow\hig(\cam_{x})$ is given by a Higgs bundle $x^{loc}_{A'}$ on $\spec(A'[[t]])$. We want to lift the $A'$ point of $\hig(\cam_{x})$ to an $A'$ point of $\hig(\cam)$. To do this, we interpret a Higgs bundle on $X\times\spec(A')$ with cameral cover $\cam$ as a Higgs bundle on $(X-x)\times\spec(A')$, a Higgs bundle on $\spec(A'[[t]])$ together with an isomorphism of the two over the formal punctured disk $\spec(A'((t)))$. So what we need to is a deformation of $x_{A}\mid_{(X-x)\times\spec(A)}$ to $(X-x)\times\spec(A')$ and then a deformation of the isomorphism between $x_{A}\mid_{\spec(A((t)))}$ and $x^{loc}_{A'}\mid_{\spec(A((t)))}$ to $\spec(A'((t)))$. But since $\cam$ is smooth away from $x$, deformations of $x_{A}\mid_{(X-x)\times\spec(A)}$ is the same as deformations of $J_{G,\cam}$ torsors over $X-x$. Using the fact that $X-x$ is affine, we see that such deformations exists and all deformations are isomorphic. Similarly, $\cam$ is smooth over the formal punctured disk $k((t))$, so any two deformations over $\spec(A((t)))$ will be isomorphic. This finishes the proof. 

\end{proof}

\begin{remark}
Similar arguments can be used to show that if $\cam$ is a reduced cameral cover over $X$, then $\hig(\cam_{x})$ admits versal deformations. Moreover, if we denote $\hig^{0}(\cam)$ to be the open substack of $\hig(\cam)$ corresponds to Higgs bundles that are regular away from $x$, then the morphism $\hig^{0}(\cam)\rightarrow\hig(\cam_{x})$ is formally smooth.
\end{remark}

Now we can finish the proof of Theorem~\ref{main theorem of app B}:
\begin{proof}
First we claim that if $y$ is a closed point in $\mhig(\cam)$ corresponds to an irregular Higgs bundle and $O_{\mhig(\cam),y}$ is the local ring of $\mhig(\cam)$ at $y$, then the formal completion $\widehat{O}_{\mhig(\cam),y}$ is isomorphic to $k[[t_1,t_2,\cdots,t_n]]/(t_{1}t_{2})$. Indeed, since $\hig(\cam)$ is a $Z(G)$ gerbe over $\mhig(\cam)$ by Lemma~\ref{more geometry}, we only need to prove this for versal deformations of $\hig(\cam)$. Using Lemma~\ref{local deformation} we conclude that we only need to prove it for versal deformations of $\hig(\cam_{x})$. We can further reduced the problem to versal deformations of $\mathcal{H}iggs_{\textrm{GL}_{2}}(\widetilde{C}_{x})$ or $\mathcal{H}iggs_{\textrm{SL}_{2}}(\widetilde{C}_{x})$ by Lemma~\ref{local tool}, where $\widetilde{C}_{x}$ is a nodal $S_{2}$ cover over $\spec(\widehat{O}_{x})$. This can again be reduced to the versal deformations of rank two Higgs bundles on any smooth curve $C$ with a nodal cameral cover $\widetilde{C}$ using Lemma~\ref{local deformation}. In this case Higgs bundles with cameral cover $\widetilde{C}$ are parameterized by the compactified Jacobian $\overline{J}_{\widetilde{C}}$. It is well-known that the natural morphism $\widetilde{C}\times J_{\widetilde{C}}\rightarrow \overline{J}_{\widetilde{C}}$ is smooth. So we are done. 

Now we can prove $\mathcal{X}\simeq\mhig(\cam)$. Since $\mathcal{X}\rightarrow\mhig(\cam)$ is birational and finite and that $\mhig(\cam)$ is Cohen-Macaulay (See Lemma~\ref{geo of hitchin fibration}), we only need to prove that $O_{\mathcal{X}}\simeq O_{\mhig(\cam)}$ at all codimension one points of $\mhig(\cam)$. Since $\mathcal{X}$ is isomorphic to $\mhig(\cam)$ over the open set of regular Higgs bundles, we only need to look at generic points of the closed subvariety of $\mhig(\cam)$ formed by irregular Higgs bundles. Let $\eta$ be the generic point. Denote the morphism $\mathbf{P}^{1}\times^{\mathbb{G}_{m}}\mjtgcam\rightarrow\mhig(\cam)$ by $q$. Our discussions above the formal completions of $\mhig(\cam)$ at closed points implies that the cokernel of $O_{\mhig(\cam)}\rightarrow q_{*}(O_{\mathbf{P}^{1}\times^{\mathbb{G}_{m}}\mjtgcam})$ is supported on $\mjtgcamn$ and it is isomorphic to a line bundle on $\mjtgcamn$. Let us look at the local ring $O_{\mhig(\cam),\eta}$. We conclude from the discussion above that the cokernel of $O_{\mhig(\cam),\eta}\rightarrow q_{*}(O_{\mathbf{P}^{1}\times^{\mathbb{G}_{m}}\mjtgcam}))_{\eta}$ is isomorphic to the residue field of $O_{\mhig(\cam),\eta}$. Since we have injections 
$$O_{\mhig(\cam),\eta}\hookrightarrow O_{\mathcal{X},\eta}\hookrightarrow q_{*}(O_{\mathbf{P}^{1}\times^{\mathbb{G}_{m}}\mjtgcam})_{\eta} ,$$ 
and that $O_{\mathcal{X},\eta}$ is not equal to $q_{*}(O_{\mathbf{P}^{1}\times^{\mathbb{G}_{m}}\mjtgcam})_{\eta}$, we are done.

\end{proof}

As a corollary, we will prove the following result about $H^{1}(O_{\hig})$. 
\begin{corollary}\label{desciption of H1}
Let $G$ be simply connected and $H$ be the Hitchin base. Let $\textrm{Lie}(\mjtg)$ be the vector bundle on $H$ corresponds to the lie algebra of $\mjtg$. To simplify the notation, let us also denote the restrictions of $\hig$ and $\textrm{Lie}(\mjtg)$ to $H_{int}$ still by $\hig$ and $\textrm{Lie}(\mjtg)$. Then we have $R^{1}\pi_{*}(O_{\hig})\simeq \textrm{Lie}(\mjtg)$ as coherent sheaves on $H_{int}$. Here $\pi$ is the morphism $\hig\xrightarrow{\pi} H_{int}$. 
\end{corollary}
\begin{proof}
Let us first construct the morphism $\textrm{Lie}(\mjtg)\rightarrow R^{1}\pi_{*}(O_{\hig})$. Pick a line bundle $\mathcal{M}$ on $\hig$ which is ample when restricted to the semistable locus. Pick a local section $\xi\in\textrm{Lie}(\mjtg)$. $\xi$ induces an infinitesimal automorphism of $\hig$ via the action of $\jtg$ on $\hig$. Consider the line bundle $\xi^{*}(\mathcal{M})\otimes\mathcal{M}^{-1}$ on $\hig\times D$ where $D$ denotes the ring of dual numbers. This line bundle is canonically trivial when restricted to $\hig$, hence it defines an element in $R^{1}\pi_{*}(O_{\hig})$. Next we will prove this induces an isomorphism between them. Consider the open set $U$ of $H_{int}$ defined in Lemma~\ref{an open set of H} in Appendix $C$. Denote the restriction of $\pi$ to $U$ by $\pi_{U}$. We will first prove that the morphism we constructed above induces an isomorphism $R^{1}\pi_{U*}O_{\hig\mid_{U}}\simeq\textrm{Lie}(\mjtg)\mid_{U}$ as coherent sheaves on $U$. To do this we fix a cameral cover $\cam$ and look at the cohomology $H^{1}(O_{\hig(\cam)})$. By Lemma~\ref{an open set of H}, we either have $\hig(\cam)=\hig^{reg}(\cam)$ or $\cam$ satisfies the assumptions given at the beginning of this section. If $\hig(\cam)=\hig^{reg}(\cam)$, then $\hig(\cam)\simeq\jtgcam$ and that $\mjtgcam$ is an abelian variety. Since $\mathcal{M}$ also restricts to an ample line bundle on $\mjtgcam$, the claim is well-known, see \cite{Abelian variety} or \cite{AF}. Now let us assume $\cam$ satisfies the assumptions given at the beginning of this section. Using Theorem~\ref{main theorem of app B} we can describe $H^{1}(O_{\mhig(\cam)})$ as follows. First we have an exact sequence (To simplify the notations we will write $\mjtg$ and $\mjtgn$ instead of $\mjtgcam$ and $\mjtgcamn$ in the rest of the proof):
$$0\rightarrow O_{\mhig(\cam)}\rightarrow O_{\mathbf{P}^{1}\times^{\mathbb{G}_{m}}\mjtg}\oplus O_{\mjtgn}\rightarrow O_{\mjtgn}\oplus O_{\mjtgn}\rightarrow 0 .$$
Since $\mathbf{P}^{1}\times^{\mathbb{G}_{m}}\mjtg$ is a $\mathbf{P}^{1}$ bundle over $\mjtgn$, we have $H^{i}(O_{\mathbf{P}^{1}\times^{\mathbb{G}_{m}}\mjtg})\simeq H^{i}(O_{\mjtgn})$. Since $\mjtgn$ is an abelian variety, translation by an element induces the trivial isomorphism on $H^{i}$. Using Corollary~\ref{the gluing map} we see that the induced morphism 
$$H^{1}(O_{\mathbf{P}^{1}\times^{\mathbb{G}_{m}}\mjtg})\rightarrow H^{1}(O_{\mjtgn})\oplus H^{1}(O_{\mjtgn})$$
can be identified with the diagonal embedding: 
$$H^{1}(O_{\mjtgn})\xrightarrow{\vartriangle} H^{1}(O_{\mjtgn})\oplus H^{1}(O_{\mjtgn}) .$$ 
The same is true for 
$$H^{1}(O_{\mjtgn})\rightarrow H^{1}(O_{\mjtgn})\oplus H^{1}(O_{\mjtgn}) .$$
Hence we see that $H^{1}(O_{\mhig(\cam)})$ fits into an exact sequence:
\begin{equation}\label{exact sequence}
0\rightarrow H^{0}(O_{\mjtgn})\rightarrow H^{1}(O_{\mhig(\cam)})\rightarrow H^{1}(O_{\mjtgn})\rightarrow 0 .
\end{equation}
Since $G$ is simply connected, $\mjtgn$ is connected. From this one concludes that $\dimension H^{1}(O_{\hig(\cam)})=\dimension H^{1}(O_{\mjtgn})+1$. Since we also have $\dimension\textrm{Lie}(\mjtg)=\dimension\textrm{Lie}(\mjtgn)+1$, we see that $H^{1}(O_{\hig(\cam)})$ and $\textrm{Lie}(\mjtg)$ has the same dimension. To prove they are isomorphic we need to show the morphism $\textrm{Lie}(\mjtg)\rightarrow H^{1}(O_{\hig(\cam)})$ is injective. Let $\mathcal{M}$ be the ample line bundle on $\hig(\cam)$ we pick at the beginning. Then the restriction of $\mathcal{M}$ to $\mjtgn$ via the embedding $\mjtgn\rightarrow\hig(\cam)$ in Theorem~\ref{main theorem of app B} is also an ample line bundle on the abelian variety $\mjtgn$. Hence if $\xi\in\textrm{Lie}(\mjtg)$ and that $\xi^{*}(\mathcal{M})\otimes\mathcal{M}^{-1}$ is trivial, then its restriction to $\mjtgn$ is also trivial. Because $\mathcal{M}$ is ample and $\mjtgn$ is an abelian variety, we see that the image of $\xi$ in $\textrm{Lie}(\mjtgn)$ is trivial. It is not hard to verify that the following diagram commutes:
$$\xymatrix{
\textrm{Lie}(\mjtg) \ar[r] \ar[d] & \textrm{Lie}(\mjtgn) \ar[d] \\
H^{1}(O_{\hig(\cam)}) \ar[r] & H^{1}(O_{\mjtgn}) .
}
$$
Hence by the exact sequence ~\ref{exact sequence}, we see that if $\xi$ lies in the kernel of $\textrm{Lie}(\mjtg)\rightarrow H^{1}(O_{\hig(\cam)})$, then $\xi$ is in kernel of $\textrm{Lie}(\mjtg)\rightarrow \textrm{Lie}(\mjtgn)$, which can be identifies with an element in the Lie algebra of $\mathbb{G}_{m}$ by Lemma~\ref{structure of jtors}. To show $\xi$ is zero, let us look at the pullback of $\mathcal{M}$ to $\mathbf{P}^{1}\times^{\mathbb{G}_{m}}\mjtg$. Let $O(1)$ be the universal line bundle on the projective bundle $\mathbf{P}^{1}\times^{\mathbb{G}_{m}}\mjtg$ over $\mjtgn$, see Corollary~\ref{description of projective bundle}. The pullback of $\mathcal{M}$ to $\mathbf{P}^{1}\times^{\mathbb{G}_{m}}\mjtg$ is isomorphic to $O(n)\otimes f^{*}(\mathcal{M}')$ where $\mathcal{M}'$ is an ample line bundle on $\mjtgn$, $f$ is the morphism $\mathbf{P}^{1}\times^{\mathbb{G}_{m}}\mjtg\rightarrow\mjtgn$ and $n>0$. Now let us notice that by Theorem~\ref{main theorem of app B}, a line bundle on $\hig(\cam)$ is the same a line bundle on $\mathbf{P}^{1}\times^{\mathbb{G}_{m}}\mjtg$ together with a gluing map when restricted to $\mjtgn\amalg\mjtgn$.
We have two sections:
\begin{gather}
0\times^{\mathbb{G}_{m}}\mjtg\simeq\mjtgn \xrightarrow{i_{0}}\mathbf{P}^{1}\times^{\mathbb{G}_{m}}\mjtg \notag\\
\infty\times^{\mathbb{G}_{m}}\mjtg\simeq\mjtgn\xrightarrow{i_{\infty}}\mathbf{P}^{1}\times^{\mathbb{G}_{m}}\mjtg . \notag
\end{gather}
Using Corollary~\ref{description of projective bundle} one may assume that 
\begin{gather}
i_{0}^{*}(O(n)\otimes f^{*}(\mathcal{M}'))\simeq \mathcal{M}' \notag\\
i_{1}^{*}(O(n)\otimes f^{*}(\mathcal{M}'))\simeq \mathcal{Q}^{\otimes n}\otimes\mathcal{M}' . \notag
\end{gather}
Since we have shown that $\xi$ belongs to the lie algebra of $\mathbb{G}_{m}$, one may view it as an element in the unit of the ring of dual numbers $k[\epsilon]/\epsilon^{2}$. From the expression of the restriction of $O(n)\otimes f^{*}(\mathcal{M}')$ to $i_{0}(\mjtgn)$ and $i_{1}(\mjtgn)$, we conclude that the gluing maps for $\xi^{*}(O(n)\otimes f^{*}(\mathcal{M}'))$ and $O(n)\otimes f^{*}(\mathcal{M}')$ differs by $\xi^{n}$ where we view $\xi$ as a unit in $k[\epsilon]/\epsilon^{2}$. This shows that $\xi^{*}(\mathcal{M})\otimes\mathcal{M}^{-1}$ gives a nontrivial element in $H^{1}(O_{\hig(\cam)})$ as long as $\xi$ is nonzero. 

The analysis above shows that if $\cam$ is a cameral cover corresponds to a point in $U\subseteq H_{int}$, then the morphism $\textrm{Lie}(\mjtgcam)\rightarrow H^{1}(O_{\hig(\cam)})$ is an isomorphism. Since $U$ is smooth, Grauert's theorem on cohomology and base change shows that $R^{1}\pi_{U *}O_{\hig\mid_{U}}$ is locally free and isomorphic to $\textrm{Lie}(\mjtg)\mid_{U}$. Let $j$ be the open embedding $U\xrightarrow{j}H_{int}$. Then one gets a morphism $R^{1}\pi_{*}O_{\hig}\rightarrow j_{*}(\textrm{Lie}(\mjtg)\mid_{U})$. Since the complement of $U$ has codimension greater than or equals to two, we conclude that $j_{*}(\textrm{Lie}(\mjtg)\mid_{U})\simeq \textrm{Lie}(\mjtg)$. Now arguing as in Theorem $6.5$ of \cite{Hitchin} we get the result.

\end{proof}

\section{An open locus of the Hitchin base}
In this section we determine an open subset $U$ of the Hitchin base $H$ whose complement has codimension greater than or equals to two, see Lemma~\ref{an open set of H}. Lemma~\ref{regularity criterion} is due to Professor Dima Arinkin and the author would like to thank him for giving the permission to publish it. 
\begin{lemma}\label{regularity criterion}
Let $G$ be a reductive group and $r$ be the semisimple rank of $G$. Let $x=x_{0}+tx_{1}+\cdots$ be an element in $\mathfrak{g}(k[[t]])$ with $x_{0}$ nilpotent. Let $\disc$ be the discriminant function on the Chevalley base $\mathfrak{c}$ and denote its pullback to $\mathfrak{g}$ via the Chevalley map $\mathfrak{g}\rightarrow\mathfrak{c}$ also by $\disc$. Then we have:
\begin{enumerate1}
\item $\disc(x)\in t^{r}k[[t]]$. 
\item If the order of vanishing of $\disc(x)$ is exactly equal to $r$, then $x_{0}$ is regular. 
\item Let $\spec(\widehat{O}_{x})\xrightarrow{x} \mathfrak{g}$ be a morphism. Let $\chi(x_{0})$ be the image of the closed point of $\spec(\widehat{O}_{x})$ under the morphism $\mathfrak{g}\xrightarrow{\chi}\mathfrak{c}$. If the image of the tangent space of $\spec(\widehat{O}_{x})$ in $T_{\zeta}\mathfrak{c}$ under the morphism $\spec(\widehat{O}_{x})\xrightarrow{x} \mathfrak{g}\xrightarrow{\chi}\mathfrak{c}$ does not lies inside the tangent cone  of $\disc$ in $T_{\chi(x_{0})}\mathfrak{c}$, then $x$ factors through $\mathfrak{g}_{reg}$. 
\end{enumerate1}
\end{lemma}
\begin{proof}
We may assume $G$ is semisimple. Let us recall that if $a\in\mathfrak{g}$, then we have that the $\disc(a)$ is equal to the coefficient of $T^{r}$ of the polynomial $\de(T-ad(a))$. We apply this to $a=x$. If $n=\dimension\mathfrak{g}$, then it is well-known that the coefficient of $T^{r}$ is equal to the sum of $(n-r)\times(n-r)$ principle minors of the matrix corresponds to the linear transform $ad(x)$ on $\mathfrak{g}$. Now observe that if $A$ is a $n\times n$ nilpotent matrix such that the dimension of the kernel of $A$ is greater than or equals to $r$. Then for any $n\times n$ matrix $B=B_{0}+tB_{1}+\cdots$, all $(n-r)\times (n-r)$ principle minors of $A+tB$ are power series whose initial terms have degree greater than or equals to the dimension of the kernel of $A$. This is easy to check by reducing to the case when $A$ has Jordan normal form. We apply this to the case when $A=ad(x_{0})$ and $B=ad(x_{1}+tx_{2}+\cdots)$. This proves part $(1)$. Moreover, using the same argument, it is not hard to see that if the assumption of part $(2)$ holds, then the dimension of the kernel of $x_{0}$ must equal to $r$, hence $x_{0}$ is regular. 

For part $(3)$, write $x$ as $x=x_{0}+tx_{1}+\cdots$. Let us first consider the case when $x_{0}$ is nilpotent. Then the claim a reformulation of part $(2)$, using the fact that the multiplicity of the divisor $\disc$ at $0\in\mathfrak{c}$ is equal to $r$, see Section $3.18$ of \cite{Coxeter groups}. The general case can be reduced to the nilpotent case as follows. Let us consider the Jordan decomposition $x_{0}=t+n$ where $t$ is semisimple and $n$ is nilpotent. Let $L=Z_{G}(t)$. Then we see that $\chi(x_{0})$ is in the image of $\mathfrak{c}^{\circ}_{L}$ (Here we will denote the Chevalley base of $G$ and $L$ by $\mathfrak{c}_{G}$ and $\mathfrak{c}_{L}$ in order to distinguish between them. Similarly for the discriminant functions $\disc_{G}$ and $\disc_{L}$), see our notations in Lemma~\ref{reduce to levi}. Since $\widehat{O}_{x}$ is complete, use Lemma~\ref{reduce to levi} we can assume we have a lift $x$ to an element in $y\in\mathfrak{l}(\widehat{O}_{x})$. We claim that $y$ satisfies the same assumption as $x$. Indeed, it is not hard to check that the quotient $\disc_{G}/\disc_{L}$ is an invertible function on $\mathfrak{c}^{\circ}_{L}$. Since $\mathfrak{c}^{\circ}_{L}\rightarrow\mathfrak{c}_{G}$ is etale, the tangent cone of $\disc_{G}$ at $\chi_{G}(x_{0})\in\mathfrak{c}_{G}$ can be identified with the tangent cone of $\disc_{L}$ at $\chi_{L}(y_{0})\in\mathfrak{c}_{L}$. Now write $y=y_{[L,L]}+y_{Z(L)}$ where $y_{[L,L]}$ lies in the lie algebra of $[L,L]$ and $y_{Z(L)}$ lies in the lie algebra of the center of $L$. Using the discussions above one can reduce everything to the semisimple group $[L,L]$. Also notice that the initial term of $y_{[L,L]}$ is nilpotent, so we are done.

\end{proof}

Let us consider the following subset $U$ of $H_{int}$, which consists of points such that the corresponding discriminant divisor on $X$ is multiplicity free except at one point $x\in X$, where it has order less than or equals to $2$. 
\begin{lemma}\label{an open set of H}
\begin{enumerate1}
\item $U$ is an open subset of $H$ whose complement has codimension greater than or equals to two.
\item If $\cam$ is a cameral cover corresponds to a point in $U$, then either we have $\hig(\cam)=\hig^{reg}(\cam)$ or $\cam$ satisfies the assumptions on cameral covers we give at the beginning of Appendix $B$. 
\end{enumerate1}
\end{lemma}
\begin{proof}

For part $(1)$, let us recall that the morphism $X\times H\rightarrow \mathfrak{c}\times^{\mathbb{G}_{m}}{L}$ is smooth, this follows from the expression $\mathfrak{c}\times^{\mathbb{G}_{m}}{L}\simeq L^{e_{1}}\oplus\cdots\oplus L^{e_{r}}$ where $e_{i}$ are the degrees of the $W$ invariant fundamental polynomials on $\mathfrak{t}$, see Subsection $4.13$ of \cite{Fundamental lemma}. By definition, the complement of $U$ corresponds to points in $H$ such that either there exists a point $x\in X$ such that the discriminant divisor has multiplicity greater than or equals to three at $x$, or there exists two points $x,y\in X$ such that the discriminant divisor has multiplicities greater than or equals to two at $x$ and $y$. So part $(1)$ follows from the following claims:
\begin{enumerate1}
\item The set of pairs consisting of $(x,h)\in X\times H$ such that the discriminant divisor has multiplicity greater than or equals to three at $x$ has codimension greater than or equals to $3$ in $X\times H$.
\item The set of triples $(x,y,h)\in X^{2}\times H$ such that the discriminant divisor has multiplicities greater than or equals to two at $x$ and $y$ has codimension greater than or equals to $4$ in $X^{2}\times H$. 
\end{enumerate1}
Let us look at the first claim. To simplify the notation, we will denote the bundle $\mathfrak{c}\times^{\mathbb{G}_{m}}{L}$ by $\mathfrak{c}^{L}$ in the rest of the proof. We will prove that if we fix $x$, then the set of $h$ such that the discriminant divisor has multiplicity greater than or equals to three at $x$ has codimension greater than or equals to $3$ in $H$. Let us look at the smooth morphism $H\rightarrow \mathfrak{c}^{L}_{x}$ obtained by restricting $X\times H\rightarrow \mathfrak{c}^{L}$ to $x$. If $(x,h)\in X\times H$ is such that $x$ lies in the locus of $\mathfrak{c}^{L}_{x}$ where $\disc_{G}$ has multiplicity greater than or equals to $3$, then by Lemma~\ref{codimension multiplicity} conclude that all such $h$ form a subset of codimension greater than or equals to $3$. Now assume $x$ lies in the locus of $\mathfrak{c}^{L}_{x}$ where the multiplicity of $\disc_{G}$ is equal to two. Since the discriminant divisor has multiplicity greater than or equals to three at $x$, it implies that if we look at the map $\spec(\widehat{O}_{x})\rightarrow\mathfrak{c}$ induced by $h$ (Plus a trivialization of the line bundle $L$ on $\spec(\widehat{O}_{x})$), then the image of the tangent space of $\spec(\widehat{O}_{x})$ lies inside the tangent cone of the discriminant function. Notice that the morphism $H\rightarrow \mathfrak{c}^{L}/m^{2}_{x}\mathfrak{c}^{L}$ is a surjective morphism of vector spaces (See Lemma $4.7.2$ of \cite{Fundamental lemma}). Combine this with Lemma~\ref{codimension multiplicity}, we see that the condition that $x$ lies inside the locally closed subset of $\mathfrak{c}^{L}_{x}$ where $\disc_{G}$ has multiplicity two plus the condition that the image of the tangent space of $\spec(\widehat{O}_{x})$ lies inside the tangent cone of  $\disc_{G}$ gives a codimension three subset in $H$. Finally let us assume that $x$ lies in the locus of $\mathfrak{c}^{L}_{x}$ where $\disc_{G}$ has multiplicity one. In this case $x$ actually lies in the open set $\mathfrak{c}^{L,sm}_{x}$ where the morphism $\mathfrak{c}^{L,sm}_{x}\xrightarrow{\disc} \mathbb{A}^{1}$ is smooth. So if we identify $\mathfrak{c}^{L}_{x}/\mathfrak{m}^{3}_{x}\mathfrak{c}^{L}_{x}$ with $\mathfrak{c}(O_{x}/\mathfrak{m}^{3}_{x})$ and consider the smooth surjective morphism 
$H\rightarrow \mathfrak{c}^{L}_{x}/\mathfrak{m}^{3}_{x}\mathfrak{c}^{L}_{x}\simeq \mathfrak{c}(O_{x}/\mathfrak{m}^{3}_{x})$,
then $h$ lies inside the open set $\mathfrak{c}^{sm}(O_{x}/\mathfrak{m}^{3}_{x})$. Since
$$\mathfrak{c}^{sm}(O_{x}/\mathfrak{m}^{3}_{x})\rightarrow \mathbb{A}^{1}(O/\mathfrak{m}^{3}_{x})\simeq O/\mathfrak{m}^{3}_{x}$$
is a smooth morphism, we conclude that the condition that $x$ lies in the locus of $\mathfrak{c}^{L}_{x}$ where $\disc_{G}$ has multiplicity one and that the discriminant divisor has order greater than or equals to three at $x$ cuts out a locally closed subset of $H$ with codimension greater than or equals to three. Combining these together, we have established the first claim. The proof for the second claim is similar. Namely, assume first that the image of both $x$ and $y$ lies in the locus of $\mathfrak{c}^{L}_{x}$ and $\mathfrak{c}^{L}_{y}$ where $\disc_{G}$ has multiplicity equals to one. In this case, the induced morphisms $\spec(\widehat{O}_{x})\rightarrow\mathfrak{c}$ and $\spec(\widehat{O}_{y})\rightarrow\mathfrak{c}$ has the property that the image of the tangent spaces of $\spec(\widehat{O}_{x})$ and $\spec(\widehat{O}_{y})$ lies inside the tangent cone of $\disc_{G}$. Since the morphism $H\rightarrow\mathfrak{c}^{L}_{x}/\mathfrak{m}^{2}_{x}\mathfrak{c}^{L}_{x}\oplus \mathfrak{c}^{L}_{y}/\mathfrak{m}^{2}_{y}\mathfrak{c}^{L}_{y}$ is smooth and surjective, this cuts out a codimension $4$ locally closed subset of $H$. Now assume that the image of $x$ lies in the locus of $\mathfrak{c}^{L}_{x}$ while the image of $y$ lies in the closed subset where $\disc_{G}$ has multiplicity greater than or equals to two. Using Lemma~\ref{codimension multiplicity} and the argument above, we see that this also cuts out a locally closed subset of codimension greater than or equals to $4$. The last case is when both $x$ and $y$ lies in the locus of $\mathfrak{c}^{L}_{x}$ and $\mathfrak{c}^{L}_{y}$ where the multiplicity of $\disc_{G}$ is greater than or equals to two. Arguing as above we get a codimension $4$ closed subset in $H$. This finishes the proof for claim $(2)$. 

Now let us look at part $(2)$. Under our assumptions, the discriminant divisor is either multiplicity free or there exists a unique point $x\in X$ where it is not multiplicity free and the multiplicity at $x$ is equal to two. In the case when it is multiplicity free everywhere, it is well-known that the cameral cover is smooth and we have $\hig(\cam)=\hig^{reg}(\cam)$, see Proposition $4.7.7$ of \cite{Fundamental lemma}. When it has multiplicity two at $x\in X$, $\cam$ is smooth away from $x$. Over the formal disk $\spec(\widehat{O}_{x})$, $\cam$ is determined by a morphism $\spec(\widehat{O}_{x})\rightarrow\mathfrak{c}_{G}$. Choose an element $t\in\mathfrak{t}$ such that the image of $t$ in $\mathfrak{c}_{G}$ is equal to the image of the closed point $x$ in $\mathfrak{c}_{G}$. Let $L=Z_{G}(t)$. Then arguing as in the proof of part $(3)$ of Lemma~\ref{regularity criterion}, we see that the image of $x$ in $\mathfrak{c}_{G}$ lies in the image of $\mathfrak{c}^{\circ}_{L}\rightarrow\mathfrak{c}_{G}$. Hence $\spec(\widehat{O}_{x})\rightarrow\mathfrak{c}_{G}$ can be lifted to $\spec(\widehat{O}_{x})\rightarrow\mathfrak{c}_{L}$ and that the multiplicity of the pullback of $\disc_{L}$ to $\spec(\widehat{O}_{x})$ has multiplicity two. Moreover, if we identify $\mathfrak{c}_{L}\simeq\mathfrak{c}_{[L,L]}\times\mathfrak{z}_{L}$, then the induced morphism $\spec(\widehat{O}_{x})\rightarrow\mathfrak{c}_{[L,L]}$ has the property that the image of $x$ is the origin. By part $(1)$ of Lemma~\ref{regularity criterion} we conclude that $L$ has semisimple rank $1$ or $2$. If it has semisimple rank two, then by part $(2)$ of Lemma~\ref{regularity criterion} we conclude that whenever we have a lift of the morphism $\spec(\widehat{O}_{x})\rightarrow\mathfrak{c}_{L}$ to $\spec(\widehat{O}_{x})\xrightarrow{\varphi}\mathfrak{l}$, then $\varphi$ is a regular element. Since $\cam$ is smooth away from $x$, combining this with Lemma~\ref{reduce to levi} we see that every Higgs bundle with cameral cover $\cam$ is regular at $x$, hence we have $\hig(\cam)=\hig^{reg}(\cam)$. If $L$ has semisimple rank one, then $\cam$ satisfies the conditions at the beginning of Appendix B. This finishes the proof.

\end{proof}

\begin{lemma}\label{codimension multiplicity}
Let $G$ be a reductive algebraic group and $\mathfrak{g}$ be its lie algebra. Let $\mathfrak{c}_{G}$ be the Chevalley base. Then the closed subset of $\mathfrak{c}_{G}$ where the discriminant function $\disc_{G}$ has multiplicity greater than or equals to $m$ has codimension greater than or equals to $m$. Here we recall that for a smooth variety $Y$, the multiplicity of an effective divisor $D$ at a point $y\in Y$ is defined to be the unique integer $k$ such that if $\mathfrak{m}_{y}$ is the maximal ideal of the local ring at $y$, then the local equation of $D$ belongs $\mathfrak{m}^{k}_{y}\setminus\mathfrak{m}^{k+1}_{y}$
\end{lemma}
\begin{proof}
We will first prove that if $G$ has semisimple rank less than $m$, then the multiplicity of the discriminant function at any point in $\mathfrak{c}_{G}$ is always less than $m$. It is enough to prove this when $G$ is semisimple. Let us consider:
$$\mathfrak{t}\rightarrow\mathfrak{c}_{G}\xrightarrow{\disc_{G}}\mathbb{A}^{1} .$$
The natural $\mathbb{G}_{m}$ action on $\mathfrak{t}$ induces $\mathbb{G}_{m}$ actions on $\mathfrak{c}_{G}$ and $\mathbb{A}^{1}$ such that $\disc_{G}$ is $\mathbb{G}_{m}$ equivariant. So multiplicity of $\disc_{G}$ stays constant on each $\mathbb{G}_{m}$ orbit. Consider the closed subset where $\disc$ has multiplicity greater than or equals to $m$. Since it is invariant under $\mathbb{G}_{m}$ action, if it is nonempty, it must contains $0\in\mathfrak{c}_{G}$. On the other hand, it is well-known that the multiplicity of $\disc_{G}$ at $0\in\mathfrak{c}_{G}$ is equal to the rank of $G$ (See Section $3.18$ of \cite{Coxeter groups}), which is less than $m$ by our assumption. 

Now we shall consider the general case. We may assume the semisimple rank of $G$ is greater than or equals to $m$ by our discussions above. Suppose the multiplicity of $\disc$ is greater than or equals to $m$ at $x\in\mathfrak{c}_{G}$ and let $t$ be an element in $\mathfrak{t}$ that maps to $x$. Consider the levi subgroup $L=Z_{G}(t)$ and the morphisms:
$$\mathfrak{t}\rightarrow\mathfrak{c}_{L}\rightarrow\mathfrak{c}_{G} .$$
Denote the image of $t$ in $\mathfrak{c}_{L}$ by $x'$. Since $L=Z_{G}(t)$, we see that $x'\in\mathfrak{c}^{\circ}_{L}$ (See the notation in Lemma~\ref{reduce to levi}). Moreover, it is not hard to see that $\disc_{G}/\disc_{L}$ is an invertible function on $\mathfrak{c}^{\circ}_{L}$. Hence the multiplicity of $\disc_{G}$ at $x$ is equal to the multiplicity of $\disc_{L}$ at $x'$. Since we take $x$ to be a point where $\disc_{G}$ has multiplicity greater than or equals to $m$, we see that the multiplicity of $\disc_{L}$ at $x'$ is also greater than or equals to $m$. Our discussions above implies that $L=Z_{G}(t)$ must have semisimple rank greater than or equals to $m$. All such elements $t$ form a closed subset with codimension greater than or equals to $m$. This finishes the proof. 

\end{proof}

\end{document}